\documentclass[a4paper]{amsart}

\RequirePackage{amsmath} 
\RequirePackage{amssymb}
\usepackage{amscd,latexsym,amsthm,amsfonts,amssymb,amsmath,amsxtra}
\usepackage[colorlinks=true,urlcolor=blue,citecolor=blue]{hyperref}
\usepackage{color}
\usepackage[all]{xy}
\usepackage[OT2,T1]{fontenc}
\usepackage{bm}
\usepackage{mathtools}
\usepackage{ mathrsfs }
\usepackage{xcolor}
\usepackage{comment}
\usepackage{marginnote}
\usepackage{enumitem}
\usepackage{thmtools, thm-restate}

\DeclareSymbolFont{cyrletters}{OT2}{wncyr}{m}{n}
\DeclareMathSymbol{\Sha}{\mathalpha}{cyrletters}{"58}

\let\Re\undefined

\DeclareMathOperator{\Re}{Re}

\DeclareMathOperator{\Tr}{Tr}

\DeclareMathOperator{\supp}{supp}

\DeclareMathOperator{\Spec}{Spec}

	\newcommand{\Res}{\operatorname{Res}}
	
	\newcommand{\Eis}{\operatorname{Eis}}
	\newcommand{\Geo}{\operatorname{Geo}}
	
	\newcommand{\K}{\operatorname{K}}
	\newcommand{\sgn}{\operatorname{sgn}}
	
	\newcommand{\du}{\operatorname{Dual}}
	
	\newcommand{\Ad}{\operatorname{Ad}}

	\newcommand{\Reg}{\operatorname{Reg}}

	\newcommand{\fin}{\operatorname{fin}}
	\newcommand{\diag}{\operatorname{diag}}

	\newcommand{\Vol}{\operatorname{Vol}}

	\newcommand{\sm}{\operatorname{Small}}
	\newcommand{\bi}{\operatorname{Big}}

	\newcommand{\dist}{\operatorname{dist}}

	\newcommand{\Ind}{\operatorname{Ind}}
	\newcommand{\ER}{\operatorname{ER}}

	\newcommand{\RNum}[1]{\uppercase\expandafter{\romannumeral #1\relax}}

\begin{document}
\theoremstyle{plain}
	\newtheorem{thm}{Theorem}[section]
	
	\newtheorem{cor}[thm]{Corollary}
	\newtheorem{thmy}{Theorem}
	\renewcommand{\thethmy}{\Alph{thmy}}
	\newenvironment{thmx}{\stepcounter{thm}\begin{thmy}}{\end{thmy}}
	\newtheorem{cory}{Corollary}
	\renewcommand{\thecory}{\Alph{cory}}
	\newenvironment{corx}{\stepcounter{thm}\begin{cory}}{\end{cory}}
	\newtheorem{hy}[thm]{Hypothesis}
	\newtheorem*{thma}{Theorem A}
	\newtheorem*{corb}{Corollary B}
	\newtheorem*{thmc}{Theorem C}
	\newtheorem{lemma}[thm]{Lemma}  
	\newtheorem{prop}[thm]{Proposition}
	\newtheorem{conj}[thm]{Conjecture}
	\newtheorem{fact}[thm]{Fact}
	\newtheorem{claim}[thm]{Claim}
	
	\theoremstyle{definition}
	\newtheorem{defn}[thm]{Definition}
	\newtheorem{example}[thm]{Example}
	\theoremstyle{remark}
	
	\newtheorem{remark}[thm]{Remark}	
	\numberwithin{equation}{section}
	
\title[]{Relative Trace Formula and Twisted $L$-functions: the Burgess Bound}%
\author{Liyang Yang}
	
\begin{abstract}
Let $F$ be a number field, $\pi$ either a unitary cuspidal automorphic representation of $\mathrm{GL}(2)/F$ or a unitary Eisenstein series, and $\chi$ a unitary Hecke character of analytic conductor $C(\chi).$ We develop a regularized relative trace formula to prove a refined hybrid subconvex bound for $L(1/2,\pi\times\chi).$  
In particular, we obtain the Burgess subconvex bound 
\begin{align*}
L(1/2,\pi\times\chi)\ll_{\pi,F,\varepsilon}C(\chi)^{\frac{1}{2}-\frac{1}{8}+\varepsilon},
\end{align*}  
where the implied constant depends on $\pi,$ $F$ and $\varepsilon.$ 
\end{abstract}
	
\date{\today}%
\maketitle
\tableofcontents
	
\section{Introduction}
\subsection{The Burgess Bound for $\mathrm{GL}(2)\times\mathrm{GL}(1)$ over Number Fields}\label{sec1.1}

Let $F$ be a number field. Let $\pi$ be a pure isobaric representation of $\mathrm{GL}(2)/F,$ i.e., $\pi$ is either a unitary cuspidal automorphic representation, or an Eisenstein series induced from unitary Hecke characters. In this paper we consider the \textit{subconvexity problem for character twists over a number field} (\textbf{ScP($\pi\otimes\chi$})): finding a constant $\delta\in (0,1/2],$ such that 
\begin{equation}\label{scp}
L(1/2,\pi\times\chi)\ll_{\pi,F,\varepsilon} C(\chi)^{\frac{1}{2}-\delta+\varepsilon}
\end{equation}
holds for all unitary Hecke characters $\chi$ over $F,$ where the implied constant relies on $\pi,$ $F$ and $\varepsilon.$ Here $C(\chi)$ is the analytic conductor of $\chi$ (cf. \textsection\ref{2.1.1}). Note that \eqref{scp} involves the bound of $L(s,\pi\times\chi)$ on the critical line $\Re(s)=1/2$ as $L(1/2+it,\pi\times\chi)=L(1/2,\pi\times\chi|\cdot|^{it}).$ Moreover, if $\pi=\textbf{1}\boxplus\textbf{1}$ is the Eisenstein series induced from the trivial character, then $L(1/2,\pi\times\chi)=L(1/2,\chi)^2,$ and thus \eqref{scp} becomes the classical subconvexity problem for Hecke $L$-functions.

The above \textbf{ScP($\pi\otimes\chi$}) is pioneering study of the automorphic subconvexity problem in a systematic manner  and has been extensively investigated over the years, e.g., cf. \cite{Bur63}, \cite{HB78}, \cite{HB80}, \cite{CI00}, \cite{DFI02}, \cite{Mil16}, \cite{You17},   \cite{Wu19}, \cite{PY22} and \cite{PY20} in the case that $\pi$ is an Eisenstein series, and cf. \cite{DFI93}, \cite{By96}, \cite{CPSS01}, \cite{BHM07}, \cite{BH08}, \cite{Ven10}, \cite{MV10}, \cite{BH10}, \cite{BH14}, \cite{Wu14}, \cite{Mag14}, \cite{Mun18a}, \cite{Nel19}, \cite{AHLS20}, and \cite{FS22} in the case that $\pi$ is cuspidal. The subconvex bound of type \eqref{scp} is a versatile tool with numerous applications in arithmetic geometry, as highlighted in Remark \ref{rem1.4} below, which provides several illustrative examples. One notable instance involves the work of \cite{CPSS01}, where the authors were the first to examine the case of $F$ being a number field other than $\mathbb{Q}$ in connection with Hilbert's eleventh problem. This groundbreaking investigation paved the way for further developments in the analytic theory of $L$-functions, leading to the emergence of the \textbf{ScP($\pi\otimes\chi$}) problem over number fields as a significant area of study.

The Lindel\"{o}f hypothesis predicts that  \eqref{scp} holds with $\delta=1/2$. When $F=\mathbb{Q}$, the tightest bound for \eqref{scp} is $\delta=1/8$ (e.g., cf. \cite{BH08}). In the special case that $\pi\otimes\chi$ has a trivial central character, Conrey and Iwaniec \cite{CI00} showed that $\delta = 1/6$ is admissible. In the more general situation where $F\neq\mathbb{Q}$, the best known result is $\delta=(1-2\vartheta)/8$ (cf. \cite{Wu14}, \cite{Wu19}), where $0\leq \vartheta<1/2$ is any exponent towards the Ramanujan-Petersson conjecture (\textbf{RPC}) for unitary cuspidal representations of $\mathrm{GL}(2)/F$. Currently the strongest known value is  $\vartheta=7/64$ (cf. \cite{KS03}, \cite{BB11}). Hence, the most favorable unconditional bound in the general case is $\delta=1/8-7/256$. Nevertheless, under the \textbf{RPC}, i.e., $\vartheta=0$, one gets the bound $\delta=1/8$. This is  conventionally referred to as the \textit{Burgess bound}, which represents natural barriers that may require new ideas to overcome.

The first achievement of this paper  confirms \eqref{scp} with $\delta=1/8$ for any pure isobaric representation $\pi$ and unitary Hecke character $\chi$ over $F,$ obtaining the Burgess bound without assuming the Ramanujan-Petersson conjecture.  
\begin{thmx}\label{A}
Let notation be as before. Then 
\begin{equation}\label{a.}
L(1/2,\pi\times\chi)\ll_{\pi,F,\varepsilon}C(\chi)^{\frac{1}{2}-\frac{1}{8}+\varepsilon},
\end{equation}
where the implied constant depends on $\pi,$ $F$ and $\varepsilon.$ 
\end{thmx}

\begin{cor}
Let notation be as before. Then 
\begin{equation}\label{equ1.1}
L(1/2,\chi)\ll_{F,\varepsilon}C(\chi)^{\frac{1}{4}-\frac{1}{16}+\varepsilon}.
\end{equation}
 Let $\pi$ be a  unitary cuspidal representation of $\mathrm{GL}(2)/F.$ Then 
\begin{equation}\label{equ1.2}
L(1/2,\pi\times\chi)\ll_{\pi,F,\varepsilon}C(\chi)^{\frac{1}{2}-\frac{1}{8}+\varepsilon}.
\end{equation}
\end{cor}

\begin{remark}
The estimate \eqref{equ1.1} improves \cite{Wu19}, and \eqref{equ1.2} improves \cite{By96}, \cite{BHM07}, \cite{BH08}, \cite{BH10}, \cite{BH14}, \cite{Wu14}, and \cite{Mag14}. 
\end{remark}

\begin{remark}\label{rem1.4}
Via Waldspurger type formulas, Theorem \ref{A} assumes a crucial role in 
\begin{itemize}
	\item bounding Fourier coefficients of general half-integral weight Hilbert modular forms;
	\item computing the asymptotic of representation numbers of quadratic forms with error terms, cf. \cite{CPSS01}, \cite{Cog03}, and \cite{BH10};
	\item equidistribution of CM-points and certain subvarieties on the Hilbert modular surface, cf. e.g., \cite{Duk88}, \cite[Theorem 3.2]{Zha05}, \cite[Theorem 1.2]{Coh05}, and \cite[\textsection 1.1]{Ven10}. 
\end{itemize}
\end{remark}
\begin{remark}
Upon closer inspection, it is observable from the proof that the implied constants in \eqref{a.}, \eqref{equ1.1} and \eqref{equ1.2} have polynomial dependence in the analytic conductor  of $\pi,$ and in the absolute discriminant of $F.$
\end{remark}

\subsection{Explicit Hybrid Bounds with Applications}\label{sec1.2}

Theorem \ref{A} follows readily from the following refined hybrid subconvex bound for twisted $L$-functions. 
\begin{thmx}\label{B}
Let $F$ be a number field with ring of adeles $\mathbb{A}_F$. Let $\pi$ be a pure isobaric representation of $\mathrm{GL}(2)/F$ with central character $\omega_{\pi}.$ Let $\chi$ be a Hecke character of $\mathbb{A}_F^{\times}/F^{\times}$. Suppose that $\pi_{v}\otimes\chi_{v}$ has uniform parameter growth of size $(T_v;c_v,C_v),$ for all $v\mid\infty,$ cf. \textsection\ref{2.1.2}. Then 
\begin{align*}
L&(1/2,\pi\times\chi)\ll C_{\fin}(\pi)^{\frac{1}{4}+\varepsilon}C_{\fin}(\chi)^{\frac{3}{8}+\varepsilon}C_{\infty}(\pi\otimes\chi)^{\frac{3}{16}+\varepsilon}\gcd(C_{\fin}(\omega_{\pi}),C_{\fin}(\chi))^{\frac{1}{8}}\\
&+C_{\fin}(\pi)^{\frac{1}{2}+\varepsilon}C_{\fin}(\chi)^{\frac{1}{4}+\varepsilon}C_{\infty}(\pi\otimes\chi)^{\frac{1}{8}+\varepsilon}C_{\infty}(\pi)^{\frac{1}{4}+\varepsilon}\gcd(C_{\fin}(\omega_{\pi}),C_{\fin}(\chi))^{\frac{1}{4}},
\end{align*}
where the implied constant depends on $\varepsilon,$ $F,$ $c_v,$ and $C_v,$ $v\mid\infty$. Here for a representation $\sigma=\otimes_v\sigma_v$ of $\mathrm{GL}(1)$ or $\mathrm{GL}(2)/F,$ $C_v(\sigma)$ is the associated local conductor (cf. \textsection\ref{2.1.1}), and $C_{\infty}(\sigma)=\prod_{v\mid\infty}C_v(\sigma),$ and   $C_{\fin}(\sigma)=\prod_{v<\infty}C_v(\sigma).$  
\end{thmx}

\begin{remark}
Theorem \ref{B} refines and generalizes the bounds in several prior works, including \cite{BHM07}, \cite{BH08}, and \cite{BH14}. Notably, when $F=\mathbb{Q}$, the exponents of the finite parts $C_{\fin}(\pi),$ $C_{\fin}(\chi),$ and $\gcd(C_{\fin}(\omega_{\pi})),C_{\fin}(\chi))$ are identical to those in \cite{BH08} and \cite{BH14}. However, the exponents of $C_{\infty}(\pi\otimes\chi)$ and $C_{\infty}(\pi)$ have been improved, partially due to the use of Nelson's test function at the archimedean places.
\end{remark}

\begin{remark}
As in \cite[\textsection 3.4.2]{MV06}, \cite{BHM07}, and \cite[\textsection 5]{MV10}, Theorem \ref{B} can also be used as an input for the subconvexity problem on $\mathrm{GL}(2)\times\mathrm{GL}(2)$.	
\end{remark}

\subsection{Discussion of the Proofs} 
Different methods have been explored in the literature to approach the Burgess subconvex bound (for general $\chi$):
\begin{itemize}
	\item Bykovski\u{\i}'s formula in \cite{By96}, which is evolved from the Kuznetsov formula, gives rise to the Burgess bound in the $C_{\fin}(\chi)$-aspect when $F=\mathbb{Q}$ (cf. \cite{By96}, \cite{BH08}, and \cite{BH14}). Nevertheless, it remains unclear how to generalize this to $C(\chi)$-aspect even over $\mathbb{Q}.$ 
	\item Recently, various delta methods have been developed to prove the Burgess bound when $F=\mathbb{Q}$ and $\chi$ is a Dirichlet character, e.g., cf. \cite{Mun18a}, \cite{AHLS20}, and \cite{FS22}. 
	\item When $F\neq \mathbb{Q}$ is a number field, the Burgess-like bound $\delta=(1-2\vartheta)/8$ can be achieved by three  different methods:
	\begin{itemize}
    \item the large sieve, and bounding shift convolution sums by Jutila's circle method, cf. \cite{BHM07};
	\item the spectral decomposition of shift convolution sums in Hecke eigenvalues, cf. \cite{BH10} and \cite{Mag14};
	\item the method of Michel and and Venkatesh \cite{MV10}, cf. \cite{Wu14} and \cite{Wu19}.
	\end{itemize}
\end{itemize}

Varied from the techniques listed above, our strategy to prove  Theorem \ref{B} depends essentially on a new relative trace formula (\textbf{RTF}) as a  vital structural input (cf. Theorem \ref{thmD'} in \textsection\ref{sec2}). When $F=\mathbb{Q}$, the amplified relative trace formula Theorem \ref{thmD'} simplifies under certain classical test functions to the Bykovski\u{\i} formula in \cite{By96}, providing a new proof of the latter.


In principle the \textbf{RTF} is based on the
equality of geometric and spectral expansions of the integral
\begin{equation}\label{1.4}
\int_{A(F)\backslash A(\mathbb{A}_F)}\int_{A(F)\backslash A(\mathbb{A}_F)}\K(x,y)\chi(x)\overline{\chi}(y)d^{\times}xd^{\times}y,
\end{equation}
where $\K$ is an automorphic kernel function and $A=\diag(\mathrm{GL}(1),1).$ Despite the fact that \eqref{1.4} does not converge in general, we regularize it to obtain the \textbf{RTF} of the form (cf. \eqref{2.24} in \textsection\ref{sec2})
\begin{align*}
J_{\Spec}^{\Reg,\heartsuit}(f,\textbf{s},\chi)=J_{\Geo}^{\Reg,\heartsuit}(f,\textbf{s},\chi),\ \ \textbf{s}\in\mathbb{C}^2,
\end{align*}
where $f$ is a suitable test function constructed in \textsection\ref{3.2}. To illustrate the basic idea, let us consider the level aspect. Suppose, for the sake of argument, that $\pi$ and $\chi$ have disjoint ramifications. Let $\mathbf{s}=\mathbf{s}_0$ be a suitable point near the origin $(0,0).$ 

The test function $f$ allows us to select a family $\mathcal{A}$ of automorphic representations of level $C_{\fin}(\pi)C_{\fin}(\chi)$, which includes $\pi$. Hence, on the spectral side, we have 
\begin{equation}\label{1.6}
J_{\Spec}^{\Reg,\heartsuit}(f,\textbf{s}_0,\chi)\gg  (C_{\fin}(\pi)C_{\fin}(\chi))^{-\varepsilon} |L(1/2,\pi\times\chi)|^2.
\end{equation}

Regarding the geometric side, the contribution from irregular orbital integrals can be estimated as $\ll |\mathcal{A}|^{1+o(1)}\asymp C_{\fin}(\pi)^{1+\varepsilon}C_{\fin}(\chi)^{1+\varepsilon}.$ It is worth noting that we only require a family containing $\pi$, while the size of $\mathcal{A}$ may be significantly larger than $C_{\fin}(\pi)$. However, our approach of constructing local test functions at nonarchimedean ramified places produces favorable behaviors of regular orbital integrals. Along with sharp estimates of certain trace functions as in \cite{BH08}, we can majorize the regular orbital integrals by $O(C_{\fin}(\pi)^{\varepsilon}C_{\fin}(\chi)^{\frac{1}{2}+\varepsilon}).$ Therefore, 
\begin{equation}\label{1.7}
J_{\Geo}^{\Reg,\heartsuit}(f,\textbf{s}_0,\chi)\ll C_{\fin}(\pi)^{1+\varepsilon}C_{\fin}(\chi)^{1+\varepsilon}+C_{\fin}(\pi)^{\varepsilon}C_{\fin}(\chi)^{\frac{1}{2}+\varepsilon}.
\end{equation}

By combining \eqref{1.6} and \eqref{1.7}, we obtain the inequality 
\begin{align*}
L(1/2,\pi\times\chi)\ll_{\pi,\chi_{\infty}} C_{\fin}(\chi)^{\frac{1}{2}+\varepsilon}+C_{\fin}(\chi)^{\frac{1}{4}+\varepsilon}.
\end{align*}
Here, $C_{\fin}(\chi)^{\frac{1}{2}+\varepsilon}$ is the convex bound and the other term is much smaller, which leaves room for obtaining subconvexity by some analytic techniques.

To make the above illustration rigorous and complete, we employ a variant of the \textbf{RTF} equipped with arithmetic amplification in the spirit of Duke-Friedlander-Iwaniec (cf. \cite{DFI02} and references). The amplified \textbf{RTF} (cf. Theorem \ref{thmD'} in \textsection\ref{3.7.3}) has the form
\begin{equation}\label{1.5}
\mathcal{J}_{\Spec}^{\heartsuit}(\boldsymbol{\alpha},\chi)=\mathcal{J}_{\Geo}^{\heartsuit}(\boldsymbol{\alpha},\chi),
\end{equation}

Through out \textsection\ref{sec3}--\textsection\ref{sec7}, 
we shall prove that the spectral side 
\begin{align*}
\mathcal{J}_{\Spec}^{\heartsuit}(\boldsymbol{\alpha},\boldsymbol{\ell})\gg C_{\infty}(\pi\otimes\chi)^{-\frac{1}{4}-\varepsilon}(C_{\fin}(\pi)C_{\fin}(\chi))^{-\varepsilon} |L(1/2,\pi\times\chi)|^2\cdot L^{2-\varepsilon}
\end{align*} 
if $L\gg_{F,\varepsilon} C_{\fin}(\chi)^{\varepsilon}C_{\infty}(\pi)^{1/2+\varepsilon}C_{\fin}(\pi)^{1+\varepsilon};$ and the geometric side 
\begin{align*}
\mathcal{J}_{\Geo}^{\heartsuit}(\boldsymbol{\alpha},\chi)\ll &[C_{\fin}(\pi),C_{\fin}(\omega_{\pi}))C_{\fin}(\chi)]^{1+\varepsilon}C_{\infty}(\pi\otimes\chi)^{\frac{1}{4}+\varepsilon}L^{1+\varepsilon}\\
&+C_{\infty}(\pi\otimes\chi)^{\varepsilon}C_{\fin}(\pi)^{\varepsilon}C_{\fin}(\chi)^{\frac{1}{2}+\varepsilon}\cdot \gcd(C_{\fin}(\omega_{\pi})),C_{\fin}(\chi))^{\frac{1}{2}}L^{3+\varepsilon}.
\end{align*}	

Here $L$ is the parameter from the amplification, and the factors involving $C_{\infty}(\pi\otimes\chi)$ on both sides follow from applying Nelson's microlocal lifted vectors at the archimedean places. Then Theorem \ref{B} arises from \eqref{1.5} by optimizing $L$ in the above inequalities. 

\begin{remark}
When $F=\mathbb{Q},$ choosing the archimedean test function $f_{\infty}$ via a unipotent transform as in \cite{Sar85}, the \textbf{RTF} approach will give a new proof of the hybrid bounds in \cite{BH08} and \cite{BH14}.  

\end{remark}


\subsection{Outline of the Paper}
\subsubsection{The Relative Trace Formula with  Amplification}
In \textsection\ref{sec2.1}--\textsection\ref{3.2} we set up the intrinsic local and global data, choice of test functions, and parameters for the amplification. In \textsection\ref{2.3}--\textsection\ref{sec3.7} 
we prove the regularized amplified relative trace formula in Theorem \ref{thmD'}.

\subsubsection{The Spectral Side}
In \textsection\ref{sec3} we describe the amplified spectral side and obtain its meromorphic continuation. Together with the local estimates developed in \textsection\ref{sec3.5}--\textsection\ref{sec3.6}, we prove a lower bound of the spectral side (cf. Theorem \ref{thm6}) in terms of central $L$-values and sum of Hecke eigenvalues.

\subsubsection{The Geometric Side}
In \textsection\ref{sec4}--\textsection\ref{sec6} we handle the geometric side $J_{\Geo}^{\Reg,\heartsuit}(f,\mathbf{s}),$ which will be separated into 3 parts as follows.
\begin{enumerate}
	\item the small cell orbital integral $J^{\Reg}_{\Geo,\sm}(f,\mathbf{s})$, which, as one of the main terms, is handled by Proposition \ref{prop12} in \textsection\ref{4.1.4} by combining local estimates from  \textsection\ref{4.1.1}--\textsection\ref{4.1.3}. 
	\item the dual orbital integral $J_{\Geo,\du}^{\bi}(f,\textbf{s})$ is bounded by Proposition \ref{prop17} and Lemmas \ref{lem14.}, \ref{lem5.4} and \ref{lem5.5} in \textsection\ref{sec5}. This integral is `dual' to the small cell  orbital integral $J^{\Reg}_{\Geo,\sm}(f,\mathbf{s})$ via the Poisson summation, and contributes the other main term.
	\item the regular orbital integrals, which comprises the majority of the  geometric side $\mathcal{J}_{\Geo}^{\heartsuit}(\boldsymbol{\alpha},\chi)$. We describe their behaviors by Theorem \ref{thmD} in \textsection\ref{sec6}. The saving in the $C_{\fin}(\chi)$-aspect  follows from a square-root cancellation of trace functions of certain Kummer sheaves and Artin-Schreier sheaves (cf. \textsection\ref{sec6.2}), and the saving in the $C_{\infty}(\pi\otimes\chi)$-aspect is a consequence of properties of microlocal lifted vectors (cf. \textsection\ref{sec6.4}).  
\end{enumerate}
\subsubsection{Proof of Main Results}
With the aforementioned preparations, we are able to prove the main results in \textsection\ref{sec7}. 

In \textsection\ref{sec7.1}--\textsection\ref{sec7.3} we specify the amplification data and put estimates from the spectral and geometric side all together, obtaining Theorem \ref{thmE}, which yields Theorem \ref{B}.

\subsection{Notation}\label{notation}
\subsubsection{Number Fields and Measures}\label{1.1.1}
Let $F$ be a number field with ring of integers $\mathcal{O}_F.$ Let $[F:\mathbb{Q}]$ be the degree. Let $N_F$ be the absolute norm. Let $\mathfrak{O}_F$ be the different of $F.$ Let $\mathbb{A}_F$ be the adele group of $F.$ Let $\Sigma_F$ be the set of places of $F.$ Denote by $\Sigma_{F,\fin}$ (resp. $\Sigma_{F,\infty}$) the set of nonarchimedean (resp. archimedean) places. For $v\in \Sigma_F,$ we denote by $F_v$ the corresponding local field and $\mathcal{O}_v$ its ring of integers. Denote by $\widehat{\mathcal{O}_F}=\prod_{v<\infty}\mathcal{O}_{v}.$ For a nonarchimedean place $v,$ let  $\mathfrak{p}_v$ be the maximal prime ideal in $\mathcal{O}_v.$ Given an integral ideal $\mathcal{I},$ we say $v\mid \mathcal{I}$ if $\mathcal{I}\subseteq \mathfrak{p}_v.$ Fix a uniformizer $\varpi_{v}\in\mathfrak{p}_v$ of the discrete valuation ring $\mathcal{O}_v.$ Denote by $e_v(\cdot)$ the evaluation relative to $\varpi_v$ normalized as $e_v(\varpi_v)=1.$ Denote by $\mathbb{F}_v=\mathcal{O}_v/\mathfrak{p}_v$ the residue field with cardinality $\#\mathbb{F}_v=q_v.$ We use $v\mid\infty$ to indicate an archimedean place $v$ and write $v<\infty$ if $v$ is nonarchimedean. Let $|\cdot|_v$ be the norm in $F_v.$ Put $|\cdot|_{\infty}=\prod_{v\mid\infty}|\cdot|_v$ and $|\cdot|_{\fin}=\prod_{v<\infty}|\cdot|_v.$ Let $|\cdot|_{\mathbb{A}_F}=|\cdot|_{\infty}\otimes|\cdot|_{\fin}$. We will simply write $|\cdot|$ for $|\cdot|_{\mathbb{A}_F}$ in calculation over $\mathbb{A}_F^{\times}$ or its quotient by $F^{\times}$.   

Denote by $\Tr_F$ the trace map, extended to $\mathbb{A}_F\rightarrow \mathbb{A}_{\mathbb{Q}}.$ Let $\psi_{\mathbb{Q}}$ be the additive character on $\mathbb{Q}\backslash \mathbb{A}_{\mathbb{Q}}$ such that $\psi_{\mathbb{Q}}(t_{\infty})=\exp(2\pi it_{\infty}),$ for $t_{\infty}\in \mathbb{R}\hookrightarrow\mathbb{A}_{\mathbb{Q}}.$ Let $\psi_F=\psi_{\mathbb{Q}}\circ \Tr_F.$ Then $\psi_F(t)=\prod_{v\in\Sigma_F}\psi_v(t_v)$ for $t=(t_v)_v\in\mathbb{A}_F.$ For $v\in \Sigma_F,$ let $dt_v$ be the additive Haar measure on $F_v,$ self-dual relative to $\psi_v.$ Then $dt=\prod_{v\in\Sigma_F}dt_v$ is the standard Tamagawa measure on $\mathbb{A}_F$. Let $d^{\times}t_v=\zeta_{F_v}(1)dt_v/|t_v|_v,$ where $\zeta_{F_v}(\cdot)$ is the local Dedekind zeta factor. In particular, $\Vol(\mathcal{O}_v^{\times},d^{\times}t_v)=\Vol(\mathcal{O}_v,dt_v)=N_{F_v}(\mathfrak{D}_{F_v})^{-1/2}$ for all finite place $v.$ Moreover, $\Vol(F\backslash\mathbb{A}_F; dt_v)=1$ and $\Vol(F\backslash\mathbb{A}_F^{(1)},d^{\times}t)=\underset{s=1}{\Res}\ \zeta_F(s),$ where $\mathbb{A}_F^{(1)}$ is the subgroup of ideles $\mathbb{A}_F^{\times}$ with norm $1,$ and $\zeta_F(s)=\prod_{v<\infty}\zeta_{F_v}(s)$ is the finite Dedekind zeta function. Denote by $\widehat{F^{\times}\backslash\mathbb{A}_F^{(1)}}$  the Pontryagin dual of $F^{\times}\backslash\mathbb{A}_F^{(1)}.$

\subsubsection{Reductive Groups}
For an algebraic group $H$ over $F$, we will denote by $[H]:=H(F)\backslash H(\mathbb{A}_F).$ We equip measures on $H(\mathbb{A}_F)$ as follows: for each unipotent group $U$ of $H,$ we equip $U(\mathbb{A}_F)$ with the Haar measure such that, $U(F)$ being equipped with the counting measure and the measure of $[U]$ is $1.$ We equip the maximal compact subgroup $K$ of $H(\mathbb{A}_F)$ with the Haar measure such that $K$ has total mass $1.$ When $H$ is split, we also equip the maximal split torus of $H$ with Tamagawa measure induced from that of $\mathbb{A}_F^{\times}.$

In this paper we set  $A=\diag(\mathrm{GL}(1),1),$ and $G=\mathrm{GL}(2).$ Let $B$ be the group of upper triangular matrices in $G$.  
Let $\overline{G}=Z\backslash G$ and $B_0=Z\backslash B,$ where $Z$ is the center of $G.$ Let $T_B$ be the diagonal subgroup of $B$. Then $A\simeq Z\backslash T_B.$ Let $N$ be the unipotent radical of $B$. Let $K=\otimes_vK_v$ be a maximal compact subgroup of $G(\mathbb{A}_F),$ where $K_v=\mathrm{U}_2(\mathbb{C})$ is $v$ is complex, $K_v=\mathrm{O}_2(\mathbb{R})$ if $v$ is real, and $K_v=G(\mathcal{O}_v)$ if $v<\infty.$ For $v\in \Sigma_{F,\fin},$ $m\in\mathbb{Z}_{\geq 0},$ define 
\begin{equation}\label{2.1}
K_v[m]:=\Big\{\begin{pmatrix}
a&b\\
c&d
\end{pmatrix}\in G(\mathcal{O}_v):\ c\in \varpi_v^{m}\mathcal{O}_v\Big\}.
\end{equation}

\subsubsection{Automorphic Data}
Let $\textbf{s}=(s_1, s_2)\in\mathbb{C}^2.$ 
Let $\omega\in \widehat{F^{\times}\backslash\mathbb{A}_F^{(1)}}.$ Denote by $\mathcal{A}_0\left([G],\omega\right)$ the set of cuspidal representations on $G(\mathbb{A}_F)$ with central character $\omega.$ 

For $\eta_1, \eta_2\in \widehat{F^{\times}\backslash\mathbb{A}_F^{(1)}},$ let $\Ind(\eta_1\otimes\eta_2)$ be the unitary parabolic induction from $B(\mathbb{A}_F)$ to $G(\mathbb{A}_F)$ associated with $\eta_1\otimes\eta_2,$ and let   $\eta_1\boxplus\eta_2$ be Langlands sum. Define 
\begin{equation}\label{R}
\mathcal{R}:=\big\{\mathbf{s}=(s_1,s_2)\in\mathbb{C}^2:\ \Re(s_1),\ \Re(s_2)>-1/2\big\}.
\end{equation}

\subsubsection{Other Conventions}\label{sec1.5.4}
For a function $h$ on $G(\mathbb{A}_F),$ we define $h^*$ by assigning $h^*(g)=\overline{h({g}^{-1})},$ $g\in G(\mathbb{A}_F).$ Let $F_1(s), F_2(s)$ be two meromorphic functions. Write $F_1(s)\sim F_2(s)$ if there exists an \textit{entire} function $E(s)$ such that $F_1(s)=E(s)F_2(s).$ Denote by $\alpha\asymp \beta$ for $\alpha, \beta \in\mathbb{R}$ if there are absolute constants $c$ and $C$ such that $c\beta\leq \alpha\leq C\beta.$

Throughout, we follow the $\varepsilon$-convention: that is, $\varepsilon$ will always be positive number which can be taken as small as we like, but may differ from one occurrence to another.

\textbf{Acknowledgements}
I am deeply grateful to Peter Sarnak for his helpful discussions. I would also like to express my gratitude to Valentin Blomer, Yongxiao Lin, Paul Nelson, and Dinakar Ramakrishnan for their precise comments and valuable suggestions.

\section{Test Functions and the Amplified Relative Trace Formula}\label{sec2}
\textit{The notations presented in this section will be used extensively throughout the remainder of this paper.}
\subsection{Intrinsic Data}\label{sec2.1}
Let $F$ be a number field. Let $\chi=\otimes_v\chi_v$ be a primitive unitary Hecke character of $F^{\times}\backslash\mathbb{A}_F^{\times}$. Let $\pi=\otimes_{v}\pi_v$ be a pure isobaric representation of $G(\mathbb{A}_F)$ with central character $\omega_{\pi}=\omega=\otimes_v\omega_v.$ Throughout this paper we assume that $|L(1/2,\pi\times\chi)|\geq 2.$
\subsubsection{Analytic Conductors}\label{2.1.1}
For $F_v\simeq \mathbb{R},$ $\chi_v=\sgn^{m_v'}|\cdot|^{i\kappa_v},$ $m_v'\in\{0,1\},$ we define 
\begin{align*}
C_v(\chi)=1+\Big|\frac{m_v'+i\kappa_v}{2}\Big|;
\end{align*}
if $F_v\simeq \mathbb{C},$ and $\chi_v(a)=(a|a|^{-1})^{m_v'}|a|^{2i\kappa_v},$ $a\in F_v^{\times},$ we define
\begin{align*}
C_v(\chi):=(1+|i\kappa_v+|m_v'|/2|)^2;
\end{align*}
and for $v<\infty,$ let $\varpi_v^{r_{\chi_v}}$ be the arithmetic conductor of $\chi_v,$ i.e., $r_{\chi_v}$ is the smallest nonnegative integer such that $\chi_v$ is trivial over $1+\varpi_v^{r_{\chi_v}}\mathcal{O}_v^{\times}$ but not over $1+\varpi_v^{r_{\chi_v}-1}\mathcal{O}_v^{\times}.$ The integer $r_{\chi_v}$ is called the exponent of $\chi_v.$ Let $C_v(\chi)=q_v^{r_{\chi_v}}.$ Denote by $C_{\infty}(\chi):=\otimes_{v\mid\infty}C_v(\chi)$ and $C_{\fin}(\chi):=\otimes_{v<\infty}C_v(\chi).$ Let $C(\chi):=C_{\infty}(\chi)C_{\fin}(\chi)$ be the \textit{analytic} conductor of $\chi.$ 

Let $C(\pi):=\otimes_vC(\pi_v)$ be the analytic conductor of $\pi,$ where each local conductor $C(\pi_v)$ is defined as follows.
\begin{itemize}
	\item For $v<\infty,$ denote by  $r_{\pi_v}\geq 0$ the exponent of $\pi_v,$ which is the unique integer such that $\pi_v$ has a vector which is $K_{v}[r_{\pi_v}]$-invariant (cf. \eqref{2.1}) but is not $K_{v}[r_{\pi_v}-1]$-invariant. Let $C_{v}(\pi):=q_v^{r_{\pi_v}}$ be the local conductor of $\pi_v.$ 
	\item For $v\mid\infty,$ the local $L$-function of $\pi_v$ is a product of shifted Gamma factors of the shape $L_v(s,\pi_v)=\Gamma_v(s+\beta_{1,v})\Gamma_v(s+\beta_{2,v}).$ Here $\Gamma_v$ is the Gamma function over $F_v$, and $\beta_{1,v}$ and $\beta_{2,v}$ are complex numbers. Let
	\begin{align*}
		C(\pi_v):=[(1+|\beta_{1,v}|)(1+|\beta_{2,v}|)]^{[F_v:\mathbb{R}]}
	\end{align*}  

\end{itemize}  
 Let $C_{\fin}(\pi)=\prod_{v<\infty}C(\pi_v)$ be the arithmetic conductor of $\pi$ and let $C_{\infty}(\pi)=\prod_{v\mid\infty}C(\pi_v)$ be the archimedean conductor of $\pi.$ We define $C_{\infty}(\pi\otimes\chi)$ and $C(\pi\otimes\chi)$ in a  similar manner.  


\subsubsection{Uniform Parameter Growth}\label{2.1.2}
Let $\Pi$ be an automorphic representation of $\mathrm{GL}(2)/F.$ For $v\mid\infty,$ let $L_v(s,\Pi_v)=\Gamma_v(s+\gamma_{1,v})\Gamma_v(s+\gamma_{2,v})$ be the associated $L$-factor of $\Pi_{v}.$ We say that $\Pi_v$ has \textit{uniform parameter growth of size $(T_v;c_v,C_v)$} if for $1\leq j\leq 2,$ $c_vT_v\leq |\gamma_{j,v}|\leq C_vT_v.$ 
\begin{hy}\label{hy}
Throughout this paper we assume that $\Pi_v=\pi_v\otimes\chi_v$ has \textit{uniform parameter growth of size $(T_v;c_v,C_v)$} for some constants $c_v$ and $C_v,$ and parameters $T_v,$ at all archimedean places $v\mid\infty.$
\end{hy}

In particular, the Hypothesis \ref{hy} holds if $\pi_{\infty}$ is fixed, or $\chi_{\infty}=1$ and $\omega_{\infty}=1.$  

\subsubsection{Ramification Parameters}\label{2.1.5} 
Denote by $\mathfrak{Q}=\prod_{v<\infty}\mathfrak{p}_v^{r_{\chi_v}}.$ For simplicity we write $Q=C_{\fin}(\chi),$ $M=C_{\fin}(\pi),$ and $M'=C(\omega_{\fin}).$ Then $M'\mid M.$ Suppose that $Q>1.$ Let $[M, M'Q]$ be the least common multiple of $M$ and $M'Q.$ 

For $v\in\Sigma_{F,\fin},$ let $e_v(\cdot)$ be the normalized evaluation of $F_v$ such that $e_v(\varpi_v)=1.$ Following the notation in \textsection\ref{2.1.1}, let  $r_{\chi_v}$ (resp. $r_{\omega_v}$) be the exponent of $\chi_v$ (resp. $\omega_v$). Let $r_{\pi_v}$ be the exponent of $\pi_v.$ We set $m_v:=r_{\chi_v},$  and 
$$
n_v:=\max\{r_{\pi_v},r_{\chi_v}+r_{\omega_v}\}\geq m_v.
$$ 
Let $K_v[n_v]$ and $K_v[m_v]$ be defined by \eqref{2.1}. 

\subsubsection{Deformation Parameters}\label{2.1.5.} 
Throughout this paper we shall the deformation parameter by 
\begin{equation}\label{eq2.1}
s_0=\begin{cases}
d'/\sqrt{2},\ \ &\text{if $\pi=\eta\boxplus\eta$ with $\eta^2=\omega_{\pi}$},\\
d',\ \ &\text{otherwise},
\end{cases}
\end{equation}
where $d'\in [2^{-1}\exp(-3\sqrt{\log C(\pi\times\chi)}),\exp(-3\sqrt{\log C(\pi\times\chi)})]$ is determined by $\pi$ and $\chi$ such that for all $s$ with $|s-1/2|=d',$ 
\begin{align*}
|L(1/2,\pi\times\chi)|\ll \exp(\log^{3/4}C(\pi\times\chi))\cdot |L(s,\pi\times\chi)|.
\end{align*}
Here the implied constant depends only on $F.$ Such a $d'$ will be constructed in Lemma \ref{lem3.6} in \textsection\ref{stab}. 

The utilization of the parameter $s_0$ is necessary only when $\pi=\eta\boxplus\eta$ is an Eisenstein series induced from  some Hecke character $\eta$ with $\eta^2=\omega_{\pi},$ in which case $L(1/2,\pi\times\chi)=L(1/2,\eta\chi)^2.$ See \textsection\ref{stab} for details. 

\subsubsection{Hecke Eigenvalues}\label{2.1.4}
For $v\in \Sigma_{F,\fin}$ where $\pi_v$ is unramified, let $L_v(s,\pi)=(1-\alpha_{1,v}q_v^{-s})^{-1}(1-\alpha_{2,v}q_v^{-s})^{-1}$ be the local $L$-factor of $\pi_v,$ where $\alpha_{1,v}$ and $\alpha_{2,v}$ are the Satake parameters. We denote by $\mu_{\pi_v}=\alpha_{1,v}+\alpha_{2,v}$ the the trace of the normalized Hecke operator. For a square-free integral ideal $\mathfrak{m},$ we denote by $\mu_{\pi}(\mathfrak{m}):=\prod_{v\mid\mathfrak{m}}\mu_{\pi_v}$ the $\mathfrak{m}$-th Hecke trace.  

\subsubsection{Amplification Parameters}\label{sec2.1.5}
Let $L\gg 1$ be such that $\log L\asymp \log Q.$ Let $\mathcal{L}$ be a subset of the set $\{\mathfrak{m}\in \mathcal{O}_{F}:\  L<N_F(\mathfrak{m})\leq 2L,\ (\mathfrak{m},\mathfrak{D}_FMQ)=1,\ \text{$\mathfrak{m}$ is square-free}\}.$ Let $\boldsymbol{\alpha}=(\alpha_{\mathfrak{m}})_{\mathfrak{m}\in\mathcal{L}}$ with $\alpha_{\mathfrak{m}}\in\mathbb{C}.$ 

\subsubsection{Other Notations}
For a function $h$ on $G(\mathbb{A})$ or $G(\mathbb{Q}_v),$ $v\in\Sigma_F,$ define $h^*(g)=\overline{h(g^{-1})}$ and 
\begin{equation}\label{3..}
(h*h^*)(g)=\int h(gg'^{-1})h^*(g')dg'=\int h(gg')\overline{h(g')}dg',
\end{equation}
where $g'$ ranges over the domain of $h.$ 

\subsection{Construction of Test Functions}\label{3.2}

We shall construct a test function $f$ on $G(\mathbb{A}_F)$ in the following way.
\begin{itemize}
	\item At the archimedean places (cf. \textsection \ref{3.2.1}), $f_{\infty}$ is built upon Nelson's work \cite{Nel20} (cf. \textsection1.5.2 and \textsection14 on p.80) in the manner described in \cite{NV21}, \textsection1.10. See also \cite{Nel21}, Part 2. 
	\medskip
	\item At ramified finite  places, we construct the local test function via a double average over unipotent translations weighted by characters (cf. \textsection\ref{2.2.2}). This step is essential in handling the level aspect and also simplifies the regularized relative trace formula (cf. \textsection\ref{2.4.1}).  
	\medskip
	\item At certain unramified places we introduce the arithmetic amplification in the spirit of Duke-Friedlander-Iwaniec \cite{DFI02} (cf. \textsection\ref{3.2.4}). 
\end{itemize}

\subsubsection{Construction of ${f}_{\infty}$}\label{3.2.1}
Let $v\mid\infty.$ Recall that $\pi_v\otimes\chi_v$ has \textit{uniform parameter growth of size $(T_v;c_v,C_v)$} (cf. Hypothesis \ref{hy} in \textsection\ref{2.1.2}). Then $\pi_v\otimes\chi_v|\cdot|_v^{s_0}$ has \textit{uniform parameter growth of size $(T_v;c_v/2,2C_v)$}, where $s_0$ is the parameter defined by \eqref{eq2.1} in \textsection\ref{2.1.5.}.

Let $\mathfrak{g}$ (resp. $\mathfrak{g}'$) be the Lie algebras of $G(F_v)$ (resp. $A(F_v)$),  with imaginal dual $\hat{\mathfrak{g}}$ (resp. $\hat{\mathfrak{g}}'$). One can choose an element $\tau\in\hat{\mathfrak{g}}$ with the restriction $\tau'=\tau\mid_{A}\in \hat{\mathfrak{g}}',$ so that $\tau$ (resp. $\tau'$) lies in the coadjoint orbit $\mathcal{O}_{\pi_v\otimes\chi_v|\cdot|_v^{s_0}}$ of $\pi_v\otimes\chi_v|\cdot|_v^{s_0}$ (resp. $\mathcal{O}_{\textbf{1}_v}$ of $\textbf{1}_v$ the trivial representation of $A(F_v)$). Let $\tilde{f}^{\wedge}_{v}:$ $\hat{\mathfrak{g}}\rightarrow\mathbb{C}$ be a smooth bump function concentrated on $\{\tau+(\xi,\xi^{\bot}):\ \xi\ll T_v^{\frac{1}{2}+\varepsilon},\ \xi^{\bot}\ll T_v^{\varepsilon}\},$ where $\xi$ lies in the tangent space of $\mathcal{O}_{\pi_v\otimes\chi_v}$ at $\tau,$ and $\xi^{\bot}$ has the normal direction. Let $\tilde{f}_{v}\in C_c^{\infty}(G(F_v))$ be the pushforward of the Fourier transform of $\tilde{f}_{v}^{\wedge}$ truncated at the essentially support, namely,  
\begin{equation}\label{245}
	\supp \tilde{f}_{v}\subseteq \big\{g\in G(F_v):\ g=I_{n+1}+O(T_v^{-\varepsilon}),\ \Ad^*(g)\tau=\tau+O(T_v^{-\frac{1}{2}+\varepsilon})\big\},
\end{equation}
where the implied constants rely on $c_v$ and $C_v.$ 

Then, in the sense of \cite{NV21}, \textsection {2.5}, the operator $\pi_{v}(\tilde{f}_{v})$ is approximately a rank one projector with range spanned by a unit vector microlocalized at $\tau.$ Let
\begin{equation}\label{3.}
f_{v}(g):=f_v(g,\chi_v)*f_v(g,\chi_v)^*,
\end{equation} 
where $v\mid\infty,$ $g\in G(F_v),$ and 
\begin{equation}\label{6.}
f_v(g,\chi_v):=\chi_v(\det g)|\det g|_v^{s_0}\int_{Z(F_v)}\tilde{f}_{v}(zg)\omega_{v}(z)d^{\times}z.
\end{equation}

By the support of $\tilde{f},$ we have $f_v(g)\neq 0$ unless $|\det g|_v>0.$ So $\sgn(|\det g|_v)=1.$ As a consequence, the function $f_v$ is a smooth function on $G(F_v).$

\subsubsection{Application of Transversality}\label{2.2.2.}
By definition, one has (cf. (14.13) in \cite{Nel21}) 
\begin{equation}\label{250}
	\|\tilde{f}_{v}\|_{\infty}\ll_{\varepsilon} T_v^{1+\varepsilon}, \ \ v\mid\infty,
\end{equation}
where $\|\cdot \|_{\infty}$ is the sup-norm. For $g\in\overline{G}(F_v),$ we may write 
\begin{align*}
g=\begin{pmatrix}
a&b\\
c&d
\end{pmatrix}\in G(F_v),\ \ \ g^{-1}=\begin{pmatrix}
a'&b'\\ 
c'&d'
\end{pmatrix}\in G(F_v).
\end{align*}
Define 
\begin{equation}\label{dg}	
d_{v}(g):=\begin{cases}
\min\big\{1, |d^{-1}b|_v+ |d^{-1}c|_v+ |d'^{-1}b'|_v+|d'^{-1}c'|_v \big\},&\ \text{if $dd'\neq 0,$}\\
1,&\ \text{if $dd'=0$}.
\end{cases}
\end{equation}
\begin{prop}[Theorem 15.1 of \cite{Nel20}]\label{prop3.1}
Let notation be as above. Then there is a fixed neighborhood $\mathcal{Z}$ of the identity in $A(F_v)$ with the following property. Let $g$ be in a small neighborhood of $I_{n+1}$ in $\overline{G}(F_v).$ Let $r_v>0$ be small. Then 
\begin{equation}\label{2.7..}
\Vol\left(\big\{z\in\mathcal{Z}:\ \dist(gz\tau, A(F_v)\tau)\leq r_v\big\}\right)\ll \frac{r_v}{d_v(g)}. 
\end{equation} 
Here $\dist(\cdots)$ denotes the infimum over $g'\in A(F_v)$ of $\|gz\tau-g'\tau\|,$ where $\|\cdot\|$ is a fixed norm on $\hat{\mathfrak{g}}.$
\end{prop}

Proposition \ref{prop3.1} (with $r_v=T_v^{-1/2+\varepsilon}$) will be used to detect the restriction $\Ad^*(g)\tau=\tau+O(T_v^{-\frac{1}{2}+\varepsilon})$ in the support of $\tilde{f}_v.$ By \eqref{245}, \eqref{250}, and \eqref{2.7..}, 
\begin{equation}\label{2.7}
|\tilde{f}_v(g)|\ll T^{1+\varepsilon}\cdot\textbf{1}_{|\Ad^*(g)\tau-\tau|\ll T_v^{-1/2+\varepsilon}}\cdot\textbf{1}_{|g-I_2|\ll T_v^{-\varepsilon}}\cdot \Big\{1,\frac{T_v^{-1/2+\varepsilon}}{d_v(g)}\Big\}.
\end{equation}

\subsubsection{Finite Places}\label{2.2.2}
For $v\in\Sigma_{F,\fin},$ we define a function on $G(F_v),$ supported on $Z(F_v)\backslash K_{v}[n_v],$ by  
\begin{equation}\label{5.}
f_{v}(z_vk_v;\omega_v)={\Vol(\overline{K_{v}[n_v]})^{-1}}\omega_v(z_v)^{-1}\omega_v(E_{2,2}(k_v))^{-1},
\end{equation}
where $\overline{K_v[n_v]}$ is the image of $K_v[n_v]$ in $\overline{G}(F_v),$ and $E_{2,2}(k_v)$ is the $(2,2)$-th entry of $k_v\in K_v[n_v].$ For $g_v\in G(F_v),$ define by  
\begin{align*}
f_v(g_v)=\frac{1}{|\tau(\chi_v)|^{2}}\sum_{\alpha\in (\mathcal{O}_v/\varpi_v^{m_v}\mathcal{O}_v)^{\times}}\sum_{\beta\in (\mathcal{O}_v/\varpi_v^{m_v}\mathcal{O}_v)^{\times}}\chi_v(\alpha)\overline{\chi}_v(\beta)f_{v}\left(g_{\alpha,\beta,v};\omega_v\right),
\end{align*}
where 
$$
\tau(\chi_v)=\sum_{\alpha\in (\mathcal{O}_v/\varpi_v^{m_v}\mathcal{O}_v)^{\times}}\psi_v(\alpha\varpi_v^{-m_v})\chi_v(\alpha)
$$ 
is the Gauss sum relative to the additive character $\psi_v$, and 
\begin{align*}
	g_{\alpha,\beta,v}:=\begin{pmatrix}
		1&\alpha \varpi_v^{-m_v}\\
		&1
	\end{pmatrix}g_v\begin{pmatrix}
		1&\beta \varpi_v^{-m_v}\\
		&1
	\end{pmatrix}. 
\end{align*}
 
Note that $m_v=0$ for almost all $v\in\Sigma_{F,\fin}.$ Hence, for all but finitely many $v\in\Sigma_{F,\fin},$ the test function $f_v(\cdot)=f_v(\cdot;\omega_v)$ (cf. \eqref{5.}) supports in $Z(F_v)\backslash K_{v}[n_v].$

\subsubsection{The Amplification}\label{3.2.4}
For $v<\infty$ and $r_{\pi_v}=r_{\chi_v}=0,$ i.e., both $\pi_v$ and $\chi_v$ are unramified, we let $\mathcal{T}_v=q_v^{-1/2}\textbf{1}_{K_v\diag(\varpi_v,1)K_v},$ the normalized Hecke operator of $\pi_v$ at $v.$ Let $\mathcal{T}_v^*=q_v^{-1/2}\textbf{1}_{K_v\diag(1,\varpi_v^{-1})K_v}$ be the adjoint operator. By Hecke relations, there are constants $c_{v,0}, c_{v,1}\ll 1$ such that 
\begin{align*}
\mathcal{T}_v*\mathcal{T}_v^*=c_{v,0}\textbf{1}_{K_v}+c_{v,1}q_v^{-1}\textbf{1}_{K_v\diag(\varpi_v,\varpi_v^{-1})K_v}.
\end{align*}

For $v_0, v_1, v_2\in \mathcal{L},$ $v_1\neq v_2,$ we define 
$$
{f}_{v_0,i_{v_0}}(g_{v_0})=c_{v_0,i_{v_0}}q_{v_0}^{-i_{v_0}}\int_{Z(F_{v_0})}\textbf{1}_{K_{v_0}\diag(\varpi_{v_0}^{i_{v_0}},\varpi_{v_0}^{-i_{v_0}})K_{v_0}}(z_{v_0}g_{v_0})\omega_{{v_0}}(z_{v_0})d^{\times}z_{v_0},
$$ 
where $i_{v_0}\in\{0,1\},$ $g_{v_0}\in G(F_{v_0});$ and for  $g_{v_j}\in G(F_{v_j}),$ $j=\{1,2\},$ we set 
\begin{align*}
f_{v_1}(g_{v_1})=&q_{v_1}^{-1/2}\int_{Z(F_{v_1})}\textbf{1}_{K_{v_1}\diag(\varpi_{v_1},1)K_{v_1}}(z_{v_1}g_{v_1})\omega_{v_1}(z_{v_1})d^{\times}z_{v_1},\\
f_{v_2}^*(g_{v_2})=&q_{v_1}^{-1/2}\int_{Z(F_{v_2})}\textbf{1}_{K_{v_2}\diag(1,\varpi_{v_2})K_{v_2}}(z_{v_2}g_{v_2})\omega_{v_2}(z_{v_2})d^{\times}z_{v_2}.
\end{align*}

\subsubsection{Construction of the Test Function}\label{3.2.5}
Let notation be as in \textsection\ref{3.2.1}-\textsection\ref{3.2.4}. 
Let $\mathbf{v}_0, \mathbf{v}_1$ and $\mathbf{v}_2$ be finite sets of non-archimedean places of $F.$ Let $\mathbf{v}:=\mathbf{v}_0\sqcup \mathbf{v}_1\sqcup \mathbf{v}_2.$ Suppose that $\mathbf{v}_0\cap \mathbf{v}_1=\mathbf{v}_0\cap \mathbf{v}_2=\mathbf{v}_1\cap \mathbf{v}_2=\emptyset,$ and neither $\pi$ nor $\chi$ is ramified at $v\in\mathbf{v}.$ Define 
\begin{equation}\label{2.8}
f(g;\mathbf{v},\mathbf{i}):=(\otimes_{v\in\Sigma_F-\mathbf{v}}f_v\otimes\otimes_{v_0\in\textbf{v}_0}{f}_{v_0,i_{v_0}}\otimes\otimes_{v_1\in\textbf{v}_1}{f}_{v_1}\otimes\otimes_{v_2\in \textbf{v}_2} f_{ v_2}^*)(g),
\end{equation}
where $\textbf{i}=\{(i_{v_0})_{v_0\in\textbf{v}_0}:\ i_{v_0}\in\{0,1\}\}.$
Here we allow that the sets $\mathbf{v}_j$ to be empty. 

Ultimately we shall take sets $\textbf{v}=\mathbf{v}_0\sqcup \mathbf{v}_1\sqcup \mathbf{v}_2$ as follows. Let $\mathfrak{m}_1, \mathfrak{m}_2\in\mathcal{L}.$ Let $\mathfrak{m}_j=\mathfrak{n}_j\mathfrak{l},$ $j\in\{1, 2\},$ where $\mathfrak{l}=\gcd(\mathfrak{m}_1,\mathfrak{m}_2).$ Let $\mathbf{v}_0=\{\text{$\mathfrak{p}_v$ prime}:\ \mathfrak{p}_v\mid \mathfrak{l}\},$ $\mathbf{v}_j=\{\text{$\mathfrak{p}_v$ prime}:\ \mathfrak{p}_v\mid \mathfrak{n}_j\},$ $j=1, 2.$ See \textsection\ref{3.7.3} for details.

\subsubsection{Some Auxiliary Notations}\label{3.6.2}
Let $f=f(g;\mathbf{v},\mathbf{i}).$ Denote by 
\begin{equation}\label{61}
\nu(f):=\prod_{v_0\in\textbf{v}_0}\varpi_{v_0}\prod_{v\in\textbf{v}_1\sqcup \textbf{v}_2}\varpi_{v},\ \ \mathfrak{N}_f=\prod_{v_0\in\textbf{v}_0}\mathfrak{p}_{v_0}^{2i_{v_0}}\prod_{v\in\textbf{v}_1\sqcup \textbf{v}_2}\mathfrak{p}_{v}.
\end{equation}
Write $\mathcal{N}_f:=\prod_{v_0\in\textbf{v}_0}q_{v_0}^{i_{v_0}}\prod_{v\in\textbf{v}_1\sqcup \textbf{v}_2}q_{v}^{1/2}$ for the square-root of the norm of $\mathfrak{N}_f.$

Note that neither $\pi$ nor $\chi$ is ramified at $v\in\mathbf{v},$ namely,  $r_{\chi_v}=r_{\pi_v}=0,$ $v\in\mathbf{v}.$  

\subsection{Fourier Expansion of the Kernel Function}\label{2.3}
Let $f=f(g;\mathbf{v},\mathbf{i})$ be defined in \textsection\ref{3.2.5}. Then $f$ defines an integral operator 
\begin{equation}\label{p}
	R(f)\phi(g)=\int_{\overline{G}(\mathbb{A}_F)}f(g')\phi(gg')dg'
\end{equation}
on the space $L^2\left([G],\omega\right)$ of functions on $[G]$ which transform under $Z(\mathbb{A}_F)$ by $\omega$ and are square integrable on $[\overline{G}].$ This operator is represented by the kernel function
\begin{equation}\label{11}
	\K(g_1,g_2)=\sum_{\gamma\in \overline{G}(F)}f(g_1^{-1}\gamma g_2),\ \ g_1, g_2\in G(\mathbb{A}_F).
\end{equation}

It is well known that $L^2\left([G],\omega\right)$ decomposes into the direct sum of the space $L_0^2\left([G],\omega\right)$ of cusp forms and spaces $L_{\Eis}^2\left([G],\omega\right)$ and $L_{\Res}^2\left([G],\omega\right)$ defined using Eisenstein series and residues of Eisenstein series respectively. 

The spectral decomposition of the $L^2\left([G],\omega\right)$  (cf. \cite{GJ79}) gives that
\begin{equation}\label{decom}
R=\bigoplus_{\text{$\sigma$ cuspidal}}\sigma\ \oplus \int_{i\mathbb{R}}\bigoplus_{\eta\in\widehat{F^{\times}\backslash \mathbb{A}_F^{(1)}}}\sigma_{\lambda,\eta}\frac{d\lambda}{4\pi i}\ \oplus \bigoplus_{\substack{\eta\in \widehat{F^{\times}\backslash\mathbb{A}_F^{(1)}}\\ \eta^2=\omega}}\eta\circ \det,
\end{equation}
where $\sigma_{\lambda,\eta}=\Ind(\eta|\cdot|^{\lambda},\eta^{-1}\omega|\cdot|^{-\lambda})$ is the parabolic induction.

Then the kernel function $\K(g_1,g_2)$ splits up as the {pre-trace formula}:
\begin{equation}\label{ker}
	\K_0(g_1,g_2)+\K_{\ER}(g_1,g_2)=\K(g_1,g_2)=\sum_{\gamma\in \overline{G}(F)}f(g_1^{-1}\gamma g_2),
\end{equation}
where $\K_{\ER}(g_1,g_2)$ is the contribution from the non-cuspidal spectrum. The subscript `$\ER$' refers to Eisenstein series and their  residues. Explicit expansions of $\K_0(g_1,g_2)$ and $\K_{\ER}(g_1,g_2)$ will be given in \textsection\ref{sec3.1}. 

Note that $\K(g_1,g_2)=\K_{\sm}(g_1,g_2)+\K_{\bi}(g_1, g_2),$ where 
	\begin{align*}
		\K_{\sm}(g_1,g_2)=\sum_{\gamma\in B_0(F)}f(g_1^{-1}\gamma g_2),\quad \K_{\bi}(g_1,g_2)=\sum_{\gamma\in B_0(F)wN(F)}f(g_1^{-1}\gamma g_2).
	\end{align*}
Let $\mathcal{K}(\cdot,\cdot)\in \{\K(\cdot,\cdot), \K_0(\cdot,\cdot), \K_{\ER}(\cdot,\cdot), \K_{\sm}(\cdot,\cdot),\K_{\bi}(\cdot,\cdot)\}.$ Define 
\begin{align*} 
\mathcal{F}_0\mathcal{F}_1\mathcal{K}(g_1,g_2):=&\int_{[N]}\mathcal{K}(g_1,u_2g_2)du_2,\ \ \mathcal{F}_1\mathcal{F}_0\mathcal{K}(g_1,g_2):=\int_{[N]}\mathcal{K}(u_1g_1,g_2)du_1,\\
\mathcal{F}_1\mathcal{F}_1\mathcal{K}(g_1,g_2):=&\int_{[N]}\int_{[N]}\mathcal{K}(u_1g_1,u_2g_2)du_2du_1,\\
\mathcal{F}_2\mathcal{F}_2\K(g_1,g_2):=&\sum_{\alpha\in A(F)}\sum_{\beta\in A(F)}\int_{[N]}\int_{[N]}\K(u_1\alpha g_1,u_2\beta g_2)\theta(u_1)\overline{\theta}(u_2)du_2du_1.
\end{align*}
Using Poisson summation twice the integral $\mathcal{F}_2\mathcal{F}_2\mathcal{K}(g_1,g_2)$ is equal to 
\begin{equation}\label{8}
\mathcal{K}(g_1,g_2)-\mathcal{F}_0\mathcal{F}_1\mathcal{K}(g_1,g_2)-\mathcal{F}_1\mathcal{F}_0\mathcal{K}(g_1,g_2)+\mathcal{F}_1\mathcal{F}_1\mathcal{K}(g_1,g_2).
\end{equation}

Since $f$ has compact support modulo the center, \eqref{8} converges absolutely. 

\begin{lemma}\label{lem3}
Let notation be as before. Then for $x, y\in A(\mathbb{A}_F),$ 
\begin{align*}
\mathcal{F}_0\mathcal{F}_1\K_{\bi}(x,y)=\mathcal{F}_1\mathcal{F}_0\K_{\bi}(x,y)=\mathcal{F}_1\mathcal{F}_1\K_{\bi}(x,y)\equiv 0.
\end{align*}
\end{lemma}
\begin{proof}
By Bruhat decomposition, we have
\begin{align*}
\mathcal{F}_0\mathcal{F}_1\K_{\bi}(x,y)=\sum_{b\in B_0(F)}\int_{N(\mathbb{A})}f(x^{-1}bwu_2y)du_2,
\end{align*}
where $w=\begin{pmatrix}
	&1\\
	1&
\end{pmatrix}.$ Let $v\mid\mathfrak{Q}$ and $m_v=r_{\chi_v}\geq 1.$ By the definition of $f_v,$ the integral  
\begin{equation}\label{9.}
|\tau(\chi_v)|^2\Vol(K_{v}[n_v])\int_{N(F_v)}f_v(x_{v}^{-1}bwu_{v}y_{v})du_{v},
\end{equation}
where $n_v=\max\{r_{\pi_v},r_{\chi_v}+r_{\omega_v}\},$ is equal to 
\begin{equation}\label{10.}
\sum_{\alpha, \beta}\chi_v(\alpha)\overline{\chi}_v(\beta)\int f_v\left(\begin{pmatrix}
	1&\alpha q_v^{m_v}\\
	&1
\end{pmatrix}x_{v}^{-1}bwu_{v}y_{v}\begin{pmatrix}
	1&\beta q_v^{m_v}\\
	&1
\end{pmatrix};\omega_v\right)du_{v}, 
\end{equation}
where $\alpha, \beta\in (\mathcal{O}_v/\varpi_v^{m_v}\mathcal{O}_v)^{\times},$ $u_{v}$ ranges through $N(F_v),$ and the local test function $f_v(\cdot;\omega_v)$ is defined by \eqref{5.}. Since $y_{v}\in A(F_v),$ then after the change of variables $u_{v}\mapsto u_{v}y_{v} \begin{pmatrix}
	1&-\beta q_v^{m_v}\\
	&1
\end{pmatrix}y_{v}^{-1},$ \eqref{10.} becomes
\begin{align*}
\sum_{\beta}\overline{\chi}_v(\beta)\sum_{\alpha}\chi_v(\alpha)\int f_v\left(\begin{pmatrix}
	1&\alpha q_v^{m_v}\\
	&1
\end{pmatrix}x_{v}^{-1}bwu_{v}y_{v};\omega_v\right)du_{v}=0,
\end{align*}
as a consequence of orthogonality $\sum_{\beta}\overline{\chi}_v(\beta)=0.$

Hence \eqref{9.} vanishes, which implies that $\mathcal{F}_0\mathcal{F}_1\K_{\bi}(x,y)=0.$ Likewise, we have $\mathcal{F}_1\mathcal{F}_0\K_{\bi}(x,y)=\mathcal{F}_1\mathcal{F}_1\K_{\bi}(x,y)=0.$
\end{proof}

\subsection{The Amplified Relative Trace Formula}\label{sec3.7}
Let the local and global data be chosen as in \textsection\ref{sec2.1}--\textsection\ref{3.2}: $\mathcal{L},$ $\boldsymbol{\alpha},$ and $f=f(g;\mathbf{v},\mathbf{i}).$ 
\subsubsection{The Relative Trace Formula}\label{2.4.1}
Let $\Re(s_1)\gg 1$ and $\Re(s_2)\gg 1.$ Define
\begin{equation}\label{11}
J_{\Spec}^{\Reg}(f,\textbf{s},\chi):=J_{0}^{\Reg}(f,\textbf{s},\chi)+J_{\Eis}^{\Reg}(f,\textbf{s},\chi),
\end{equation}
the spectral side, where 
\begin{align*}
J_{0}^{\Reg}(f,\textbf{s},\chi):=&\int_{[A]}\int_{[A]}\K_0(x,y)|\det x|^{s_1}|\det y|^{s_2}\chi(x)\overline{\chi}(y)d^{\times}xd^{\times}y,\\
J_{\Eis}^{\Reg}(f,\textbf{s},\chi):=&\int_{[A]}\int_{[A]}\mathcal{F}_2\mathcal{F}_2\K_{\ER}(x,y)|\det x|^{s_1}|\det y|^{s_2}\chi(x)\overline{\chi}(y)d^{\times}xd^{\times}y.
\end{align*}

By Proposition 6.4 in \cite{Yan23a} (cf. \textsection 6.2), the integral $J_{\Spec}^{\Reg}(f,\textbf{s},\chi)$ converges absolutely in $\Re(s_1), \Re(s_2)\gg 1.$ 

Note that $\K_0(x,y)=\mathcal{F}_2\mathcal{F}_2\K_0(x,y).$ By \eqref{8}, Lemma \ref{lem3} and the decomposition $\K(x,y)=\K_{\sm}(x,y)+\K_{\bi}(x,y),$ the integral $J_{\Spec}^{\Reg}(f,\textbf{s},\chi)$ is equal to 
\begin{equation}\label{2.17}
J_{\Geo}^{\Reg}(f,\textbf{s},\chi):=J^{\Reg}_{\Geo,\sm}(f,\textbf{s},\chi)+J^{\Reg}_{\Geo,\bi}(f,\textbf{s},\chi),
\end{equation}
the geometric side, where $\Re(s_1)\gg 1,$ $\Re(s_2)\gg 1,$ and 
\begin{align*}
J^{\Reg}_{\Geo,\sm}(f,\textbf{s},\chi):=&\int_{[A]}\int_{[A]}\mathcal{F}_2\mathcal{F}_2\K_{\sm}(x,y)|\det x|^{s_1}|\det y|^{s_2}\chi(x)\overline{\chi}(y)d^{\times}xd^{\times}y
,\\
J^{\Reg}_{\Geo,\bi}(f,\textbf{s},\chi):=&\int_{[A]}\int_{[A]}\K_{\bi}(x,y)|\det x|^{s_1}|\det y|^{s_2}\chi(x)\overline{\chi}(y)d^{\times}xd^{\times}y.
\end{align*}

Note that $J_{\Geo}^{\Reg}(f,\textbf{s},\chi)$ is much simpler than its higher rank counterpart in \cite{Yan23a}. In particular, the contributions from $\mathcal{F}_0\mathcal{F}_1\K_{\bi}(\cdot,\cdot),$ $\mathcal{F}_1\mathcal{F}_0\K_{\bi}(\cdot,\cdot)$ and $\mathcal{F}_1\mathcal{F}_1\K_{\bi}(\cdot,\cdot)$ are annihilated by the construction of the local test functions at $v\mid\mathfrak{Q}$ (i.e., Lemma \ref{lem3}). As in \cite{RR05} (or \cite{Yan23a} \textsection 5) we have
\begin{equation}\label{15.}
J^{\Reg}_{\Geo,\bi}(f,\textbf{s},\chi)=J^{\Reg}_{\Geo,\du}(f,\textbf{s},\chi)+J^{\Reg,\RNum{2}}_{\Geo,\bi}(f,\textbf{s},\chi),	
\end{equation}
where $J_{\Geo,\du}^{\bi}(f,\textbf{s},\chi)$ is defined by 
\begin{equation}\label{16.}
\int_{\mathbb{A}_F^{\times}}\int_{\mathbb{A}_F^{\times}}f\left(\begin{pmatrix}
	1\\
	x&1
\end{pmatrix}\begin{pmatrix}
	y\\
	&1
\end{pmatrix}\right)|x|^{s_1+s_2}|y|^{s_2}\overline{\chi}(y)d^{\times}yd^{\times}x,
\end{equation}
and the regular orbital $J^{\Reg,\RNum{2}}_{\Geo,\bi}(f,\textbf{s},\chi)$ is defined by 
\begin{equation}\label{17...} 
\sum_{t\in F-\{0,1\}}\int_{\mathbb{A}_F^{\times}}\int_{\mathbb{A}_F^{\times}}f\left(\begin{pmatrix}
	y&x^{-1}t\\
	xy&1
\end{pmatrix}\right)|x|^{s_1+s_2}|y|^{s_2}\overline{\chi}(y)d^{\times}yd^{\times}x.
\end{equation}
Note that \eqref{16.} converges absolutely in $\Re(s_1+s_2)>1,$ and by Theorem 5.6 in \textsection 5.4 of \cite{Yan23a} the integral $J^{\Reg,\RNum{2}}_{\Geo,\bi}(f,\textbf{s},\chi)$ converges absolutely in $\mathbf{s}\in\mathbb{C}^2,$ and in particular, the sum over $t\in F-\{0,1\}$ is \textit{finite}, which is called \textit{stability} of the regular orbital integrals (cf. \cite{FW09}, \cite{MR12}).   

In \textsection\ref{sec3} we will show that the spectral side $J_{\Spec}^{\Reg}(f,\textbf{s},\chi)$ admits a holomorphic continuation $J_{\Spec}^{\Reg,\heartsuit}(f,\textbf{s},\chi)$ to $\mathbf{s}\in\mathbb{C}^2.$ In \textsection \ref{sec4}-\textsection\ref{sec6} we will derive a holomorphic continuation $J_{\Geo}^{\Reg,\heartsuit}(f,\textbf{s},\chi)$ of the geometric side $J_{\Geo}^{\Reg}(f,\textbf{s},\chi)$ to $\mathbf{s}\in\mathbb{C}^2.$ Therefore, we obtain the relative trace formula as an equality of two holomorphic functions:
\begin{thm}\label{thm2.4}
Let notation be as before. Then 
\begin{equation}\label{2.24}
J_{\Spec}^{\Reg,\heartsuit}(f,\textbf{s},\chi)=J_{\Geo}^{\Reg,\heartsuit}(f,\textbf{s},\chi).
\end{equation}
\end{thm}

\subsubsection{The Amplified Relative Trace Formula}\label{3.7.3}
We define the spectral and the geometric side of the amplified relative trace formula respectively by
\begin{align*}
\mathcal{J}_{\Spec}^{\heartsuit}(\boldsymbol{\alpha},\chi):=&\sum_{\mathfrak{m}_1,\mathfrak{m}_2\in\mathcal{L}}\alpha_{\mathfrak{m}_1}\overline{\alpha_{\mathfrak{m}_2}}\prod_{v_0\mid\gcd(\mathfrak{m}_1,\mathfrak{m}_2)}\sum_{i_{v_0}=0}^{1}c_{v_0,i_{v_0}}J_{\Spec}^{\Reg,\heartsuit}(f(\cdot;\mathbf{v},\mathbf{i}),\mathbf{s}_0),\\
\mathcal{J}_{\Geo}^{\heartsuit}(\boldsymbol{\alpha},\chi):=&\sum_{\mathfrak{m}_1,\mathfrak{m}_2\in\mathcal{L}}\alpha_{\mathfrak{m}_1}\overline{\alpha_{\mathfrak{m}_2}}\prod_{v_0\mid\gcd(\mathfrak{m}_1,\mathfrak{m}_2)}\sum_{i_{v_0}=0}^{1}c_{v_0,i_{v_0}}J_{\Geo}^{\Reg,\heartsuit}(f(\cdot;\mathbf{v},\mathbf{i}),\mathbf{s}_0),
\end{align*}
where $\textbf{s}_0=(s_0,s_0),$ with $s_0$ is defined in \textsection\ref{sec2.1}, and $\textbf{v}=\mathbf{v}_0\sqcup \mathbf{v}_1\sqcup \mathbf{v}_2$ and $\textbf{i}=\{(i_{v_0})_{v_0\in\textbf{v}_0}:\ i_{v_0}\in\{0,1\}\}$ are defined from $\mathfrak{m}_1,$ $\mathfrak{m}_2$ as follows: $\mathbf{v}_0=\{\text{$\mathfrak{p}$ prime}:\ \mathfrak{p}\mid \gcd(\mathfrak{m}_1,\mathfrak{m}_2)\},$ $\mathbf{v}_j=\{\text{$\mathfrak{p}$ prime}:\ \mathfrak{p}\mid \mathfrak{m}_j/\gcd(\mathfrak{m}_1,\mathfrak{m}_2)\},$ $j=1, 2.$ (cf. \textsection\ref{3.2.5} for the notations.)

\begin{thmx}\label{thmD'}
Let notation be as before. Then 
\begin{equation}\label{3.7}
	\mathcal{J}_{\Spec}^{\heartsuit}(\boldsymbol{\alpha},\chi)=\mathcal{J}_{\Geo}^{\heartsuit}(\boldsymbol{\alpha},\chi).
\end{equation}
\end{thmx}

\section{The Spectral Side: Meromorphic Continuation and Bounds}\label{sec3}
In this section we shall show that $J_{\Spec}^{\Reg,\heartsuit}(f,\textbf{s},\chi)$ admits a holomorphic continuation to $\mathbf{s}\in\mathbb{C}^2.$ Hence, the spectral side $\mathcal{J}_{\Spec}^{\heartsuit}(\boldsymbol{\alpha},\chi)$ of the amplified relative trace formula is well defined. Moreover, we derive a lower bound of it as follows. 
\begin{restatable}[]{thm}{thmf}\label{thm6}
Let notation be as in \textsection\ref{sec2.1}--\textsection\ref{sec3.7}. Suppose that $|L(1/2,\pi\times\chi)|\geq 1.$
\begin{enumerate}
	\item[(a).] Suppose $\pi$ is cuspidal or $\pi=\sigma_{\lambda',\eta'}$ with $\lambda'\neq 0$ or $\omega\neq \eta'^2.$ then  
\begin{equation}\label{16..}
\mathcal{J}_{\Spec}^{\heartsuit}(\boldsymbol{\alpha},\chi)\gg_{\varepsilon}C_{\infty}(\pi\otimes\chi)^{-\frac{1}{4}-\varepsilon}(MQ)^{-\varepsilon} |L(1/2,\pi\times\chi)|^2\cdot \Big|\sum_{\mathfrak{m}\in\mathcal{L}}\alpha_{\mathfrak{m}}\mu_{\pi}(\mathfrak{m})\Big|^2,
\end{equation}
where $\mu_{\pi}(\mathfrak{m})$ is the $\mathfrak{m}$-th Hecke trace, cf. \textsection\ref{2.1.4}, and  the implied constant depends only on $F,$ $\varepsilon,$ and $c_v,$ $C_v$ at $v\mid\infty,$ cf. \textsection \ref{2.1.2}. 
\item[(b).] Suppose $\pi=\sigma_{\lambda',\eta'}$ with $\lambda'=0$ and $\omega=\eta'^2.$ Then 
\begin{equation}\label{17..}
\mathcal{J}_{\Spec}^{\heartsuit}(\boldsymbol{\alpha},\chi)\gg_{\varepsilon}C_{\infty}(\pi\otimes\chi)^{-\frac{1}{4}-\varepsilon}(MQ)^{-\varepsilon} |L(1/2,\pi\times\chi)|^2\cdot \Big|\sum_{\mathfrak{m}\in\mathcal{L}}\alpha_{\mathfrak{m}}\mu_{\pi^{\dagger}}(\mathfrak{m})\Big|^2,
\end{equation}
where the implied constant depends only on $F,$ $\varepsilon,$ and $c_v,$ $C_v$ at $v\mid\infty,$ cf. \textsection \ref{2.1.2}. Here $\pi^{\dagger}:=\sigma_{is_0,\eta'},$ with $s_0$ being defined by \eqref{eq2.1} in \textsection\ref{2.1.5.}. 
\end{enumerate}
\end{restatable}

\subsection{Eisenstein series}
Define $H_T:$ $T(\mathbb{A}_F)\longrightarrow\mathbb{R}$ by setting $H\left(\begin{pmatrix}
	a\\
	&d
\end{pmatrix}\right)=\log |ad^{-1}|_{\mathbb{A}_F}.$ Extend $H_T$ to a map $H:$ $G(\mathbb{A}_F)\longrightarrow\mathbb{R}$ by using the Iwasawa decomposition $G(\mathbb{A}_F)=N(\mathbb{A}_F)T(\mathbb{A}_F)K,$ and defining $H(utk)=H_T(t),$ $u\in N(\mathbb{A}_F),$ $t\in T(\mathbb{A}_F),$ and $k\in K.$  

For $\eta\in \widehat{F^{\times}\backslash\mathbb{A}_F^{(1)}},$ let $\sigma_{0,\eta}=\Ind(\eta,\eta^{-1}\omega)$ be the parabolic induction. Let $\phi\in\sigma_{0,\eta}$ be a section. Define the Eisenstein series 
\begin{align*}
E(g,\phi,\lambda)=\sum_{\delta\in B(F)\backslash G(F)}\phi(\delta g)e^{(1/2+\lambda)H(\delta g)},\ \ \Re(\lambda)>0.
\end{align*}
Then $E(g,\phi,\lambda)$ converges absolutely and admits a meromorphic continuation to $\lambda\in\mathbb{C}$ with possible simple poles at $\lambda\in \{0,1\}$ as will as a functional equation. 

Let $f=f(g;\mathbf{v},\mathbf{i})$ be constructed in \textsection\ref{3.2}. For $g\in G(\mathbb{A}_F),$ define   
\begin{align*}
E(g,\mathcal{I}(\lambda, f)\phi,\lambda)=\sigma_{0,\eta}(f)E(g,\phi,\lambda):=\int_{\overline{G}(\mathbb{A}_F)}f(y)E(gy,\phi,\lambda)dy,\ \ \Re(\lambda)>0.
\end{align*}

To simply notation, we also use  $E(g,\phi,\lambda)$ and $E(g,\mathcal{I}(\lambda, f)\phi,\lambda)$ to denote their meromorphic continuations.

\subsection{Spectral Expansion of the kernel functions}\label{sec3.1}
Let notation be as in \textsection\ref{3.2}. Let $v_0, v_1, v_2\in \mathcal{L},$ $v_1\neq v_2,$ $i\in\{0,1\}.$ Let $f=f(g;\mathbf{v},\mathbf{i}).$  Let $\K_0(x,y)$ and $\K_{\ER}(x,y)$ be defined by \eqref{ker} in \textsection \ref{2.3}. Then by the decomposition \eqref{decom} we have (e.g., cf. \cite{Art79}) 
\begin{align*}
&\K_0(x,y)=\sum_{\sigma\in \mathcal{A}_0([G],\omega)}\sum_{\phi\in\mathfrak{B}_{\sigma}}\sigma(f)\phi(x)\overline{\phi(y)},\\
&\K_{\ER}(x,y)=\frac{1}{4\pi}\sum_{\eta\in \widehat{F^{\times}\backslash\mathbb{A}_F^{(1)}}}\int_{i\mathbb{R}}\sum_{\phi\in \mathfrak{B}_{\sigma_{0,\eta}}}E(x,\mathcal{I}(\lambda,f)\phi,\lambda)\overline{E(y,\phi,\lambda)}d\lambda,
\end{align*}
where $\mathfrak{B}_{\sigma}$ is an orthonormal basis of the cuspidal representation $\sigma,$ and $\sigma_{0,\eta}=\Ind(\eta,\eta^{-1}\omega).$ Note that since $f$ is $K$-finite, the inner sums over $\phi$ in both $\K_{0}(x,y)$ and $\K_{\ER}(x,y)$ are finite.

\subsection{Rankin-Selberg Periods}\label{sec3.3}
Let $\varphi$ be an automorphic form on $G(\mathbb{A}_F).$ Define the associated Whittaker function by 
\begin{equation}\label{whi}
W_{\varphi}(g):=\int_{[N]}\varphi(ug)\overline{\theta(u)}du, \ \ g\in G(\mathbb{A}_F),
\end{equation}
where $\theta$ is the generic induced by the fixed additive character $\psi$ (cf. \textsection\ref{1.1.1}). By multiplicity one, $W_{\varphi}$ factors into local integrals, i.e., $W_{\varphi}(g)=\prod_{v\in\Sigma_F}W_{\varphi,v}(g_v),$ where $g=\otimes_vg_v\in G(\mathbb{A}_F)$ and $W_{\varphi,v}$ is a local Whittaker function associated with the local additive character $\psi_v,$ moreover, $W_{\varphi,v}$ is spherical for all but finitely many places $v\in\Sigma_F.$ Here $I_2$ is the identity in $G(F_v).$  Define  
\begin{align*}
\Psi(s,\varphi,\chi):=\int_{\mathbb{A}_F^{\times}}W_{\varphi}\left(\begin{pmatrix}
	x\\
	&1
\end{pmatrix}\right)|x|^s\chi(x)d^{\times}x=\prod_{v\in\Sigma_F}\Psi_v(s,\varphi,\chi),
\end{align*}
where the local integral is defined by 
\begin{align*}
\Psi_v(s,\varphi,\chi)=\int_{F_v^{\times}}W_{\varphi,v}\left(\begin{pmatrix}
	x_v\\
	&1
\end{pmatrix}\right)|x_v|_v^s\chi_v(x_v)d^{\times}x_v.
\end{align*}

By standard estimates of $W_{\varphi,v}$ the integral $\Psi(s,\varphi,\chi)$ converges absolutely in $\Re(s)>1/2.$ Furthermore, it is related to $L$-functions as follows. 

If $\varphi\in\mathfrak{B}_{\sigma},$ where $\sigma\in \mathcal{A}_0([G],\omega),$ then by Fourier expansion  
\begin{align*} 
\Psi(s,\varphi,\chi)=\int_{F^{\times}\backslash\mathbb{A}_F^{\times}}\varphi\left(\begin{pmatrix}
	x\\
	&1
\end{pmatrix}\right)|x|^s\chi(x)d^{\times}x,
\end{align*}
which, by the rapid decay of $\varphi,$ converges absolutely for all $s\in\mathbb{C}.$ Hence, the function $\Psi(s,\varphi,\chi)$ is entire. By Hecke's theory, $\Psi(s,\varphi,\chi)$ is the integral representation for the complete $L$-function $\Lambda(s+1/2,\sigma).$

If $\varphi\in\mathfrak{B}_{0,\eta}$ for some $\eta\in\widehat{F^{\times}\backslash\mathbb{A}_F^{(1)}},$ then by \cite{Jac09}, each local integral  
$$
\int_{F_v^{\times}}W_{\varphi,v}\left(\begin{pmatrix}
	x_v\\
	&1
\end{pmatrix}\right)|x_v|_v^s\chi_v(x_v)d^{\times}x_v 
$$
represents the local $L$-function $L_v(s+1/2,\eta_v)L_v(s+1/2,\eta_v^{-1}\omega_v).$ Therefore, the function $\Psi(s,\varphi,\chi)$ admits a meromorphic continuation to $s\in\mathbb{C},$ which represents the complete $L$-function $\Lambda(s+1/2,\eta)\Lambda(s+1/2,\eta^{-1}\omega).$ 

\subsection{The Spectral Side: meromorphic continuation}\label{sec3.4}
According to the construction of the test function $f,$ the Eisenstein series $E(x,\mathcal{I}(\lambda,f)\phi,\lambda),$ $\phi\in\mathfrak{B}_{0,\eta},$ vanishes unless $\phi$ is right invariant under $K_v[n_v],$ where $v\mid\mathfrak{Q},$ and $m_v=r_{\chi_v}\geq 1.$

 
Substituting the Rankin-Selberg periods (cf. \textsection\ref{sec3.3}) into the decomposition \eqref{11} we then obtain $J_{\Spec}^{\Reg}(f,\textbf{s},\chi)=J_{0}^{\Reg}(f,\textbf{s},\chi)+J_{\Eis}^{\Reg}(f,\textbf{s},\chi),$ where 
\begin{align*}
J_{0}^{\Reg}(f,\textbf{s},\chi)=&\sum_{\sigma\in\mathcal{A}_0([G],\omega,\chi)}\sum_{\phi\in\mathfrak{B}_{\sigma}}\Psi(s_1,\sigma(f)\phi)\Psi(s_2,\overline{\phi},\overline{\chi}),\\
J_{\Eis}^{\Reg}(f,\textbf{s},\chi)=&\frac{1}{4\pi}\sum_{\eta}\int_{i\mathbb{R}}\sum_{\phi\in \mathfrak{B}_{\sigma_{0,\eta}}}\Psi(s_1,\sigma_{0,\eta}(f)E(\cdot,\phi,\lambda),\chi)\Psi(s_2,\overline{E(\cdot, \phi,\lambda)},\overline{\chi})d\lambda,
\end{align*}
where $\eta$ ranges through $\in \widehat{F^{\times}\backslash\mathbb{A}_F^{(1)}},$ $\Re(s_1)\gg 1$ and $\Re(s_2)\gg 1.$
\begin{prop}\label{prop6}
Let notation be as before. Then the function $J_{\Eis}^{\Reg}(f,\textbf{s},\chi)$ extends to a holomorphic function $J_{\Eis}^{\Reg,\heartsuit}(f,\textbf{s},\chi)$ in $0<\Re(s_1), \Re(s_2)<1/4$ with $J_{\Eis}^{\Reg,\heartsuit}(f,\textbf{s},\chi)$ given by 
\begin{equation}\label{24'''}
\frac{1}{4\pi}\sum_{\eta}\int_{i\mathbb{R}}\sum_{\phi\in \mathfrak{B}_{\sigma_{0,\eta}}}\Psi(s_1,\sigma_{0,\eta}(f)E(\cdot,\phi,\lambda),\chi)\Psi(s_2,\overline{E(\cdot, \phi,\lambda)},\overline{\chi})d\lambda,
\end{equation}
where the integrand $\Psi(s_1,\sigma_{0,\eta}(f)E(\cdot,\phi,\lambda),\chi)\Psi(s_2,\overline{E(\cdot, \phi,\lambda)},\overline{\chi})$ is identified with its meromorphic continuation. In particular, $J_{\Spec}^{\Reg}(f,\textbf{s},\chi)$ is holomorphic in the region $0<\Re(s_1), \Re(s_2)<1/4.$
\end{prop}
\begin{proof}
The function $J_{0}^{\Reg}(f,\textbf{s},\chi)$ converges absolutely in $\mathbf{s}\in\mathbb{C}^2.$ Moreover, for $\eta\in \widehat{F^{\times}\backslash\mathbb{A}_F^{(1)}}-\{\chi^{-1},\chi\omega\},$ by the discussion in \textsection\ref{sec3.3} the function $\Psi(s_1,\sigma(f)\varphi,\chi)\Psi(s_2,\overline{\varphi},\overline{\chi})$ admits a holomorphic continuation to $\mathbf{s}\in\mathbb{C}^2,$ where $\varphi(\cdot)=E(\cdot,\phi,\lambda).$

Now we consider the case that $\eta\in\{\chi^{-1},\chi\omega\}.$ For $\Re(s_1), \Re(s_2)\gg 1,$ let 
\begin{equation}\label{24'}
J_{\eta}(f,\textbf{s},\chi):=\frac{1}{4\pi}\int_{i\mathbb{R}}\sum_{\phi\in \mathfrak{B}_{\sigma_{0,\eta}}}\Psi(s_1,\sigma_{0,\eta}(\tilde{f})E(\cdot,\phi,\lambda),\chi)\Psi(s_2,\overline{E(\cdot, \phi,\lambda)},\overline{\chi})d\lambda.
\end{equation}

By \textsection 6.3.3 in \cite{Yan23a} we have a meromorphic continuation 
\begin{align*}
\tilde{J}_{\eta}(f,\textbf{s},\chi)=J_{\eta}^{\heartsuit}(f,\textbf{s},\chi)-\mathcal{G}_{\eta}(f,\mathbf{s},\chi)+\widetilde{\mathcal{G}}_{\eta}(f,\mathbf{s},\chi),
\end{align*}
where $0<\Re(s_1), \Re(s_2)<1/4,$ the function $J_{\eta}^{\heartsuit}(f,\textbf{s},\chi)$ is defined by \eqref{24'} with the integrand $\Psi(s_1,\sigma_{0,\eta}(f)E(\cdot,\phi,\lambda))\Psi(s_2,\overline{E(\cdot, \phi,\lambda)})$ is understood as its meromorphic continuation, and by Theorem 7.1 in loc. cit. the function $\mathcal{G}_{\eta}(f,\mathbf{s},\chi)$ is the meromorphic continuation of the Tate integral  
\begin{equation}\label{25'}
-\int_{\mathbb{A}_F^{\times}}\int_{\mathbb{A}_F^{\times}}\int_{\mathbb{A}_F}\int_{\mathbb{A}_F}f\left(\begin{pmatrix}
	1&\\
	c&1
\end{pmatrix}\begin{pmatrix}
	y&b\\
	&1
\end{pmatrix}\right)\psi(cx)dbdc|x|^{s_1+s_2}
d^{\times}x\overline{\chi}(y)|y|^{s_2}d^{\times}y,
\end{equation}
in the region $\Re(s_1+s_2)>1;$ and the function $\widetilde{\mathcal{G}}_{\eta}(f,\mathbf{s},\chi)$ is the 
meromorphic continuation of the Tate integral  
\begin{align*}
-\int_{\mathbb{A}_F^{\times}}\int_{\mathbb{A}_F^{\times}}\int_{\mathbb{A}_F}\int_{\mathbb{A}_F}f\left(\begin{pmatrix}
	1&b\\
	&1
\end{pmatrix}\begin{pmatrix}
	1&\\
	c&1
\end{pmatrix}\begin{pmatrix}
	y&\\
	&1
\end{pmatrix}\right)\psi(cx)dbdc|x|_{\mathbb{A}_F}^{s_1+s_2}
d^{\times}x\overline{\chi}(y)|y|_{\mathbb{A}_F}^{s_2}d^{\times}y
\end{align*}
in the region $\Re(s_1+s_2)>1.$ 

Consider the absolutely convergent integral $\mathcal{J}(f;x,y,c)=\prod_{v}\mathcal{J}_v(f;x,y,c),$ where 
\begin{align*}
\mathcal{J}_v(f;x,y,c):=\int_{F_v}f_v\left(\begin{pmatrix}
	1&\\
	c_v&1
\end{pmatrix}\begin{pmatrix}
	y_v&b_v\\
	&1
\end{pmatrix}\right)db_v.
\end{align*}
Let $v\mid\mathfrak{Q}.$ By definition $\mathcal{J}_v(f;x,y,c)$ is equal to $|\tau(\chi_v)|^{-2}$ multiplying 
\begin{align*}
\sum_{\alpha, \beta}\chi_v(\alpha)\overline{\chi}_v(\beta)\int_{F_v}f_v\left(\begin{pmatrix}
	1&\alpha\varpi_v^{-m_v}\\
	&1
\end{pmatrix}\begin{pmatrix}
	1&\\
	c_v&1
\end{pmatrix}\begin{pmatrix}
	y_v&b_v\\
	&1
\end{pmatrix}\begin{pmatrix}
	1&\beta\varpi_v^{-m_v}\\
	&1
\end{pmatrix}\right)db_v,
\end{align*} 
where $\alpha, \beta\in (\mathcal{O}_v/\varpi_v^{m_v}\mathcal{O}_v)^{\times}.$ Changing variable $b_v\mapsto b_v-y_v\beta\varpi_v^{-m_v},$ the function $|\tau(\chi_v)|^{2}\mathcal{J}_v(f;x,y,c)$ becomes 
\begin{equation}\label{26'}
\sum_{\beta}\overline{\chi}_v(\beta)\sum_{\alpha}\chi_v(\alpha)\int_{F_v}f_v\left(\begin{pmatrix}
	1&\alpha\varpi_v^{-m_v}\\
	&1
\end{pmatrix}\begin{pmatrix}
	1&\\
	c_v&1
\end{pmatrix}\begin{pmatrix}
	y_v&b_v\\
	&1
\end{pmatrix}\right)db_v,
\end{equation}

Since $\chi_v$ is ramified, then by orthogonality, $\sum_{\beta}\overline{\chi}_v(\beta)=0.$ So \eqref{26'} vanishes, which implies that $\mathcal{J}_v(f;x,y,c)=0.$ So $\mathcal{J}(f;x,y,c)=0.$ As a consequence, \eqref{25'} vanishes in the region $\Re(s_1+s_2)>1,$ yielding that $\mathcal{G}_{\eta}(f,\mathbf{s},\chi)\equiv 0.$ Similarly, $\widetilde{\mathcal{G}}_{\eta}(f,\mathbf{s},\chi)\equiv 0.$

Hence, by Lemma 6.7 in \cite{Yan23a} the integral \eqref{24'''} converges absolutely in $0<\Re(s_1), \Re(s_2)<1/4,$ and gives the a holomorphic continuation of $J_{\Eis}^{\Reg}(f,\textbf{s},\chi)$ therein.  Therefore, Proposition \ref{prop6} follows.
\end{proof}

\subsection{The Spectral Side: a variant and local calculations}\label{sec3.5}
\subsubsection{Rewrite the spectral expansion}\label{3.5.1}
Set $n_v=\max\{r_{\pi_v},r_{\chi_v}+r_{\omega_v}\}$ and $m_v=r_{\chi_v}$ as before. Let ${f}_v$ and ${f}_{v_0,i}$ be defined in \textsection\ref{3.2.1}-\textsection\ref{3.2.4}.

Let ${f}^{\dagger}(g;i,v_0):=\otimes_{v\mid\infty}f_v(g_v,\chi_v)\otimes\otimes_{v\in \Sigma_{F,\fin}-\{v_0\}}f_v(g_{v};\omega_{v})\otimes{f}_{v_0,i}(g_{v_0}),$ and ${f}^{\dagger}(g;v_1,v_2):=\otimes_{v\mid\infty}f_v(g_v,\chi_v)\otimes{f}_{v_1}(g_{v_1})\otimes {f}_{v_2}^*(g_{v_2})\otimes\otimes_{v\in \Sigma_{F,\fin}-\{v_1,v_2\}}f_v(g_{v};\omega_{v}).$ Denote by ${f}^{\ddagger}(g)=\otimes_{v\mid\infty}f_v(g_v,\chi_v)\otimes\otimes_{v\in \Sigma_{F,\fin}} f_v(g_{v};\omega_v),$ where $f_v(\cdot;\omega_v)$ is defined by  \eqref{5.}:
$$
f_v(z_vk_v;\omega_v)={\Vol(\overline{K_{v}[n_v]})^{-1}}\omega_v(z_v)^{-1}\omega_v(E_{2,2}(k_v))^{-1}.
$$

For ${f}^{\dagger}\in\{{f}^{\dagger}(\cdot;i,v_0),\{f^{\dagger}(\cdot;v_1,v_2)\},$ $\varphi\in \sigma,$ where $\sigma\in\mathcal{A}([G],\omega),$ define $\varphi^{\dagger}(x)$ by 
\begin{align*}
\int_{G(\mathbb{A})}{f}^{\dagger}(g)\prod_{v\mid\mathfrak{Q}}\Bigg[\frac{1}{\tau(\chi_v)}\sum_{\alpha\in (\mathcal{O}_v/\varpi_v^{m_v}\mathcal{O}_v)^{\times}}\chi_v(\alpha)\sigma_v\left(\begin{pmatrix}
	1&\alpha \varpi_v^{-m_v}\\
	&1
\end{pmatrix}\right)\Bigg]\sigma(g)\varphi(x)dg.
\end{align*}

Similarly, we define $\varphi^{\ddagger}(x)$ by 
\begin{align*}
\int_{G(\mathbb{A})}{f}^{\ddagger}(g)\prod_{v\mid\mathfrak{Q}}\Bigg[\frac{1}{\tau(\chi_v)}\sum_{\beta\in (\mathcal{O}_v/\varpi_v^{m_v}\mathcal{O}_v)^{\times}}\chi_v(\beta)\sigma_v\left(\begin{pmatrix}
	1&\beta\varpi_v^{-m_v}\\
	&1
\end{pmatrix}\right)\Bigg]\sigma(g)\varphi(x)dg.
\end{align*}

To simplify notations, we shall still write $\Psi(s,\phi^{\ddagger},\chi)$ and $\Psi(s,E(\cdot,\phi,\lambda)^{\ddagger},\chi)$ for their holomorphic continuations, respectively.
\begin{lemma}\label{lem8}
Let notation be as before. Then $\mathcal{J}_{\Spec}^{\heartsuit}(\boldsymbol{\alpha},\chi)$ (cf. \textsection\ref{3.7.3}) converges absolutely. Moreover, 
\begin{equation}\label{14.}
\mathcal{J}_{\Spec}^{\heartsuit}(\boldsymbol{\alpha},\chi)=\mathcal{J}_{0}^{\heartsuit}(\boldsymbol{\alpha},\chi)+\mathcal{J}_{\Eis}^{\heartsuit}(\boldsymbol{\alpha},\chi),
\end{equation}
where $s_0$ is defined by \eqref{eq2.1} in \textsection\ref{2.1.5.}, and 
\begin{align*}
\mathcal{J}_{0}^{\heartsuit}(\boldsymbol{\alpha},\chi):=&\sum_{\sigma\in\mathcal{A}_0([G],\omega)}\sum_{\phi\in\mathfrak{B}_{\sigma}}\Big|\sum_{\mathfrak{m}\in\mathcal{L}}\alpha_{\mathfrak{m}}\mu_{\pi}(\mathfrak{m})\Big|^2\cdot\big|\Psi(s_0,\phi^{\ddagger},\chi)\big|^2,\\
\mathcal{J}_{\Eis}^{\heartsuit}(\boldsymbol{\alpha},\chi):=&\frac{1}{4\pi}\sum_{\eta}\int_{i\mathbb{R}}\sum_{\phi\in \mathfrak{B}_{\sigma_{0,\eta}}}\Big|\sum_{\mathfrak{m}\in\mathcal{L}}\alpha_{\mathfrak{m}}\mu_{\sigma_{\lambda,\eta}}(\mathfrak{m})\Big|^2\cdot\big|\Psi(s_0,E(\cdot,\phi,\lambda)^{\ddagger},\chi)\big|^2d\lambda,
\end{align*}
with $\eta\in \widehat{F^{\times}\backslash\mathbb{A}_F^{(1)}}.$ Here $\mu_{\pi}(\mathfrak{m})$ and $\mu_{\sigma_{\lambda,\eta}}(\mathfrak{m})$ are the $\mathfrak{m}$-th Hecke trace of $\pi$ and $\sigma_{\lambda,\eta},$ respectively (cf. \ref{2.1.4}).  
\end{lemma}

\begin{proof}
Let $\nu(f)$ be defined by \eqref{61} in \textsection\ref{3.6.2}. By construction, we have 
\begin{align*}
\varphi^{\dagger}(x)=\mu_{\sigma,f}\varphi^{\ddagger}(x),
\end{align*}
where $\mu_{\sigma,f}$ is a scalar determined by $f_v$ and $\sigma_v$ at $v\mid\nu(f)$ with the property that 
\begin{equation}\label{13.}  
\sum_{\mathfrak{m}_1,\mathfrak{m}_2\in\mathcal{L}}\alpha_{\mathfrak{m}_1}\overline{\alpha_{\mathfrak{m}_2}}\prod_{v_0\mid\gcd(\mathfrak{m}_1,\mathfrak{m}_2)}\sum_{i_{v_0}=0}^{1}c_{v_0,i_{v_0}}\mu_{\sigma,f(g;\mathbf{v},\mathbf{i})}=\Big|\sum_{\mathfrak{m}\in\mathcal{L}}\alpha_{\mathfrak{m}}\mu_{\sigma}(\mathfrak{m})\Big|^2,
\end{equation}
where $\mathbf{v}$ and $\mathbf{i}$ are determined by $\mathfrak{m}_1$ and $\mathfrak{m}_2$ as in \textsection\ref{3.7.3}. 

Moreover, by \eqref{3.} and \eqref{6.}, i.e., for $v\mid\infty,$ $f_v=(f_v(\cdot,\chi_v)*f_v(\cdot,\chi_v)^*)(g_v),$ the spectral expansion of the kernel functions becomes
\begin{align*}
&\K_0(x,y)=\sum_{\sigma\in \mathcal{A}_0([G],\omega)}\sum_{\phi\in\mathfrak{B}_{\sigma}}\mu_{\sigma,f}\phi^{\ddagger}(x)\overline{\phi^{\ddagger}(y)},\\
&\K_{\ER}(x,y)=\frac{1}{4\pi}\sum_{\eta\in \widehat{F^{\times}\backslash\mathbb{A}_F^{(1)}}}\int_{i\mathbb{R}}\sum_{\phi\in \mathfrak{B}_{\sigma_{0,\eta}}}\mu_{\sigma_{\lambda,\eta},f}E(x,\phi,\lambda)^{\ddagger}\overline{E(y,\phi,\lambda)^{\ddagger}}d\lambda.
\end{align*}

By Proposition \ref{prop6}, the spectral expansion (cf. \textsection\ref{sec3.4}) becomes 
\begin{align*}
J_{0}^{\Reg}(f,\textbf{s}_0,\chi)=&\sum_{\sigma\in\mathcal{A}_0([G],\omega)}\sum_{\phi\in\mathfrak{B}_{\sigma}}\mu_{\sigma,f}\big|\Psi(s_0,\phi^{\ddagger},\chi)\big|^2,\\
J_{\Eis}^{\Reg}(f,\textbf{s}_0,\chi)=&\frac{1}{4\pi}\sum_{\eta\in \widehat{F^{\times}\backslash\mathbb{A}_F^{(1)}}} \int_{i\mathbb{R}}\sum_{\phi\in \mathfrak{B}_{\sigma_{0,\eta}}}\mu_{\sigma_{\lambda,\eta},f}\big|\Psi(s_0,E(\cdot,\phi,\lambda)^{\ddagger},\chi)\big|^2d\lambda,
\end{align*}
which converges absolutely. Then \eqref{14.} follows from  \eqref{13.}. 
\end{proof}
\subsubsection{Local calculations}
In this section we compute the period integrals $\Psi(s_0,\phi^{\ddagger},\chi)$ and $\Psi(s_0,E(\cdot,\phi,\lambda)^{\ddagger},\chi).$ 

\begin{lemma}\label{lem9}
Let notation be as before. Let $\sigma\in\mathcal{A}_0([G],\omega).$ Let $\phi\in\sigma$ be a pure tensor. Suppose that $\phi^{\ddagger}\neq 0.$ Then for $v\in\Sigma_{F,\fin},$ we have 
\begin{equation}\label{15..}
\Psi_v(s,\phi^{\ddagger},\chi)=W_{\phi,v}(I_2)L_v(s+1/2,\sigma_v\times\chi_v),\ \ \Re(s)\geq 0. 
\end{equation}
\end{lemma}
\begin{proof}
Since $\phi^{\ddagger}\neq 0,$ the cusp form $\phi$ must be  right invariant under the compact subgroup $\otimes_{v<\infty}K_v[n_v],$ with $n_v:=\max\{r_{\pi_v},r_{\chi_v}+r_{\omega_v}\}$ as before (cf. \textsection\ref{2.1.5}).

When $v\nmid\mathfrak{Q},$ $\chi_v$ is unramified. By Casselman-Shalika formula (at $r_{\pi_v}=0$) and \cite{JPSS81} (at $r_{\pi_v}\geq 1$) we have \eqref{15..}.  

Now we suppose $v\mid\mathfrak{Q},$ i.e., $m_v=r_{\chi_v}\geq 1.$ Then $\Psi_v(s,\phi^{\ddagger},\chi)$ is defined by 
\begin{equation}\label{17.}
\frac{1}{\tau(\chi_v)}\sum_{\alpha}\chi_v(\alpha)\int_{F_v^{\times}}W_{\phi,v}\left(\begin{pmatrix}
	x_v\\
	&1
\end{pmatrix}\begin{pmatrix}
	1&\alpha \varpi_v^{-m_v}\\
	&1
\end{pmatrix}\right)\chi_v(x_v)|x_v|_v^sd^{\times}x_v,
\end{equation}
where $\alpha\in (\mathcal{O}_v/\varpi_v^{m_v}\mathcal{O}_v)^{\times}.$ Note that 
\begin{equation}\label{16.}
W_{\phi,v}\left(\begin{pmatrix}
	x_v\\
	&1
\end{pmatrix}\begin{pmatrix}
	1&\alpha \varpi_v^{-m_v}\\
	&1
\end{pmatrix}\right)=\psi_v(\alpha x_v\varpi_v^{-m_v})W_{\phi,v}\left(\begin{pmatrix}
	x_v\\
	&1
\end{pmatrix}\right).
\end{equation}
Substituting \eqref{16.} into \eqref{17.} we then derive that 
\begin{align*}
\Psi_v(s,\phi^{\ddagger},\chi)=\frac{1}{\tau(\chi_v)}\sum_{\alpha}\int_{F_v^{\times}}W_{\phi,v}\left(\begin{pmatrix}
	x_v\\
	&1
\end{pmatrix}\right)\chi_v(\alpha x_v)\psi_v(\alpha x_v\varpi_v^{-m_v})|x_v|_v^sd^{\times}x_v.
\end{align*}

Note that $W_{\phi,v}\left(\begin{pmatrix}
	x_v\\
	&1
\end{pmatrix}\right)=0$ unless $e_v(x_v)\geq 0.$ Moreover, 
\begin{equation}\label{21.}
\sum_{\alpha\in (\mathcal{O}_v/\varpi_v^{m_v}\mathcal{O}_v)^{\times}}\chi_v(\alpha x_v)\psi_v(\alpha x_v\varpi_v^{-m_v})=\begin{cases}
	\tau(\chi_v),\ &\text{if $e_v(x_v)=0,$}\\
	0,\ &\text{if $e_v(x_v)\geq 1$.}
\end{cases}
\end{equation}

Therefore, we have $\Psi_v(s,\phi^{\ddagger},\chi)=W_{\phi,v}(I_2)=W_{\phi,v}(I_2)L_v(s+1/2,\sigma_v\times\chi_v),$ which implies \eqref{15.} as $L_v(s+1/2,\sigma_v\times\chi_v)\equiv 1.$  
\end{proof}

\begin{lemma}\label{lem10.}
Let notation be as before. Let $\phi\in\sigma_{\lambda,\eta}$ be a pure tensor. Let $\varphi=E(\cdot,\phi,\lambda).$ Suppose that $\varphi^{\ddagger}\neq 0.$ Then for $v\in\Sigma_{F,\fin},$ $\Re(s)\geq 0,$ we have 
\begin{equation}\label{15.}
\Psi_v(s,\varphi^{\ddagger},\chi)=W_{\varphi,v}(I_2)L_v(s+1/2+\lambda,\eta_v\chi_v)L_v(s+1/2-\lambda,\eta_v^{-1}\chi_v\omega_v). 
\end{equation}
\end{lemma}
\begin{proof}
The proof is similar to that of Lemma \ref{lem9}, so we omit it.
\end{proof}


\subsection{Spectral Side: the lower bound}\label{sec3.6}
In this section, our goal is to prove Theorem \ref{thm6}. To do so, we will analyze the structure of $\pi$ and consider the three cases separately in \textsection\ref{3.6.1}--\textsection\ref{3.6.3}. 

\subsubsection{Stability of deformation}\label{stab}

Suppose $\pi=\eta\boxplus \omega\eta^{-1}$ is an Eisenstein series induced from a character $\eta$. In the continuous spectrum, the associated Rankin-Selberg period behaves approximately as follows:
$$
\int_{\mathbb{R}}\frac{|L(1/2+it,\eta\chi)|^2|L(1/2+it,\eta^{-1}\omega\chi)|^2}{|L(1+2it,\eta^2\omega^{-1})|^2}dt. 
$$ 

However, if $\eta^2=\omega$, then $L(1+2it,\eta^2\omega^{-1})$ approaches $0$ as $t$ approaches $0$. Therefore, for this particular type of $\pi$, we do not obtain $|L(1/2,\eta\chi)|^4$ from the spectral side. In this case, we can consider the deformation $|L(1/2+is_0,\pi\times\chi)|^2$ with the following properties:

\begin{itemize}
	\item $|s_0|$ is not too small, or else the denominator $|L(1+2is_0,\eta^2\omega^{-1})|^2=|\zeta(1+2is_0)|^2$ would be too large.
	\item $|L(1/2+is_0,\pi\times\chi)|$ has a similar magnitude to $|L(1/2,\pi\times\chi)|$.
\end{itemize}

We construct such an $s_0$ in the following lemma, which we will use in \textsection\ref{3.6.1} through \textsection\ref{3.6.3}. 
\begin{lemma}\label{lem3.6}
Let notation be as before. Let $\sigma$ be an automorphic representation of $\mathrm{GL}(2)/F.$ Suppose that $|L(1/2,\sigma\times\chi)|\geq 2.$ Let $\textbf{C}:=C(\sigma\times\chi)$ be the analytic conductor of $\sigma\boxtimes\chi.$ Then there exists some $d'\in [2^{-1}\exp(-3\sqrt{\log \textbf{C}}),\exp(-3\sqrt{\log \textbf{C}})]$ such that for all $s\in\mathbb{C}$ with $|s-1/2|=d',$ we have
\begin{equation}\label{11.1}
|L(1/2,\sigma\times\chi)|\ll \exp(\log^{3/4}\textbf{C})\cdot |L(s,\sigma\times\chi)|,
\end{equation} 
where the implied constant depends only on $F.$
\end{lemma}
\begin{proof}
Denote by $\mathcal{D}:=\big\{\rho\in\mathbb{C}:\ L(\rho,\sigma\times\chi)=0,\ |\rho-1/2|\leq \exp(-2\sqrt{\log\textbf{C}})\big\}.$ By Jensen's formula and the convexity bound for $L(s,\sigma\times\chi),$ we have 
 	\begin{align*}
 		\#\mathcal{D}\cdot \log (\exp(\sqrt{\log\textbf{C}}))\leq \max_{|s-1/2|\leq \exp(-\sqrt{\log\textbf{C}})}\log \frac{|L(s,\sigma\times\chi)|}{|L(1/2,\sigma\times\chi)|}\ll \log\textbf{C}, 	\end{align*} 	
 	where the implied constant depends only on $F$. Consequently, 
 	$$
 	\#\mathcal{D}\ll\sqrt{\log\textbf{C}},
 	$$ 
 	where the implied constant relies on $F.$ Dividing $\mathcal{D}$ into disjoint annuli, by Pigeonhole principle there exists some $d\in [2^{-1}\exp(-2\sqrt{\log\textbf{C}}),\exp(-2\sqrt{\log\textbf{C}})]$ such that for all $|s-1/2|=d$ and for all $\rho\in\mathcal{D},$ one has 
 	\begin{equation}\label{11.3.}
 	|s-\rho|\geq \frac{\exp(-2\sqrt{\log\textbf{C}})}{6\cdot (1+\#\mathcal{D})}.
 	\end{equation}
 	
Likewise, there exists some $d'\in [2^{-1}\exp(-3\sqrt{\log\textbf{C}}),\exp(-3\sqrt{\log\textbf{C}})]$ such that for all $|s-1/2|=d'$ and for all $\rho\in\mathcal{D},$ we have 
\begin{equation}\label{3.17}
 	|s-\rho|\geq \frac{\exp(-3\sqrt{\log\textbf{C}})}{6\cdot (1+\#\mathcal{D})},
 \end{equation}
 	
 	 Let $\mathcal{B}:=\{s:\ |s-1/2|\leq d\}$ and $\mathcal{B}':=\{s:\ |s-1/2|\leq d'\}.$

 	Set $
 		\mathbf{P}(s):=\prod_{\rho\in \mathcal{D}}(s-\rho).$ 	Then $\mathbf{P}$ is a polynomial of degree $\deg\mathbf{P}=\#\mathcal{D}.$ Define
 	$$
 	\mathbb{L}(s):=\log\frac{L(s,\sigma\times\chi)}{\mathbf{P}(s)},\ \ s\in\mathcal{B}_l.
 	$$
 	Then $\mathbb{L}(s)$ is holomorphic in $\mathcal{B}_l.$ In particular, it is holomorphic in $\mathcal{B}_l'.$ Therefore, by Borel-Carath\'{e}odory theorem, for all $s\in\mathcal{B}_l',$
 \begin{align*}
 		\big|\mathbb{L}(1/2)-\mathbb{L}(s)\big|\leq \frac{10\exp(-3\sqrt{\log\textbf{C}})\cdot \big[\max_{s'\in \mathcal{B}_l}\Re(\mathbb{L}(s'))-\Re(\mathbb{L}(1/2))\big]}{\exp(-2\sqrt{\log\textbf{C}})-\exp(-3\sqrt{\log\textbf{C}})}.
 		\end{align*}
 	
 	Let $s_*\in\partial \mathcal{B}_l$ be such that $\Re(\mathbb{L}(s_*))=\max_{s\in \mathcal{B}_l}\Re(\mathbb{L}(s)),$ where $\partial\mathcal{B}_l$ is the boundary of $\mathcal{B}_l$.  Then 
 	\begin{align*}
 	\max_{s'\in \mathcal{B}_l}\Re(\mathbb{L}(s'))-\Re(\mathbb{L}(1/2))=\log\Big|\frac{L(s_*,\sigma\times\chi)}{L(1/2,\sigma\times\chi)}\Big|+\log \Big|\frac{\mathbf{P}(1/2)}{\mathbf{P}(s_*)}\Big|.	
 	\end{align*}
 	
 	Using the trivial bound $|\mathbf{P}(1/2)|\leq 1$ and \eqref{11.3.}, 
 	\begin{align*}
 		\Big|\frac{\mathbf{P}(1/2)}{\mathbf{P}(s_*)}\Big|\ll \frac{1}{|\mathbf{P}(s_*)|}\ll \Big[6\cdot (1+\#\mathcal{D})\exp(2\sqrt{\log\textbf{C}})\Big]^{\deg \mathbf{P}}.
 	\end{align*}
 	Hence, we have, by $\#\mathcal{D}\ll \sqrt{\log\textbf{C}},$ that
 	$$
 	\log \Big|\frac{\mathbf{P}(1/2)}{\mathbf{P}(s_*)}\Big|\ll \deg \mathbf{P}\cdot 
 	\log \Big[6\cdot (1+\#\mathcal{D})\exp(2\sqrt{\log\textbf{C}})\Big]\ll \log\textbf{C},
 	$$
 	where the implied constant is absolute. Moreover, by convexity bound, 
 	$$
 	\log\Big|\frac{L(s_*,\sigma\times\chi)}{L(1/2,\sigma\times\chi)}\Big|\leq \log\Big|\frac{L(s_*,\sigma\times\chi)}{2}\Big|\ll \log |L(s_*,\sigma\times\chi)|\ll  \log\textbf{C},
 	$$
where the implied constant depends on $F.$ Therefore, 
 	\begin{align*}
 		\big|\mathbb{L}(1/2)-\mathbb{L}(s)\big|\ll \exp(-\sqrt{\log\textbf{C}})\cdot \log\textbf{C}\ll 1,
 		\end{align*}
 		which leads to 
 		\begin{align*}
 L(1/2,\sigma\times\chi)\ll  \big|\mathbf{P}(1/2)\big|\cdot \Big|\frac{L(s,\sigma\times\chi)}{\mathbf{P}(s)}\Big|,\ \ s\in \mathcal{B}_l'.
 		\end{align*}

By \eqref{3.17} there exists some  absolute constant $C$ such that for all $\rho\in\mathcal{D}$ and $s\in\partial\mathcal{B}_l',$
 	$$
 	\frac{|1/2-\rho|}{|s-\rho|}\leq 
 	\begin{cases}
 		1,&\ \text{if $|s-1/2|\leq d'/2,$}\\
 		C,&\ \text{if $d'/2\leq |s-1/2|\leq 3d'/4,$}\\
 		C\cdot (1+\#\mathcal{D}),&\ \text{if $3d'/4\leq |s-1/2|\leq 5d'/4,$}\\
 		C,&\ \text{if $5d'/4\leq |s-1/2|\leq d.$}
 		\end{cases}
 	$$
Note that 
$\deg\mathbf{P}=\#\mathcal{D}\leq C_2\sqrt{\log\textbf{C}}$ 
for some constant $C_2$ depending at most on $F.$ Hence, we have, uniformly for $s\in\partial\mathcal{B}_l',$ that 
\begin{align*}
	\Big|\frac{\mathbf{P}(1/2)}{\mathbf{P}(s)}\Big|\leq (C\log\textbf{C})^{C_2\sqrt{\log\textbf{C}}}\ll \exp(\log^{3/4}\textbf{C}),
\end{align*}
where the implied constant depends on $F.$ So \eqref{11.1} follows.
\end{proof}
\begin{remark}
Note that $d'$ may depend on $\sigma$ and $\chi.$
\end{remark}


\subsubsection{$\pi$ is cuspidal.}\label{3.6.1}
Suppose that $\pi$ is a cuspidal representation. Denote by $f_v^{\circ}:=\int_{Z(F_v)}\tilde{f}_{v}(zg)\omega_{v}(z)d^{\times}z.$ By Kirillov theorem (cf. e.g., \cite{Nel21}), 
\begin{equation}\label{37.}
T_v^{-1/2-\varepsilon}\ll_{\varepsilon}\int_{F_v^{\times}}(\pi_v(f_v^{\circ})(W_v\otimes \chi_v|\cdot|_v^{s_0}))\left(\begin{pmatrix}
	x_v\\
	&1
\end{pmatrix}
\right)d^{\times}x_v\ll_{\varepsilon} T_v^{-1/2+\varepsilon}
\end{equation}
for some $W_v$ in the Kirillov model of $\pi_v,$ where $T_v\asymp C(\pi_{v}\otimes\chi_v)^{1/2}$ is defined in \textsection\ref{2.1.2}. By definition \eqref{6.} in \textsection \ref{3.2.1}, we have 
\begin{align*}
\pi_v(f_v(\cdot,\chi_v))W_v\left(\begin{pmatrix}
	x_v\\
	&1
\end{pmatrix}
\right)\chi_v(x_v)|x_v|_v^{s_0}=(\pi_v(f_v^{\circ})(W_v\otimes \chi_v|\cdot|_v^{s_0}))\left(\begin{pmatrix}
	x_v\\
	&1
\end{pmatrix}
\right).
\end{align*}

Hence, $\Psi_v(s_0,\pi_v(f_v(\cdot,\chi_v))W_v,\chi_v)\gg_{\varepsilon} C(\pi_{v}\otimes\chi_v)^{-1/4-\varepsilon}$ for some $W_v$ in the Kirillov model of $\pi_v.$ Let $\phi\in\pi$ be a cusp form with Petersson norm $\langle\phi,\phi\rangle=1,$ and Whittaker function $W_{\phi}=\otimes_vW_{\phi,v}$ (defined by \eqref{whi}), such that $W_{\phi,v}=W_v,$ for all $v\mid\infty,$ and $W_{\phi,v}$ is $\prod_{v<\infty}K_v[r_{\chi_v}]$-invariant. Then 
\begin{align*}
\Psi_{v}(s_0,\phi^{\ddagger},\chi)=\Psi_v(s_0,\pi_v(f_v(\cdot,\chi_v))W_v,\chi_v)\gg_{\varepsilon} C(\pi_{v}\otimes\chi_v)^{-1/4-\varepsilon}.
\end{align*}

Together with Lemmas \ref{lem8} and \ref{lem9} we have
\begin{align*}
\mathcal{J}_{0}^{\heartsuit}(\boldsymbol{\alpha},\chi)\gg_{\varepsilon} \Big|\sum_{\mathfrak{m}\in\mathcal{L}}\alpha_{\mathfrak{m}}\mu_{\pi}(\mathfrak{m})\Big|^2\cdot C_{\infty}(\pi\otimes\chi)^{-\frac{1}{4}+\varepsilon}\cdot \big|W_{\fin}(I_2)L(1/2+s_0,\pi\times\chi)\big|^2.
\end{align*}
Therefore, \eqref{16..} follows from  the bound $|W_{\fin}(I_2)|\gg C(\pi)^{-\varepsilon}$ (cf. \cite{HL94}), and Lemma \ref{lem3.6}.

\subsubsection{$\pi$ is an Eisenstein series \RNum{1}}  
Now we consider the case that $\pi=\sigma_{\lambda',\eta'}$ for some $\lambda'\in i\mathbb{R}$ and $\eta'\in \widehat{F^{\times}\backslash\mathbb{A}_F^{(1)}}.$ Similar to \eqref{37.} in  \textsection\ref{3.6.1}, for  $v\mid\infty,$ there exists a Whittaker function $W_v$ in the Kirillov model of $\pi_{v}$ such that 
\begin{align*}
C(\pi_{v}\otimes\chi_v)^{-1/4-\varepsilon}\ll_{\varepsilon} \Psi_{v}(s_0,E(\cdot,\phi,\lambda)^{\ddagger},\chi)\ll_{\varepsilon} C(\pi_{v}\otimes\chi_v)^{-1/4+\varepsilon}.
\end{align*}

Let $\phi\in \mathfrak{B}_{\lambda',\eta'}$ be such that the Eisenstein series $\varphi=E(\cdot,\phi,\lambda')\in \pi$ satisfies that $W_{\varphi,v}=W_v,$ for all $v\mid\infty,$ and $\varphi$ is right $\prod_{v<\infty}K_v[n_v]$-invariant. Here $W_{\varphi,v}$ is the local Whittaker function (cf. \textsection\ref{sec3.3}).   

By Lemmas \ref{lem8} and \ref{lem10.} we have
\begin{equation}\label{28}
\mathcal{J}_{\Eis}^{\heartsuit}(\boldsymbol{\alpha},\chi)\gg \int_{i\mathbb{R}}h(\lambda;\eta')d\lambda,
\end{equation}
where the function $h(\lambda;\eta')$ is defined by
\begin{align*}
\Big|\sum_{v\in\mathcal{L}}\alpha_v\mu_{\sigma_{\lambda,\eta,v}}\Big|^2\cdot\big|\Psi_{\infty}(s_0,E(\cdot,\phi,\lambda)^{\ddagger},\chi)\big|^2\cdot \frac{|L(s_0'+\lambda,\eta'\chi)L(s_0'-\lambda,\eta'^{-1}\omega\chi)|^2}{|L(1+2\lambda,\eta'^2\omega^{-1})|^2},
\end{align*}
where $s_0':=1/2+s_0.$ Here we appeal to the calculation  (cf. e.g., \cite{Sha10}, \textsection 7) 
$$
W_{\varphi,\fin}(I_2)=\prod_{v<\infty}W_{\varphi,v}(I_2)=L(1+2\lambda,\eta'^2\omega^{-1})^{-1}
$$

Suppose $\lambda'\neq 0$ or $\omega\neq \eta'^2.$ Then $L(1+2\lambda,\eta'^2\omega^{-1})^{-1}\neq 0$ for all $\lambda\in i\mathbb{R}.$ Moreover, for $|\lambda-\lambda'|\leq |\lambda'|/2,$ we have $L(1+2\lambda,\eta'^2\omega^{-1})^{-1}\ll_{\varepsilon} C(\eta'^2\omega^{-1})^{\varepsilon}.$ In conjunction with the convex bound we obtain by \cite{Ram95} ( the inequality (1.2.3) in \textsection 1.2) that  
\begin{align*}
\int_{i\mathbb{R}}h(\lambda;\eta')d\lambda\gg h(\lambda';\eta')C(\eta'\chi)^{-\varepsilon}C(\eta'^{-1}\omega\chi)^{-\varepsilon},
\end{align*}
where the implied constant relies only on $F$, $\varepsilon,$ and $c_v, C_v,$ for all $v\mid\infty.$ So $\mathcal{J}_{\Eis}^{\heartsuit}(\boldsymbol{\alpha},\chi)$ is 
\begin{equation}\label{29} 
\gg \Big|\sum_{\mathfrak{m}\in\mathcal{L}}\alpha_{\mathfrak{m}}\mu_{\pi}(\mathfrak{m})\Big|^2\cdot (MQ)^{-\varepsilon}C_{\infty}(\pi\otimes\chi)^{-\frac{1}{4}+\varepsilon}\cdot \big|L(1/2+s_0,\pi\times\chi)\big|^2.
\end{equation}
Therefore, \eqref{16..} follows from \eqref{11.1} and \eqref{29}.
  
\subsubsection{$\pi$ is an Eisenstein series \RNum{2}}\label{3.6.3}
Suppose $\pi=\sigma_{\lambda',\eta'}$ for $\lambda'=0$ and $\eta'\in \widehat{F^{\times}\backslash\mathbb{A}_F^{(1)}}$ with $\omega=\eta'^2.$  Then $L(s,\pi\times\chi)=L(s,\eta'\chi)^2,$ $s\in\mathbb{C}.$ 

By assumption we have $|L(1/2,\eta'\chi)|=\sqrt{|L(1/2,\pi\times\chi)|}\geq 1.$ Then by \eqref{11.1},
\begin{equation}\label{30}
L(1/2,\eta'\chi)\ll \exp(\log^{3/4}C(\eta'\chi))\cdot |L(1/2+is_0,\eta'\chi)|,
\end{equation} 
where we recal $s_0\asymp  \exp(-2\sqrt{\log C(\eta'\chi)})$ is defined by Lemma \ref{lem3.6}.

Let $\pi^{\dagger}=\sigma_{is_0,\eta'}.$ Similar to the arguments that yields  \eqref{29} we have
\begin{align*}
\mathcal{J}_{\Eis}^{\heartsuit}(\boldsymbol{\alpha},\chi)\gg  (MQ)^{-\varepsilon}C_{\infty}(\pi\otimes\chi)^{-\frac{1}{4}+\varepsilon}\big|L(1/2+s_0+is_0,\pi^{\dagger}\times\chi)\big|^2\Big|\sum_{\mathfrak{m}\in\mathcal{L}}\alpha_{\mathfrak{m}}\mu_{\pi^{\dagger}}(\mathfrak{m})\Big|^2,
\end{align*}
where the implied constant relies only on $F$, $\varepsilon,$ and $c_v, C_v,$ for all $v\mid\infty.$ Together with \eqref{30}, we derive that $\mathcal{J}_{\Eis}^{\heartsuit}(\boldsymbol{\alpha},\chi)$ is 
\begin{equation}\label{33}
\gg\Big|\sum_{\mathfrak{m}\in\mathcal{L}}\alpha_{\mathfrak{m}}\mu_{\pi^{\dagger}}(\mathfrak{m})\Big|^2\cdot (MQ)^{-\varepsilon}C_{\infty}(\pi\otimes\chi)^{-\frac{1}{4}+\varepsilon}\cdot \big|L(1/2+s_0+is_0,\pi\times\chi)\big|^2.
\end{equation}
Therefore, \eqref{17..} follows from \eqref{33} and Lemma \ref{lem3.6}.

\section{The Geometric Side: the Orbital Integral $J^{\Reg}_{\Geo,\sm}(f,\textbf{s},\chi)$}\label{sec4}
Let $f=f(g;\mathbf{v},\mathbf{i})$ be constructed in  \textsection\ref{3.2.5}. By definition we have
\begin{align*}
J^{\Reg}_{\Geo,\sm}(f,\textbf{s},\chi)=\int_{\mathbb{A}_F^{\times}}\int_{\mathbb{A}_F^{\times}}f\left(\begin{pmatrix}
	y&b\\
	&1
\end{pmatrix}\right)\psi(xb)\overline{\chi}(y)|x|^{1+s_1+s_2}|y|^{s_2}d^{\times}xd^{\times}y,
\end{align*}
which is a Tate integral representing $\Lambda(1+s_1+s_2,\textbf{1}_F).$ Hence, $J^{\Reg}_{\Geo,\sm}(f,\textbf{s}_0,\chi)$ converges, where $s_0\in [4^{-1}\exp(-3\sqrt{\log C(\pi\times\chi)}),\exp(-3\sqrt{\log C(\pi\times\chi)})]$ is the deformation parameter defined by \eqref{eq2.1} in \textsection\ref{2.1.5.}, and $\mathbf{s}_0=(s_0,s_0).$

\begin{prop}\label{prop12}
Let notation be as before. Then 
\begin{align*}
J^{\Reg}_{\Geo,\sm}(f,\textbf{s}_0,\chi)\ll_{\varepsilon} \mathcal{N}_f^{-1}[M,M'Q]^{1+\varepsilon}C_{\infty}(\pi\otimes\chi)^{\frac{1}{4}+\varepsilon},
\end{align*} 
where $\mathcal{N}_f$ is defined by \eqref{61}, and the implied constant depends only on $F,$ $\varepsilon,$ and $c_v,$ $C_v$ at $v\mid\infty,$ cf. \textsection \ref{2.1.2}.
\end{prop}

By definition, we have
\begin{align*}
J^{\Reg}_{\Geo,\sm}(f,\textbf{s}_0,\chi):=&\prod_{v\in\Sigma_{F}}\int_{F_v^{\times}}\mathcal{I}_v(x_v)d^{\times}x_v,
\end{align*}
where we define by 
\begin{equation}\label{34.}
\mathcal{I}_v(x_v):=|x_v|_v^{1+2s_0}\int_{F_v^{\times}}\int_{F_v}f_v\left(\begin{pmatrix}
	y_v&b_v\\
	&1
\end{pmatrix}\right)\psi_v(x_vb_v)\overline{\chi}_v(y_v)|y|_v^{s_0}db_vd^{\times}y_v.
\end{equation}

\subsection{Local Calculation at $v\in\Sigma_{F,\fin}-\{v:\ v\mid \nu(f)\}$}\label{4.1.1}
\begin{lemma}\label{lem13}
Let $v\in\Sigma_{F,\fin}-\{v:\ v\mid \nu(f)\}.$ Let $n_v=\max\{r_{\pi_v},r_{\chi_v}+r_{\omega_v}\}$. Then 
\begin{equation}\label{19}
\mathcal{I}_v(x_v)=|x_v|_v^{1+2s_0}\Vol(K_v[n_v])^{-1}N_{F_v}(\mathfrak{D}_{F_v})^{-1/2}\textbf{1}_{\mathfrak{D}_{F_v}}(x_v).\end{equation}
\end{lemma}
\begin{proof}
By definition, we have $ f_v\left(\begin{pmatrix}
	y&b\\
	&1 
\end{pmatrix}\right)\neq 0$ iff there exists $z\in F_v^{\times}$ such that 
$$
z\begin{pmatrix}
	1&\alpha \varpi_v^{-m_v}\\
	&1
\end{pmatrix}\begin{pmatrix}
	y&b\\
	&1
\end{pmatrix}\begin{pmatrix}
	1&\beta \varpi_v^{-m_v}\\
	&1
\end{pmatrix}\in K_v[n_v],
$$ 
which implies that $z, y\in \mathcal{O}_v^{\times}.$ Changing variable $\beta\mapsto y^{-1}\beta,$ $\mathcal{I}_v(x_v)$ becomes
\begin{align*}
\frac{|x_v|_v^{1+2s_0}|\tau(\chi_v)|^{-2}}{\Vol(K_v[n_v])}\sum_{\alpha,\beta}\int \textbf{1}_{K_v[n_v]}\left(\begin{pmatrix}
	1&b+(\alpha +\beta) \varpi_v^{-m_v}\\
	&1
\end{pmatrix}\right)\psi_v(x_vb)\chi_v(\alpha)\overline{\chi}_v(\beta)db,
\end{align*}
where $b\in F_v,$ $\alpha,\beta\in (\mathcal{O}_v/\varpi_v^{m_v}\mathcal{O}_v)^{\times}$ with $m_v=r_{\chi_v}.$ After a further change variables $b\mapsto b-\alpha \varpi_v^{-m_v}-\beta \varpi_v^{-m_v},$  the above integral is equal to 
\begin{align*}
&\frac{|x_v|_v^{1+2s_0}|\tau(\chi_v)|^{-2}}{\Vol(K_v[n_v])}\sum_{\alpha,\beta}\psi_v(-x_v(\alpha \varpi_v^{-m_v}+\beta \varpi_v^{-m_v}))\chi_v(\alpha)\overline{\chi}_v(\beta)\int_{\mathcal{O}_v}\psi_v(x_vb)db.
\end{align*}

Since $\int_{\mathcal{O}_v}\psi_v(x_vb)db=N_{F_v}(\mathfrak{D}_{F_v})^{-1/2}\textbf{1}_{\mathfrak{D}_{F_v}^{-1}}(x_v),$ we then obtain that
\begin{align*}
\mathcal{I}_v(x_v)=&\frac{|x_v|_v^{1+2s_0}\Vol(K_v[n_v])^{-1}}{|\tau(\chi_v)|^2N_{F_v}(\mathfrak{D}_{F_v})^{1/2}}\textbf{1}_{\mathfrak{D}_{F_v}^{-1}}(x_v)\Big|\sum_{\alpha\in (\mathcal{O}_v/\varpi_v^{m_v}\mathcal{O}_v)^{\times}}\psi_v(x_v\alpha \varpi_v^{-m_v})\chi_v(\alpha)\Big|^2.
\end{align*}

Hence, \eqref{19} follows from  properties of Gauss sums (i.e., \eqref{21.}).
\end{proof}

\subsection{Local Calculation at $v\mid \nu(f)$}\label{4.1.2}
Note that $\mathcal{I}_v(x_v)=\mathcal{I}_v(\varpi_v^{r_1}),$ where $r_1=e_v(x_v).$ Let $e_v(y)=r_2,$ and $b=\varpi_v^m\gamma,$ $\gamma\in \mathcal{O}_v^{\times}.$ Since $f_v$ is bi-$K_v$-invariant, we have 
\begin{align*}
\mathcal{I}_v(x_v)=q_v^{-r_1^{\dagger}}\textbf{1}_{r_1\geq 0}\sum_{r_2\in\mathbb{Z}}\overline{\chi}_v(\varpi_v^{r_2})\sum_{m\in\mathbb{Z}}q_v^{-m}f_v\left(\begin{pmatrix}
	\varpi_v^{r_2}&\varpi_v^m\\
	&1
\end{pmatrix}\right)\int_{\mathcal{O}_v^{\times}}\psi_v(\varpi_v^{r_1+m}\gamma)d^{\times}\gamma,
\end{align*} 
where $r_1^{\dagger}:=r_1(1+2s_0).$

\begin{lemma}\label{lem14}
Let notation be as before. Let $r_1^{\dagger}:=r_1(1+2s_0).$ Then 
\begin{align*}
\mathcal{I}_v(x_v)=&c_{v,1}q_v^{-r_1^{\dagger}}\textbf{1}_{r_1\geq 1}-c_{v,1}\zeta_{F_v}(1)q_v^{-1}\textbf{1}_{r_1=0}+q_v^{-r_1^{\dagger}-1}\zeta_{F_v}(1)\overline{\chi}_v(\varpi_v^{2})\textbf{1}_{r_1\geq 0}\\
&\qquad +c_{v,1}q_v^{-r_1^{\dagger}+1}\zeta_{F_v}(1)\overline{\chi}_v(\varpi_v^{-2})\textbf{1}_{r_1\geq 2}
\end{align*}
if $v\in\textbf{v}_0;$ and if $v\in\textbf{v}_1\sqcup \textbf{v}_2,$
\begin{align*}
\mathcal{I}_v(x_v)=q_v^{-r_1^{\dagger}-1/2}\zeta_{F_v}(1)\overline{\chi}_v(\varpi_v)\textbf{1}_{r_1\geq 0}+q_v^{1/2-r_1^{\dagger}}\chi_v(\varpi_v)\zeta_{F_v}(1)\textbf{1}_{r_1\geq 1}.
\end{align*}
\end{lemma}
\begin{proof}
Since $f=f(g;\mathbf{v},\mathbf{i}),$ we shall consider the following cases. 
\begin{itemize}
	\item Suppose $v\in \textbf{v}_0.$ If $i_v=0,$ then $\supp f_v=Z(F_v)K_v,$ in which case we have $\mathcal{I}_v(x)=|x_v|_v\Vol(K_v[n_v])^{-1}\textbf{1}_{\mathcal{O}_v-\{0\}}(x)$ as in Lemma \ref{lem13}. Now we assume that $i_v=1.$ Note that 
\begin{align*}
z\begin{pmatrix}
	\varpi_v^{r_2}&\varpi_v^m\\
	&1
\end{pmatrix}\in K_v\begin{pmatrix}
	\varpi_v\\
	&\varpi_v^{-1}
\end{pmatrix}K_v
\end{align*} 
for some $z\in F_v^{\times}$ if and only if $r_2=0,$ $m=-1,$ or $r_2=2,$ $m\geq 0,$ or $r_2=-2,$ $m\geq -2.$ 

\begin{itemize}
	\item Suppose $r_2=0.$ The contribution to $\mathcal{I}_v(x_v)$ is  
\begin{align*}
q_v^{1-r_1^{\dagger}}\textbf{1}_{r_1\geq 0}f_v\left(\begin{pmatrix}
	1&\varpi_v^{-1}\\
	&1
\end{pmatrix}\right)\int_{\mathcal{O}_v^{\times}}\psi_v(\varpi_v^{r_1-1}\gamma)d^{\times}\gamma=\begin{cases}
	c_{v,1}q_v^{-r_1^{\dagger}}, &\text{if $r_1\geq 1,$}\\
	-c_{v,1}\zeta_{F_v}(1)q_v^{-1},& \text{if $r_1=0.$}
\end{cases}
\end{align*}
\item Suppose $r_2=2.$ The contribution to $\mathcal{I}_v(x_v)$ is 
\begin{align*}
&q_v^{-r_1^{\dagger}}\textbf{1}_{r_1\geq 0}\overline{\chi}_v(\varpi_v^{2})q_v^{-2s_0}\sum_{m\geq 0}q_v^{-m}f_v\left(\begin{pmatrix}
	\varpi_v^{2}&\varpi_v^m\\
	&1
\end{pmatrix}\right)\int_{\mathcal{O}_v^{\times}}\psi_v(\varpi_v^{r_1+m}\gamma)d^{\times}\gamma\\
=&q_v^{-r_1^{\dagger}-2s_0-1}\textbf{1}_{r_1\geq 0}\overline{\chi}_v(\varpi_v^{2})\sum_{m\geq 0}q_v^{-m}\int_{\mathcal{O}_v^{\times}}\psi_v(\varpi_v^{r_1+m}\gamma)d^{\times}\gamma,
\end{align*}
which is equal to $q_v^{-r_1^{\dagger}-2s_0-1}\zeta_{F_v}(1)\textbf{1}_{r_1\geq 0}\overline{\chi}_v(\varpi_v^{2}).$
\item Suppose $r_2=-2.$ The contribution to $\mathcal{I}_v(x_v)$ is 
\begin{align*}
&q_v^{-r_1^{\dagger}+2s_0}\textbf{1}_{r_1\geq 0}\overline{\chi}_v(\varpi_v^{-2})\sum_{\substack{m\geq -1-r_1\\ m\geq -2}}q_v^{-m}f_v\left(\begin{pmatrix}
	\varpi_v^{-2}&\varpi_v^m\\
	&1
\end{pmatrix}\right)\int_{\mathcal{O}_v^{\times}}\psi_v(\varpi_v^{r_1+m}\gamma)d^{\times}\gamma\\
=&q_v^{-r_1^{\dagger}+2s_0-1}\textbf{1}_{r_1\geq 0}\overline{\chi}_v(\varpi_v^{-2})\sum_{m\geq \max\{-1-r_1,-2\}}q_v^{-m}\int_{\mathcal{O}_v^{\times}}\psi_v(\varpi_v^{r_1+m}\gamma)d^{\times}\gamma\\
=&c_{v,1}q_v^{-r_1^{\dagger}+2s_0+1}\zeta_{F_v}(1)\overline{\chi}_v(\varpi_v^{-2})\textbf{1}_{r_1\geq 2}.
\end{align*}
\end{itemize}

\item Suppose $v\in\textbf{v}_1\sqcup\textbf{v}_2.$
Without loss of generality, we consdier the case that $v\in\textbf{v}_1.$ Note that 
\begin{align*}
z\begin{pmatrix}
	\varpi_v^{r_2}&\varpi_v^m\\
	&1
\end{pmatrix}\in K_v\begin{pmatrix}
	\varpi_v\\
	&1
\end{pmatrix}K_v
\end{align*} 
for some $z\in F_v^{\times}$ if and only if $r_2=1,$ $m\geq 0,$ or $r_2=-1,$ $m\geq -1.$ 
\begin{itemize}
	\item Suppose $r_2=1.$ The contribution to $\mathcal{I}_v(x_v)$ is 
	\begin{align*}
		q_v^{-r_1^{\dagger}-1/2}\overline{\chi}_v(\varpi_v)\textbf{1}_{r_1\geq 0}\sum_{m\geq 0}q_v^{-m}=q_v^{-r_1^{\dagger}-1/2}\zeta_{F_v}(1)\overline{\chi}_v(\varpi_v)\textbf{1}_{r_1\geq 0}.
	\end{align*}
	\item Suppose $r_2=-1.$ The contribution to $\mathcal{I}_v(x_v)$ is 
	\begin{align*}
	q_v^{-r_1^{\dagger}+s_0-1/2}\textbf{1}_{r_1\geq 0}\overline{\chi}_v(\varpi_v^{-1})\sum_{m\geq -1}q_v^{-m}\int_{\mathcal{O}_v^{\times}}\psi_v(\varpi_v^{r_1+m}\gamma)d^{\times}\gamma,	
\end{align*}
which is equal to $q_v^{1/2-r_1^{\dagger}+s_0}\chi_v(\varpi_v)\zeta_{F_v}(1)\textbf{1}_{r_1\geq 1}.$
\end{itemize}
\end{itemize}
Therefore, Lemma \ref{lem14} follows.
\end{proof}

\subsection{Local estimates at archimedean places}\label{4.1.3}
Let $v\mid \infty.$ Recall that\begin{align*}
\mathcal{I}_v(x_v)=|x_v|_v^{1+2s_0}\int_{F_v^{\times}}\int_{F_v}f_v\left(\begin{pmatrix}
	y_v&b_v\\
	&1
\end{pmatrix}\right)\psi_v(x_vb_v)\overline{\chi}_v(y_v)|y|_v^{s_0}db_vd^{\times}y_v.
\end{align*}

By the construction of $f$ we have $\mathcal{I}_v(x_v)\neq 0$ unless $x_vT_v^{-1}-\gamma_v\ll_{\varepsilon} T_v^{-1/2+\varepsilon},$ where $\gamma_v$ is determined by $\tau\in\hat{\mathfrak{g}}$, cf. \textsection\ref{3.2.1}. Moreover, by decaying of Fourier transform of $f_v,$ $f_v\left(\begin{pmatrix}
	y_v&b_v\\
	&1
\end{pmatrix}\right)\ll T_v^{-\infty}$ if $|b_v|_v\gg T_v^{-1/2+\varepsilon}.$ Together with \eqref{250},
\begin{equation}\label{45}
\mathcal{I}_v(x_v)\ll |x_v|_v^{1+2s_0}\textbf{1}_{x_vT_v^{-1}-\gamma_v\ll_{\varepsilon} T_v^{-1/2+\varepsilon}}\cdot T_{v}^{1+\varepsilon}\cdot T_v^{-1/2+\varepsilon}\cdot T_v^{-1/2+\varepsilon}+O(T_v^{-\infty}),
\end{equation}
where the factor $T_{v}^{1+\varepsilon}$ comes from the sup-norm estimate (cf. \eqref{250}), the first factor $T_v^{-1/2+\varepsilon}$ comes from the range of $y_v$ according to the support of $f_v$ (cf. \eqref{245}), and the second $T_v^{-1/2+\varepsilon}$ comes from the essential range of $b_v,$ i.e., $|b_v|_v\ll T_v^{-1/2+\varepsilon}.$ In particular, the implied constant in \eqref{45} depends only on $F_v,$ $\varepsilon,$ and $c_v,$ $C_v$ at $v\mid\infty.$

\subsection{Proof of Proposition \ref{prop12}}\label{4.1.4}
Let $V_v:=\Vol(K_v[n_v])$ and $V:=\prod_{v\mid [M,M'Q]}V_v.$ From the local calculations in \textsection\ref{4.1.1}--\textsection\ref{4.1.3}, one can write $J^{\Reg}_{\Geo,\sm}(f,\textbf{s}_0,\chi)$ as 
\begin{align*}
&\int_{\mathbb{A}_F^{\times}} \prod_{v\mid\infty}\mathcal{I}_v(x_v)\prod_{v\mid\nu(f)}\mathcal{I}_v(x_v)\prod_{\substack{v\nmid\mathfrak{Q}\\ v\nmid \nu(f)}}\frac{|x_v|_v^{1+2s_0}\textbf{1}_{\mathcal{O}_v-\{0\}}(x_v)}{V_v}\prod_{v\mid\mathfrak{Q}}V_v^{-1}\textbf{1}_{\mathcal{O}_v^{\times}}(x_v)d^{\times}x.
\end{align*}

Invoking Lemmas \ref{lem13} and  \ref{lem14} with the estimate \eqref{45} we derive that 
\begin{align*}
J^{\Reg}_{\Geo,\sm}(f,\textbf{s}_0,\chi)\ll_k \mathcal{N}_f^{-1}V^{-1}L(1+2s_0,\textbf{1}_F)\prod_{v\mid\infty}\int_{x_vT_v^{-1}-\gamma_v\ll_{\varepsilon} T_v^{-1/2+\varepsilon}} |x_{v}|_{v}^{2s_0}dx_{v},
\end{align*}
which is $\ll s_0^{-1}\mathcal{N}_f^{-1}V^{-1}\prod_{v\mid\infty}T_v^{1/2+\varepsilon}.$ Then Proposition \ref{prop12} follows from the volume calculation:
\begin{align*}
V^{-1}=[M,M'Q]\prod_{v\mid [M,M'Q]}(1+q_v^{-1})\ll_{\varepsilon} [M,M'Q]^{1+\varepsilon}.
\end{align*}

\section{The Geometric Side: the Orbital Integral $J_{\Geo,\du}^{\bi}(f,\textbf{s},\chi)$}\label{sec5}
Let $\mathbf{s}=(s_1,s_2)\in\mathbb{C}^2$ with $\Re(s_1+s_2)>1.$ Recall \eqref{16.} in \textsection\ref{sec3.7}:
\begin{align*}
J_{\Geo,\du}^{\bi}(f,\textbf{s},\chi):=\int_{\mathbb{A}_F^{\times}}\int_{\mathbb{A}_F^{\times}}f\left(\begin{pmatrix}
	1\\
	x&1
\end{pmatrix}\begin{pmatrix}
	y\\
	&1
\end{pmatrix}\right)|x|^{s_1+s_2}|y|^{s_2}\overline{\chi}(y)d^{\times}yd^{\times}x,
\end{align*}
which is a Tate integral representing $\Lambda(s_1+s_2,\textbf{1}_F).$ By Poisson summation, the function $J^{\Reg}_{\Geo,\du}(f,\textbf{s},\chi)$ admits a meromorphic continuation to $\mathbf{s}\in\mathbb{C}^2.$ Explicitly,
\begin{align*}
J^{\Reg}_{\Geo,\du}(f,\textbf{s},\chi)=J^{\Reg,+}_{\Geo,\du}(f,\textbf{s},\chi)+J^{\Reg,\wedge}_{\Geo,\du}(f,\textbf{s},\chi)+J^{\Reg,\Res}_{\Geo,\du}(f,\textbf{s},\chi),
\end{align*}
where 
\begin{align*}
J^{\Reg,+}_{\Geo,\du}(f,\textbf{s},\chi):=&\int_{|x|\geq 1}\int_{\mathbb{A}_F^{\times}}f\left(\begin{pmatrix}
	1\\
	x&1
\end{pmatrix}\begin{pmatrix}
	y\\
	&1
\end{pmatrix}\right)|x|^{s_1+s_2}|y|^{s_2}\overline{\chi}(y)d^{\times}yd^{\times}x;
\end{align*}
the function $J^{\Reg,\wedge}_{\Geo,\du}(f,\textbf{s},\chi)$ is defined by 
\begin{align*}
\int_{|x|\geq 1}\int_{\mathbb{A}_F^{\times}}\int_{\mathbb{A}_F}f\left(\begin{pmatrix}
	1\\
	b&1
\end{pmatrix}\begin{pmatrix}
	y\\
	&1
\end{pmatrix}\right)\psi(bx)|x|^{1-s_1-s_2}|y|^{s_2}\overline{\chi}(y)dbd^{\times}yd^{\times}x;
\end{align*}
and $J^{\Reg,\Res}_{\Geo,\du}(f,\textbf{s},\chi)=J^{\Reg,\Res,1}_{\Geo,\du}(f,\textbf{s},\chi)+J^{\Reg,\Res,2}_{\Geo,\du}(f,\textbf{s},\chi),$ with 
\begin{align*}
J^{\Reg,\Res,1}_{\Geo,\du}(f,\textbf{s},\chi):=&\frac{1}{s_1+s_2-1}\int_{\mathbb{A}_F^{\times}}\int_{\mathbb{A}_F}f\left(\begin{pmatrix}
	1\\
	b&1
\end{pmatrix}\begin{pmatrix}
	y\\
	&1
\end{pmatrix}\right)db|y|^{s_2}\overline{\chi}(y)d^{\times}y,\\
J^{\Reg,\Res,2}_{\Geo,\du}(f,\textbf{s},\chi):=&-\frac{1}{s_1+s_2}\int_{\mathbb{A}_F^{\times}}f\left(\begin{pmatrix}
	y\\
	&1
\end{pmatrix}\right)|y|^{s_2}\overline{\chi}(y)d^{\times}y.
\end{align*}
Here $\int_{|x|\geq 1}$ represents the integral over the set $\{x\in\mathbb{A}_F:\ |x|_{\mathbb{A}_F}\geq 1\}.$ 

The integrals $J^{\Reg,+}_{\Geo,\du}(f,\textbf{s},\chi)$ and $J^{\Reg,\wedge}_{\Geo,\du}(f,\textbf{s},\chi)$ converges everywhere and thus define holomorphic functions in $\mathbb{C}^2.$ The functions $(s_1+s_2-1)J^{\Reg,\Res,1}_{\Geo,\du}(f,\textbf{s},\chi)$ and $(s_1+s_2)J^{\Reg,\Res,2}_{\Geo,\du}(f,\textbf{s},\chi)$ converges absolutely in $\mathbb{C}^2$.  

Recall that $s_0\in [4^{-1}\exp(-3\sqrt{\log C(\pi\times\chi)}),\exp(-3\sqrt{\log C(\pi\times\chi)})]$ is the parameter defined by \eqref{eq2.1} in \textsection\ref{2.1.5.}. 
\begin{prop}\label{prop17}
Let notation be as before. Then 
\begin{equation}\label{41}
J^{\Reg,\wedge}_{\Geo,\du}(f,\textbf{s}_0,\chi)\ll s_0^{-1}\mathcal{N}_f^{-1+2s_0}[M,M'Q]^{1+\varepsilon}C_{\infty}(\pi\otimes\chi)^{\frac{1}{4}+\varepsilon},
\end{equation}
where the implied constant depends only on $F,$ $\varepsilon,$ and $c_v,$ $C_v$ at $v\mid\infty,$ cf. \textsection \ref{2.1.2}.
\end{prop}
\begin{proof}
By definition we have
\begin{align*}
J^{\Reg,\wedge}_{\Geo,\du}(f,\textbf{s}_0,\chi):=&\int_{|x|\geq 1}|x|^{1-2s_0}\prod_{v\in\Sigma_F}\mathcal{J}_v(x_v)d^{\times}x,
\end{align*}
where we define by 
\begin{align*}
\mathcal{J}_v(x_v):=\int_{F_v^{\times}}\int_{F_v}f_v\left(\begin{pmatrix}
	1\\
	b_v&1
\end{pmatrix}\begin{pmatrix}
	y_v\\
	&1
\end{pmatrix}\right)\psi_v(b_vx_v)|y_v|_v^{s_0}\overline{\chi}_v(y_v)db_vd^{\times}y_v.
\end{align*}

\begin{itemize}
	\item At $v\nmid\nu(f),$ by definition $f_v\left(\begin{pmatrix}
	1\\
	b_v&1
\end{pmatrix}\begin{pmatrix}
	y_v\\
	&1
\end{pmatrix}\right)=0$ unless  
\begin{align*}
z_v\begin{pmatrix}
	1&\alpha \varpi_v^{-m_v}\\
	&1
\end{pmatrix}\begin{pmatrix}
	1\\
	b_v&1
\end{pmatrix}\begin{pmatrix}
	y_v\\
	&1
\end{pmatrix}\begin{pmatrix}
	1&\beta \varpi_v^{-m_v}\\
	&1
\end{pmatrix}\in K_v[n_v]
\end{align*}
for some $\alpha,\beta\in (\mathcal{O}_v/\varpi_v^{m_v}\mathcal{O}_v)^{\times},$ and $z_v\in F_v^{\times},$ i.e., 
\begin{equation}\label{36}
z_v\begin{pmatrix}
	y_v+\alpha b_vy_v\varpi_v^{-m_v}&(y_v+\alpha b_vy_v\varpi_v^{-m_v})\beta \varpi_v^{-m_v}+\alpha \varpi_v^{-m_v}\\
	b_vy_v&1+\beta b_vy_v\varpi_v^{-m_v}
\end{pmatrix}\in K_v[n_v].
\end{equation}

Analyzing the $(2,1)$-th entry of the matrix on the LHS of \eqref{36} yields that $e_v(z_v)+e_v(b_v)+e_v(y_v)\geq n_v\geq m_v.$ Hence an investigation of the $(1,1)$-th and $(2,2)$-th entry leads to  
\begin{align*}
\begin{cases}
e_v(y_v)+2e_v(z_v)=0\\
e_v(z_v)+e_v(y_v)\geq 0,\ \ e_v(z_v)\geq 0.
\end{cases}
\end{align*}
As a consequence, $e_v(z_v)=0,$ i.e., $z_v\in\mathcal{O}_v^{\times}.$ So $e_v(y_v)=0,$ $e_v(b_v)\geq n_v.$ 
 
Hence we have $f_v\left(\begin{pmatrix}
	1\\
	b_v&1
\end{pmatrix}\begin{pmatrix}
	y_v\\
	&1
\end{pmatrix}\right)=\textbf{1}_{\mathcal{O}_v^{\times}}(y_v)\textbf{1}_{\varpi_v^{n_v}\mathcal{O}_v}(b_v).$ After a change of variable (i.e., $\beta\mapsto y_v^{-1}\beta$),
\begin{align*}
\mathcal{J}_v(x_v)=\frac{|\tau(\chi_v)|^{-2}}{\Vol(K_v[n_v])}\sum_{\alpha,\beta}\int_{\varpi_v^{n_v}\mathcal{O}_v} \textbf{1}_{K_v[n_v]}\left(X_v\right)\psi_v(b_vx_v)\chi_v(\alpha)\overline{\chi}_v(\beta)db_v,
\end{align*}
where $X_v$ denotes the matrix
\begin{align*}
\begin{pmatrix}
	1+\alpha b_v\varpi_v^{-m_v}&(1+\alpha b_v\varpi_v^{-m_v})\beta \varpi_v^{-m_v}+\alpha \varpi_v^{-m_v}\\
	b_v&1+\beta b_v\varpi_v^{-m_v}
\end{pmatrix}.
\end{align*}

Note that $\textbf{1}_{K_v[n_v]}\left(X_v\right)\neq 0$ unless $(1+\alpha b_v\varpi_v^{-m_v})\beta +\alpha\in \varpi_v^{m_v}\mathcal{O}_v.$ Hence, 
\begin{align*}
\mathcal{J}_v(x_v)=\frac{|\tau(\chi_v)|^{-2}\chi_v(-1)}{\Vol(K_v[n_v])}\sum_{\alpha\in (\mathcal{O}_v/\varpi_v^{m_v}\mathcal{O}_v)^{\times}}\int_{\varpi_v^{n_v}\mathcal{O}_v} \psi_v(b_vx_v)\chi_v(1+\alpha b_v\varpi_v^{-m_v})db_v.
\end{align*}

Write $b_v=\varpi_v^m\gamma_v,$ $\gamma_v\in\mathcal{O}_v^{\times}.$ Changing the variable $\alpha\mapsto \gamma_v^{-1}\alpha,$  
\begin{align*}
\mathcal{J}_v(x_v)=\frac{|\tau(\chi_v)|^{-2}\chi_v(-1)}{\Vol(K_v[n_v])\zeta_{F_v}(1)}\sum_{m\geq n_v}q_v^{-m}G(m)R(m,x_v).
\end{align*}
where $G$ is the character sum
\begin{align*}
G(m):=\sum_{\alpha\in (\mathcal{O}_v/\varpi_v^{m_v}\mathcal{O}_v)^{\times}}\chi_v(1+\alpha \varpi_v^{m-m_v}),
\end{align*}
and $R(m,x_v)$ is the Ramanujan sum 
\begin{align*}
R(m,x_v):=\int_{\mathcal{O}_v^{\times}} \psi_v(\gamma_v\varpi_v^mx_v)d^{\times}\gamma_v.
\end{align*}

Applying the trivial bound $G(m)\ll q_v^{m_v},$ and $R(m,x_v)=0$ if $m<-e_v(x_v)-1,$ $R(m,x_v)\ll q_v^{-1}$ if $m=-e_v(x_v)-1,$ and $R(m,x_v)=1$ if $m\geq -e_v(x_v),$ we then deduce that 
\begin{align*}
\mathcal{J}_v(x_v)\ll \frac{|\tau(\chi_v)|^{-2}}{\Vol(K_v[n_v])}\cdot q_v^{n_v}\cdot |x_v|_v^{-1}=\Vol(K_v[n_v])^{-1}|x_v|_v^{-1}.
\end{align*}

\item When $v\mid\nu(f),$ $f_v$ is invariant under the transpose inverse. So 
$$
|x_v|_v^{1+2s_0}\mathcal{J}_v(x_v)=\mathcal{I}_v(x_v),
$$ 
which is defined by \eqref{34.}  and is calculated by Lemma \ref{lem14} in \textsection\ref{4.1.2}. In particular, $\mathcal{J}_v(x_v)=0$ unless $e_v(x_v)\geq -e_v(\mathfrak{N}_f).$ 
\item At $v\mid\infty,$ we have, by \eqref{45} in \textsection\ref{4.1.3}, that  
$$ 
\mathcal{J}_v(x_v)\ll \textbf{1}_{x_vT_v^{-1}-\gamma_v\ll_{\varepsilon} T_v^{-1/2+\varepsilon}}+O(|x_v|_v^{-1-2s_0}T_v^{-\infty}).
$$  
\end{itemize}

For $x=\otimes_vx_v$ with $e_v(x_v)\geq n_v-e_v(\mathfrak{N}_f)$ at all $v<\infty,$ we may write 
$$
x_{\fin}=\otimes_{v<\infty}x_v\in\mathfrak{n}\mathfrak{N}_f^{-2}\prod_{v<\infty}\mathfrak{p}_v^{n_v}
$$ 
for a unique integral ideal $\mathfrak{n}$ determined by $x_{\fin}.$ From the constraint $|x|_{\mathbb{A}_F}\geq 1$ and $|x_{\infty}|_{\infty}\ll 1$ (from the support of $f_{\infty}$) we have $N_F(\mathfrak{n})\ll \mathcal{N}_f^2[M,M'Q]^{-1}.$ 

Gathering the above estimates together we obtain  
\begin{align*}
J^{\Reg,\wedge}_{\Geo,\du}(f,\textbf{s}_0,\chi)\ll & \int_{\mathbb{A}_F^{\times}}|x|^{1+2s_0}\prod_{v\in\Sigma_F}\mathcal{J}_v(x_v)d^{\times}x\\
\ll & C_{\infty}(\pi\otimes\chi)^{\frac{1}{4}+\varepsilon}V^{-1}\mathcal{N}_f^{-1+2s_0+\varepsilon}\sum_{\substack{\mathfrak{n}\subseteq \mathcal{O}_F\\ N_F(\mathfrak{n})\ll [M,M'Q]\mathcal{N}_f^2}}\frac{1}{N_F(\mathfrak{n})^{1-2s_0}},
\end{align*}
where $V:=\prod_{v\mid [M,M'Q]}\Vol(K_v[n_v]),$ and the implied constant depends only on $F,$ $\varepsilon,$ and $c_v,$ $C_v$ at $v\mid\infty,$ cf. \textsection \ref{2.1.2}. Therefore, \eqref{41} follows.
\end{proof}

\begin{lemma}\label{lem14.}
Let notation be as before. Then
\begin{equation}\label{48} 
J^{\Reg,+}_{\Geo,\du}(f,\textbf{s}_0,\chi)\ll C_{\infty}(\pi\otimes\chi)^{\varepsilon}[M,M'Q]^{\varepsilon}\mathcal{N}_f^{1+2s_0},
\end{equation}
where the implied constant depends only on $F,$ $\varepsilon,$ and $c_v,$ $C_v$ at $v\mid\infty,$ cf. \textsection \ref{2.1.2}.
\end{lemma}
\begin{proof}
By definition we have
\begin{align*}
J^{\Reg,+}_{\Geo,\du}(f,\textbf{s}_0,\chi):=&\int_{|x|\geq 1}|x|^{2s_0}\prod_{v\in\Sigma_F}\mathcal{J}_v(x_v)d^{\times}x,
\end{align*}
where we define by 
\begin{align*}
\mathcal{J}_v(x_v):=\int_{F_v^{\times}}f_v\left(\begin{pmatrix}
	1\\
	x_v&1
\end{pmatrix}\begin{pmatrix}
	y_v\\
	&1
\end{pmatrix}\right)|y_v|_v^{s_0}\overline{\chi}_v(y_v)d^{\times}y_v.
\end{align*}

\begin{itemize}
\item Let $v\mid\mathfrak{Q}.$ According to the arguments in the proof of Proposition \ref{prop17} we have $f_v\left(\begin{pmatrix}
	1\\
	x_v&1
\end{pmatrix}\begin{pmatrix}
	y_v\\
	&1
\end{pmatrix}\right)\neq 0$ unless $e_v(y_v)=0,$ and $e_v(x_v)\geq n_v.$ Consequently,
\begin{align*}
\mathcal{J}_v(x_v)=\frac{|\tau(\chi_v)|^{-2}}{\Vol(K_v[n_v])}\sum_{\alpha,\beta}\textbf{1}_{K_v[n_v]}\left(X_v\right)\chi_v(\alpha)\overline{\chi}_v(\beta),
\end{align*}
where $X_v$ is defined by 
\begin{align*}
\begin{pmatrix}
	1+\alpha x_v\varpi_v^{-m_v}&(1+\alpha x_v\varpi_v^{-m_v})\beta \varpi_v^{-m_v}+\alpha \varpi_v^{-m_v}\\
	x_v&1+\beta x_v\varpi_v^{-m_v}
\end{pmatrix}.
\end{align*}
It follows from $X_v\in K_v[n_v]$ that $(1+\alpha x_v\varpi_v^{-m_v})\beta \varpi_v^{-m_v}+\alpha \varpi_v^{-m_v}\in \mathcal{O}_v,$ i.e., there is a linear constraint between $\alpha$ and $\beta.$ Hence, by triangle inequality, 
\begin{align*}
\mathcal{J}_v(x_v)\ll \textbf{1}_{e_v(x_v)\geq n_v}\cdot \frac{|\tau(\chi_v)|^{-2}}{\Vol(K_v[n_v])}\sum_{\alpha}1\ll \textbf{1}_{e_v(x_v)\geq n_v}\Vol(K_v[n_v])^{-1}.
\end{align*}

\item At $v<\infty,$ $v\nmid\mathfrak{Q}\nu(f),$ we have 
$$
\mathcal{J}_v(x_v)=N_{F_v}(\mathfrak{D}_{F_v})^{-1/2}\Vol(K_v[n_v])^{-1}\textbf{1}_{\mathcal{O}_v}(x_v).
$$ 

\item For $v\mid\nu(f),$ by Lemma \ref{lem14} we have $f_v\left(\begin{pmatrix}
	1\\
	x_v&1
\end{pmatrix}\begin{pmatrix}
	y_v\\
	&1
\end{pmatrix}\right)\neq 0$ unless $e_v(x_v)\geq -e_v(\mathfrak{N}_f),$ where $\mathcal{N}_f$ is defined by \eqref{61}. Moreover, the classification in the proof of Lemma \ref{lem14} yields that 
$$
\mathcal{J}_v(x_v)
\ll_{\varepsilon}\mathcal{N}_f^{\varepsilon} \textbf{1}_{e_v(x_v)\geq -e_v(\mathfrak{N}_f)}\mathcal{N}_f^{-1+2s_0+\varepsilon}.
$$  

\item Let $v\mid\infty.$ The support of $f_v$ leads to that $\mathcal{J}_v(x_v)=0$ unless $|x_v|_v\ll 1.$
\end{itemize}

As before, for $x=\otimes_vx_v$ with $\otimes_v\mathcal{J}_v(x_v)\neq 0,$ we must have  $e_v(x_v)\geq n_v-e_v(\mathfrak{N}_f)$ at all $v<\infty.$ For such an $x,$  we may write 
$$
x_{\fin}=\otimes_{v<\infty}x_v\in\mathfrak{n}\mathfrak{N}_f^{-2}
\prod_{v<\infty}\mathfrak{p}_v^{n_v}$$ 
for a unique integral ideal $\mathfrak{n}$ determined by $x_{\fin}.$ From the constraint $|x|_{\mathbb{A}_F}\geq 1$ and $|x_{\infty}|_{\infty}\ll 1$ (from the support of $f_{\infty}$) we have $N_F(\mathfrak{n})\ll \mathcal{N}_f^2[M,M'Q]^{-1}.$ 

Gathering the above estimate the function $J^{\Reg,+}_{\Geo,\du}(f,\textbf{s}_0,\chi)$ is 
\begin{align*}
\ll V^{-1}\mathcal{N}_f^{-1+2s_0+\varepsilon}\sum_{\substack{\mathfrak{n}\subseteq\mathcal{O}_F\\ N_F(\mathfrak{n})\ll [M,M'Q]\mathcal{N}_f^2}}N_F(\mathfrak{n})^{2s_0}\int |x_{\infty}|_{\infty}^{2s_0}\prod_{v\mid\infty}\mathcal{J}_v(x_v)d^{\times}x_{\infty},
\end{align*}
where $V:=\prod_{v\mid [M,M'Q]}\Vol(K_v[n_v]),$ and $x_{\infty}$ ranges in $F_{\infty}^{\times}$ with the constraint $N_F(\mathfrak{n})[M,M'Q]\mathcal{N}_f^{-2}\ll |x_{\infty}|_{\infty}\ll 1.$ 

By \eqref{2.7} we have 
\begin{align*}
\int |x_{\infty}|_{\infty}^{2s_0}\prod_{v\mid\infty}\mathcal{J}_v(x_v)d^{\times}x_{\infty}\ll T^{1/2+\varepsilon}\cdot \int_{\frac{N_F(\mathfrak{n})[M,M'Q]}{\mathcal{N}_f^{2}}\ll |x_{\infty}|_{\infty}\ll 1}\frac{T_v^{-\frac{1}{2}+\varepsilon}}{|x_v|_v}d^{\times}x_{\infty},
\end{align*}
which is $\ll (N_F(\mathfrak{n})[M,M'Q])^{-1+\varepsilon}\mathcal{N}_f^{2+\varepsilon}\cdot \prod_{v\mid\infty}T_v^{\varepsilon}.$ Here the factor $T_v^{1/2+\varepsilon}$ comes from the product of $T^{1+\varepsilon}$ and the integral over $y_v$ under the constraint $|\Ad^*(g)\tau-\tau|\ll T_v^{-1/2+\varepsilon},$ where $g=\begin{pmatrix}
	y_v&\\
	x_vy_v&1
\end{pmatrix};$ and the factor $T_v^{-1/2+\varepsilon}|x_v|_v^{-1}$ is the bound from Proposition \ref{prop3.1} in \textsection\ref{2.2.2.}. Note that $T_v\asymp C(\pi_v\otimes\chi_v)^{1/2}.$  Therefore, 
\begin{align*}
J^{\Reg,+}_{\Geo,\du}(f,\textbf{s}_0,\chi)\ll [M,M'Q]^{\varepsilon}\mathcal{N}_f^{1+2s_0+\varepsilon}\prod_{v\mid\infty}T_v^{\varepsilon}\sum_{\substack{\mathfrak{n}\subseteq\mathcal{O}_F\\ N_F(\mathfrak{n})\ll [M,M'Q]\mathcal{N}_f^2}}N_F(\mathfrak{n})^{-1+2s_0+\varepsilon},
\end{align*}
from which the estimate \eqref{48} follows.
\end{proof}
\begin{remark}
From the above proof we see that there exists a constant $C,$ depending on $c_v,$ $C_v$ at $v\mid\infty,$ cf. \textsection \ref{2.1.2}, such that $J^{\Reg,+}_{\Geo,\du}(f,\textbf{s},\chi)\equiv 0,$ $\mathbf{s}\in\mathbb{C}^2,$
if $[M,M'Q]\geq C\mathcal{N}_f^{2}.$	
\end{remark}

\begin{lemma}\label{lem5.4}
Let notation be as before. Let $s_0$ be defined by \eqref{eq2.1} in \textsection\ref{2.1.5.} and $\mathbf{s}_0=(s_0,s_0).$ Then 
\begin{equation}\label{39}
J^{\Reg,\Res,1}_{\Geo,\du}(f,\textbf{s}_0,\chi)\ll C_{\infty}(\pi\otimes\chi)^{\varepsilon}[M,M'Q]^{\varepsilon}\mathcal{N}_f^{1+2s_0+\varepsilon},
\end{equation} 
where the implied constant depends only on $F,$ $\varepsilon,$ and $c_v,$ $C_v$ at $v\mid\infty,$ cf. \textsection \ref{2.1.2}. 
\end{lemma}
\begin{proof}
By definition, $J^{\Reg,\Res,1}_{\Geo,\du}(f,\textbf{s}_0,\chi)=\prod_{v\in\Sigma_F}\mathcal{J}_v,$ where 
\begin{align*}
\mathcal{J}_v:=\int_{F_v^{\times}}\int_{F_v}f_v\left(\begin{pmatrix}
	1\\
	b_v&1
\end{pmatrix}\begin{pmatrix}
	y_v\\
	&1
\end{pmatrix}\right)db_v|y_v|_v^{s_0}\overline{\chi}_v(y_v)d^{\times}y_v.
\end{align*}

At $v\mid\nu(f),$ similar to the analysis in the proof of Proposition \ref{prop17}, we have 
\begin{align*}
|\mathcal{J}_v|\leq \frac{|\tau(\chi_v)|^{-2}}{\Vol(K_v[n_v])}\sum_{\alpha}\int_{\varpi_v^{n_v}\mathcal{O}_v} db_v\ll N_{F_v}(\mathfrak{D}_{F_v})^{-1/2}\Vol(K_v[n_v])^{-1}q_v^{-n_v},
\end{align*}
where $N_{F_v}(\mathfrak{D}_{F_v})$ is the norm of local different at $v.$ By Lemma \ref{lem14} we have 
$$
\prod_{v\mid\nu(f)}\mathcal{J}_v\ll \mathfrak{N}_f^{-1+2s_0}\int_{\mathfrak{N}_f^{-2}\mathcal{O}_f}db_v \ll \mathcal{N}_f^{1+2s_0}.
$$ 

At archimedean places, by \eqref{2.7} we have 
\begin{align*}
\mathcal{J}_{\infty}\ll \prod_{v\mid\infty}T_v^{1/2+\varepsilon}\cdot \int_{F_v}\min\{1,T_v^{-1/2}|b_v|_v^{-1}\}\textbf{1}_{|b_v|_v\ll 1}db_v\ll \prod_{v\mid\infty}T_v^{\varepsilon},
\end{align*}
which is $\asymp C_{\infty}(\pi\otimes\chi)^{\varepsilon}.$ Here the factor $T_v^{1/2+\varepsilon}$ comes from the product of $T^{1+\varepsilon}$ and the integral over $y_v$ under the constraint $|\Ad^*(g)\tau-\tau|\ll T_v^{-1/2+\varepsilon},$ where $g=\begin{pmatrix}
	y_v&\\
	b_vy_v&1
\end{pmatrix}.$ 

Therefore, gathering the above estimates, 
\begin{align*}
J^{\Reg,\Res,1}_{\Geo,\du}(f,\textbf{s}_0,\chi)=\frac{1}{2s_0-1}\prod_{v\in\Sigma_F}\mathcal{J}_v\ll C_{\infty}(\pi\otimes\chi)^{\varepsilon}[M,M'Q]^{\varepsilon}\mathcal{N}_f^{1+2s_0+\varepsilon},
\end{align*}
proving \eqref{39}.
\end{proof}

\begin{lemma}\label{lem5.5}
Let notation be as before. Then 
\begin{equation}\label{40}
J^{\Reg,\Res,2}_{\Geo,\du}(f,\textbf{s}_0,\chi)\ll s_0^{-1}\mathcal{N}_f^{-1+2s_0}[M,M'Q]^{1+\varepsilon}C_{\infty}(\pi\otimes\chi)^{\frac{1}{4}+\varepsilon},
\end{equation}
where the implied constant depends only on $F,$ $\varepsilon,$ and $c_v,$ $C_v$ at $v\mid\infty,$ cf. \textsection \ref{2.1.2}. 
\end{lemma}
\begin{proof}
By definition we have $J^{\Reg,\Res,2}_{\Geo,\du}(f,\textbf{s}_0,\chi)=-\frac{1}{2s_0}\prod_{v\in\Sigma_F}\mathcal{J}_v,$ where
\begin{align*}
\mathcal{J}_v:=&\int_{F_v^{\times}}f_v\left(\begin{pmatrix}
	y_v\\
	&1
\end{pmatrix}\right)|y_v|_v^{s_0}\overline{\chi}_v(y_v)d^{\times}y_v.
\end{align*}

\begin{itemize}
	\item At $v\mid\mathfrak{Q},$ we have $f_v\left(\begin{pmatrix}
	y_v\\
	&1
\end{pmatrix}\right)=0$ unless $y_v\in\mathcal{O}_v^{\times}.$ Then $\mathcal{J}_v$ becomes
\begin{align*}
\frac{|\tau(\chi_v)|^{-2}}{\Vol(K_v[n_v])}\sum_{\alpha,\beta}\int_{\mathcal{O}_v^{\times}} \textbf{1}_{K_v[n_v]}\left(\begin{pmatrix}
	y_v&(\alpha +y_v\beta) \varpi_v^{-m_v}\\
	&1
\end{pmatrix}\right)\chi_v(\alpha)\overline{\chi}_v(\beta)\overline{\chi}_v(y_v)d^{\times}y_v.
\end{align*}

Changing variable $\beta\mapsto y_v^{-1}\beta$ we then obtain 
\begin{align*}
\mathcal{J}_v=\frac{|\tau(\chi_v)|^{-2}}{\Vol(K_v[n_v])}\sum_{\substack{\alpha,\beta\in (\mathcal{O}_v/\varpi_v^{m_v}\mathcal{O}_v)^{\times}\\ \alpha+\beta=0}}\chi_v(\alpha)\overline{\chi}_v(\beta)\ll \Vol(K_v[n_v])^{-1}.
\end{align*}

\item At $v<\infty,$ $v\nmid\mathfrak{Q}\nu(f),$ we have $f_v\left(\begin{pmatrix}
	y_v\\
	&1
\end{pmatrix}\right)=\Vol(K_v[n_v])^{-1}\textbf{1}_{\mathcal{O}_v^{\times}}(y_v).$ So $\mathcal{J}_v=\Vol(K_v[n_v])^{-1}.$
\item At other places, we have $\prod_{v\mid\nu(f)}\mathcal{J}_v\ll \mathcal{N}_f^{-1+2s_0},$ and by \eqref{2.7},
$$
\prod_{v\mid\infty}\mathcal{J}_v\ll \prod_{v\mid\infty}T_v^{\frac{1}{2}+\varepsilon} \asymp C_{\infty}(\pi\otimes\chi)^{\frac{1}{4}+\varepsilon}. 
$$
Here the factor $T_v^{1/2+\varepsilon}$ comes from the product of $T^{1+\varepsilon}$ and the integral over $y_v$ under the constraint $|\Ad^*(g)\tau-\tau|\ll T_v^{-1/2+\varepsilon},$ where $g=\begin{pmatrix}
	y_v&\\
	&1
\end{pmatrix}.$ 
\end{itemize}

The estimate \eqref{40} follows from the above local bounds.
\end{proof}

\section{The Geometric Side: Regular Orbital Integrals}\label{sec6}
Recall the definition \eqref{17...} in \textsection\ref{sec3.7}: 
\begin{align*}
J^{\Reg,\RNum{2}}_{\Geo,\bi}(f,\textbf{s}_0,\chi):=\sum_{t\in F-\{0,1\}}\prod_{v\in\Sigma_F} \mathcal{E}_v(t),
\end{align*}
where for $v\in\Sigma_F,$ 
\begin{equation}\label{51..}
\mathcal{E}_v(t):=\int_{F_v^{\times}}\int_{F_v^{\times}}f_v\left(\begin{pmatrix}
	y_v&x_v^{-1}t\\
	x_vy_v&1
\end{pmatrix}\right)|x_v|_v^{2s_0}|y_v|_v^{s_0}\overline{\chi}_v(y_v)d^{\times}y_vd^{\times}x_v.
\end{equation}
Recall that $s_0\in [4^{-1}\exp(-3\sqrt{\log C(\pi\times\chi)}),\exp(-3\sqrt{\log C(\pi\times\chi)})]$ is the deformation parameter defined by \eqref{eq2.1} in \textsection\ref{2.1.5.}. 

By Theorem 5.6 in \cite{Yan23a} (or \cite{RR05}) the orbital integrals $J^{\Reg,\RNum{2}}_{\Geo,\bi}(f,\textbf{s}_0,\chi)$ converges absolutely. We shall establish an upper bound for it as follows. 
\begin{thmx}\label{thmD}
Let notation be as before. Then 
\begin{align*}
J^{\Reg,\RNum{2}}_{\Geo,\bi}(f,\textbf{s}_0,\chi)\ll C_{\infty}(\pi\otimes\chi)^{\varepsilon}(MQ\mathcal{N}_f)^{\varepsilon}\mathcal{N}_f\cdot \prod_{v\mid\mathfrak{Q}}q_v^{\frac{\min\{r_{\omega_v},m_v\}+m_v}{2}},
\end{align*}	
where the implied constant depends on $\varepsilon,$ $F,$ $c_v,$ and $C_v,$ $v\mid\infty$. Here $M', M, Q,$ $\mathfrak{Q},$ and $m_v$ are defined in \textsection \ref{2.1.5}, and $\mathcal{N}_f$ is defined in \textsection\ref{3.6.2}.
\end{thmx}

\subsection{Local Estimates: unramified finite places}

\begin{lemma}\label{lem6.2}
Let $v\in \Sigma_{F,\fin}$ be such that $v\nmid\mathfrak{Q}\nu(f).$ Then 
\begin{align*}
\mathcal{E}_v(t)\ll \frac{(1-e_v(1-t))(1+e_v(t)-2e_v(1-t))}{q_v^{(2n_v+e_v(t-1))s_0}\Vol(K_v[n_v])}\textbf{1}_{\substack{e_v(t-1)\leq 0\\ e_v(t)-e_v(1-t)\geq n_v}}.
\end{align*}
Moreover, $\mathcal{E}_v(t)=1$ if $e_v(t)=e_v(1-t)=0,$ $n_v=0,$ and $v\nmid \mathfrak{D}_F.$ In particular, $\mathcal{E}_v(t)=1$ for all but finitely many $v$'s.
\end{lemma}
\begin{proof}
Let $v\nmid\mathfrak{Q}\nu(f).$ Then $f_v\left(\begin{pmatrix}
	y_v&x_v^{-1}t\\
	x_vy_v&1
\end{pmatrix}\right)\neq 0$ unless there exists some $z_v\in F_v^{\times}$ such that
\begin{align*}
z_v\begin{pmatrix}
	y_v&x_v^{-1}t\\
	x_vy_v&1
\end{pmatrix}\in K_v[n_v],
\end{align*}
which implies that 
\begin{align*}
\begin{cases}
2e_v(z_v)+e_v(y_v)+e_v(1-t)=0\\
e_v(z_v)+e_v(y_v)\geq 0\\
e_v(z_v)+e_v(x_v)+e_v(y_v)\geq n_v\\
e_v(z_v)\geq 0\\
e_v(z_v)-e_v(x_v)+e_v(t)\geq 0.
\end{cases}
\end{align*}
Hence 
\begin{equation}\label{47}
\begin{cases}
2e_v(x_v)+e_v(y_v)\geq 2n_v+e_v(t-1)\\
e_v(1-t)\leq e_v(y_v)\leq -e_v(1-t)\\
e_v(1-t)+n_v\leq e_v(x_v)\leq e_v(t)-e_v(1-t)\\
0\leq n_v\leq e_v(t)-e_v(1-t).
\end{cases}
\end{equation}

Therefore, when $e_v(t)=e_v(1-t)=0$ and $n_v=0,$ \eqref{47} implies that $x_v, y_v\in\mathcal{O}_v^{\times},$ so $\mathcal{E}_v(t)=\Vol(K_v[n_v])^{-1}N_{F_v}(\mathfrak{D}_{F_v})^{-1/2}.$ In general,
\begin{align*}
\mathcal{E}_v(t)\ll \frac{(1-e_v(1-t))(1+e_v(t)-2e_v(1-t))}{q_v^{(2n_v+e_v(t-1))s_0}\Vol(K_v[n_v])}\textbf{1}_{\substack{e_v(t-1)\leq 0\\ e_v(t)-e_v(1-t)\geq n_v}},
\end{align*}
where the implied constant depends on $\mathfrak{D}_{F_v}.$ 
\end{proof}

\subsection{Local Estimates: ramification}\label{sec6.2}
\begin{prop}\label{4}
Let $v\mid\mathfrak{Q}.$ Let $m_v=r_{\chi_v}\geq 1$ and $n_v=\max\{r_{\pi_v},m_v+r_{\omega_v}\}$ (cf. \textsection\ref{2.1.5}).   Then 
\begin{align*}
\mathcal{E}_v(t)\ll 
\begin{cases}
m_v(1-e_v(t))^2 q_v^{\frac{m_v-e_v(t)}{2}-3e_v(t)s_0} \ \ &\text{if $e_v(t)\leq -1,$}\\
\kappa_vq_v^{\frac{r_{\omega_v}+n_v+e_v(t)}{2}}\textbf{1}_{r_{\omega_v}\leq m_v}\ \ &\text{if $e_v(t)\geq n_v-m_v,$  $e_v(t-1)=0,$}\\
0\ \ &\text{otherwise,}
\end{cases}
\end{align*}
where $\kappa_v=m_v(e_v(t)+m_v-n_v+1),$ and the implied constant is absolute. 
\end{prop}
\begin{proof}
By definition, $f_v\left(\begin{pmatrix}
	y_v&x_v^{-1}t\\
	x_vy_v&1
\end{pmatrix}
\right)=0$ unless 
\begin{equation}\label{48.}
\varpi_v^k\begin{pmatrix}
		1&\alpha \varpi_v^{-m_v}\\
		&1
	\end{pmatrix}\begin{pmatrix}
	y_v&x_v^{-1}t\\
	x_vy_v&1
\end{pmatrix}\begin{pmatrix}
		1&\beta \varpi_v^{-m_v}\\
		&1
	\end{pmatrix}\in K_v[n_v]
\end{equation}
for some $k\in\mathbb{Z}.$ Write $x_v=\varpi_v^{r_1}\gamma_1,$ $y_v=\varpi_v^{r_2}\gamma_2,$ where $r_1, r_2\in\mathbb{Z}$ and $\gamma_1, \gamma_2\in\mathcal{O}_v^{\times}.$ Then \eqref{48.} becomes
\begin{align*}
\varpi_v^k\begin{pmatrix}
		1&\gamma_1\alpha \varpi_v^{-m_v}\\
		&1
	\end{pmatrix}\begin{pmatrix}
	\varpi_v^{r_2}&\varpi_v^{-r_1}t\\
	\varpi_v^{r_1+r_2}&1
\end{pmatrix}\begin{pmatrix}
		1&\gamma_1\gamma_2\beta \varpi_v^{-m_v}\\
		&1
	\end{pmatrix}
\in K_v[n_v].
\end{align*}

Changing variables $\alpha\mapsto \gamma_1^{-1}\alpha,$ $\beta\mapsto \gamma_1^{-1}\gamma_2^{-1}\beta,$ the above constraint becomes 
\begin{equation}\label{49} 
\varpi_v^kY_{\alpha,\beta,r_1,r_2,t}\in K_v[n_v]
\end{equation}
for some $k\in\mathbb{Z},$ where $Y_{\alpha,\beta,r_1,r_2,t}$ is defined by 
\begin{align*}
\begin{pmatrix}
	\varpi_v^{r_2}+\alpha \varpi_v^{r_1+r_2-m_v}&(\varpi_v^{r_2}+\alpha \varpi_v^{r_1+r_2-m_v})\beta \varpi_v^{-m_v}+\varpi_v^{-r_1}t+\alpha \varpi_v^{-m_v}\\
	\varpi_v^{r_1+r_2}&1+\beta \varpi_v^{r_1+r_2-m_v}
\end{pmatrix}.
\end{align*} 

By definition the local integral $\mathcal{E}_v(t)$ becomes
\begin{align*}
\frac{1}{|\tau(\chi_v)|^2}\sum_{\alpha,\beta}\chi(\alpha)\overline{\chi}(\beta)\sum_{r_1, r_2\in\mathbb{Z}} q_v^{-2r_1s_0-r_2s_0}\textbf{1}_{Y_{\alpha,\beta,r_1,r_2,t}\in Z(F_v)K_v[n_v]}f_v(Y_{\alpha,\beta,r_1,r_2,t};\omega_v),
\end{align*} 
where $f_v(\cdot;\omega_v)$ is defined by \eqref{5.} in \textsection\ref{2.2.2}.

Note that \eqref{49} amounts to 
\begin{equation}\label{50}
\begin{cases}
2k+r_2+e_v(1-t)=0\\
k+r_1+r_2\geq n_v\\
\varpi_v^k(\varpi_v^{r_2}+\alpha \varpi_v^{r_1+r_2-m_v})\in\mathcal{O}_v^{\times}\\
\varpi_v^k(1+\beta \varpi_v^{r_1+r_2-m_v})\in\mathcal{O}_v^{\times}\\
\varpi_v^k\big[(\varpi_v^{r_2}+\alpha \varpi_v^{r_1+r_2-m_v})\beta \varpi_v^{-m_v}+\varpi_v^{-r_1}t+\alpha \varpi_v^{-m_v}\big]\in\mathcal{O}_v,
\end{cases}
\end{equation}
which implies that $k\geq 0$ as $n_v\geq m_v.$ 

\begin{enumerate}
\item Suppose that $k=0$ in \eqref{50}, which implies that 
\begin{equation}\label{52}
\begin{cases}
r_2+e_v(1-t)=0\\
r_1+r_2\geq n_v\\
\min\{r_2, r_1+r_2-m_v\}=0\\
(\varpi_v^{r_2}+\alpha \varpi_v^{r_1+r_2-m_v})\beta \varpi_v^{-m_v}+\varpi_v^{-r_1}t+\alpha \varpi_v^{-m_v}\in\mathcal{O}_v.
\end{cases}
\end{equation}	

\begin{itemize}
	\item Suppose $r_2\geq 1.$ Then $r_1+r_2=n_v=m_v.$ Since $e_v(t)-r_1\geq -m_v,$ then $e_v(t)+r_2\geq 0.$ So $e_v(t)-e_v(1-t)\geq 0.$ In this case we have $1\leq r_2=-e_v(t-1)=-e_v(t).$ Moreover, the constraint $n_v=m_v$ leads to that $r_{\omega_v}=0,$ i.e., $\omega_v$ is unramified. Therefore, the contribution to $\mathcal{E}_v(t)$ from this case is
\begin{align*}
\mathcal{E}^{(1)}_v(t):=\frac{q_v^{-2m_vs_0-e_v(t)s_0}\textbf{1}_{e_v(t)\leq -1}}{|\tau(\chi_v)|^2\Vol(K_v[n_v])}\sum_{\substack{\alpha,\beta\in (\mathcal{O}_v/\varpi_v^{m_v}\mathcal{O}_v)^{\times}\\ (\varpi_v^{-e_v(t)}+\alpha)\beta \varpi_v^{-m_v}+\varpi_v^{-r_1}t+\alpha \varpi_v^{-m_v}\in\mathcal{O}_v}}\chi(\alpha)\overline{\chi}(\beta).
\end{align*}	

Write $t=\varpi_v^{e_v(t)}\gamma$ under the embedding $F^{\times}\hookrightarrow F_v^{\times},$ where $\gamma\in\mathcal{O}_v^{\times}.$ Then the last sum over $\alpha$ and $\beta$ becomes
\begin{align*}
\sum_{\substack{\alpha,\beta\in (\mathcal{O}_v/\varpi_v^{m_v}\mathcal{O}_v)^{\times}\\ (\varpi_v^{-e_v(t)}+\alpha )(\beta+1)\equiv -\gamma+\varpi_v^{-e_v(t)}\pmod{\varpi_v^{m_v}}}}\chi(\alpha)\overline{\chi}(\beta),
\end{align*}
which, after a change of variables, is equal to 
\begin{align*}
\mathcal{J}_v^{(1)}(t):=\sum_{\substack{\alpha,\beta\in (\mathcal{O}_v/\varpi_v^{m_v}\mathcal{O}_v)^{\times}\\ \alpha\beta\equiv -\gamma+\varpi_v^{-e_v(t)}\pmod{\varpi_v^{m_v}}}}\chi(\alpha-\varpi_v^{-e_v(t)})\overline{\chi}(\beta-1).
\end{align*}

As a special case of Proposition 2 on p.71 of \cite{BH08} we have $\mathcal{J}_v^{(1)}(t)\ll m_vq_v^{m_v/2},$ where the implied constant is absolute. Hence, 
\begin{equation}\label{59.}
\mathcal{E}^{(1)}_v(t)\ll m_vq_v^{\frac{m_v}{2}-2m_vs_0-e_v(t)s_0} \textbf{1}_{e_v(t)\leq -1}.
\end{equation}
	
\item Suppose $r_2=0.$ Then $e_v(1-t)=0,$ $r_1\geq n_v,$ and $e_v(t)\geq r_1-m_v\geq 0.$ 

The contribution to $\mathcal{E}_v(t)$ from this case is
\begin{align*}
\mathcal{E}^{(2)}_v(t):=\sum_{r_1=n_v}^{e_v(t)+m_v}\frac{q_v^{-2r_1s_0}\textbf{1}_{e_v(t)\geq r_1-m_v}}{|\tau(\chi_v)|^2\Vol(K_v[n_v])}\cdot \mathcal{J}_v^{(2)}(r_1,t),
\end{align*}	
where 
\begin{align*}
\mathcal{J}_v^{(2)}(r_1,t):=\sum_{\substack{\alpha,\beta\in (\mathcal{O}_v/\varpi_v^{m_v}\mathcal{O}_v)^{\times}\\ (1+\alpha\varpi_v^{r_1-m_v})\beta \varpi_v^{-m_v}+\varpi_v^{-r_1}t+\alpha \varpi_v^{-m_v}\in\mathcal{O}_v}}\chi(\alpha)\overline{\chi}(\beta)\omega_v(1+\alpha\varpi_v^{r_1-m_v}).
\end{align*}

By Lemma \ref{lem24} below (in conjunction with $r_1\geq n_v$) the sum $\mathcal{E}^{(2)}_v(t)$ is 
\begin{align*}
\ll m_v(e_v(t)+m_v-n_v+1)\frac{\textbf{1}_{e_v(t)\geq n_v-m_v}\cdot \textbf{1}_{r_{\omega_v}\leq m_v}}{|\tau(\chi_v)|^2\Vol(K_v[n_v])}\cdot q_v^{\frac{r_{\omega_v}+m_v}{2}+\frac{m_v-n_v+e_v(t)}{2}}.
\end{align*}

Since $|\tau(\chi_v)|^2\Vol(K_v[n_v])\gg q_v^{m_v-n_v},$ then 
\begin{equation}\label{58.}
\mathcal{E}^{(2)}_v(t)\ll  m_v^2(e_v(t)+m_v-n_v+1)\textbf{1}_{e_v(t)\geq n_v-m_v}\cdot \textbf{1}_{r_{\omega_v}\leq m_v}q_v^{\frac{r_{\omega_v}+n_v+e_v(t)}{2}}.
\end{equation}
\end{itemize}

\item Suppose that $k\geq 1,$ then \eqref{50} amounts to 
\begin{equation}\label{51}
\begin{cases}
2k+r_2+e_v(1-t)=0\\
k+r_1+r_2=n_v=m_v\\
k+r_2\geq 0\\
\varpi_v^k\big[(\varpi_v^{r_2}+\alpha \varpi_v^{r_1+r_2-m_v})\beta \varpi_v^{-m_v}+\varpi_v^{-r_1}t+\alpha \varpi_v^{-m_v}\big]\in\mathcal{O}_v.
\end{cases}
\end{equation}

As a consequence, \eqref{51} yields 
\begin{align*}
\begin{cases}
e_v(1-t)=e_v(t)\leq -k\leq -1\\
e_v(t)\leq r_2\leq -e_v(t)-2\\
m_v+e_v(t)+1\leq r_1\leq m_v\\
k=m_v-r_1-r_2\geq 1\\
\big[(\varpi_v^{k+r_2}+\alpha )\beta \varpi_v^{-m_v}+\varpi_v^{k-r_1}t+\alpha \varpi_v^{k-m_v}\big]\in\mathcal{O}_v.
\end{cases}
\end{align*}

Since $m_v=n_v,$ then $r_{\omega_v}=0,$ i.e., $\omega_v$ is trivial. Hence the contribution from this case to $\mathcal{E}_v(t)$ is 
\begin{align*}
\mathcal{E}^{(3)}_v(t):=\sum_{r_1=m_v+e_v(t)+1}^{m_v}\sum_{r_2=e_v(t)}^{-e_v(t)-2}\frac{q_v^{-2r_1s_0-r_2s_0}\textbf{1}_{r_1+r_2\leq m_v-1}}{|\tau(\chi_v)|^2\Vol(K_v[m_v])}\cdot \mathcal{J}_v^{(3)}(r_1,r_2,t),
\end{align*}
where we set $k=m_v-r_1-r_2$ and 
\begin{align*}
\mathcal{J}_v^{(3)}(r_1,r_2,t):=\sum_{\substack{\alpha,\beta\in (\mathcal{O}_v/\varpi_v^{m_v}\mathcal{O}_v)^{\times}\\ (\varpi_v^{k+r_2}+\alpha)(\beta +\varpi_v^{k})+t\varpi_v^{-e_v(t)}-\varpi_v^{2k+r_2}\in \varpi_v^{m_v}\mathcal{O}_v}}\chi(\alpha)\overline{\chi}(\beta).
\end{align*}

Note that $r_2+k\geq 0.$ Then $\gamma:=t\varpi_v^{-e_v(t)}-\varpi_v^{2k+r_2}\in \mathcal{O}_v^{\times}.$ After a change of variables, we obtain 
\begin{align*}
\mathcal{J}_v^{(3)}(r_1,r_2,t)=\sum_{\substack{\alpha,\beta\in (\mathcal{O}_v/\varpi_v^{m_v}\mathcal{O}_v)^{\times}\\ \alpha\beta \equiv \gamma \pmod{\varpi_v^{m_v}}}}\chi(\alpha-\varpi_v^{k+r_2})\overline{\chi}(\beta-\varpi_v^{k}).
\end{align*}

By \cite[Proposition 2]{BH08} and the fact that $k\leq -e_v(t)$,
\begin{align*}
\mathcal{J}_v^{(3)}(r_1,r_2,t)\ll m_vq_v^{\frac{m_v}{2}}\cdot q_v^{\frac{\min\{k,k+r_2,m_v\}}{2}}\leq m_vq_v^{\frac{m_v-e_v(t)}{2}}.
\end{align*}

Therefore, we have
\begin{equation}\label{6.10.}
\mathcal{E}^{(3)}_v(t)\ll m_v(1-e_v(t))^2q_v^{-(2m_v+2s_0+3e_v(t))s_0}\cdot \textbf{1}_{e_v(t)\leq -1}\cdot q_v^{\frac{m_v-e_v(t)}{2}},
\end{equation}
where the implied constant is absolute. 
\end{enumerate}

Then Proposition \ref{4} follows from \eqref{59.},  \eqref{58.} and \eqref{6.10.}.
\end{proof}

\begin{lemma}\label{lem24}
	Let notation be as before. Then 
	\begin{equation}\label{5}
	\mathcal{J}_v^{(2)}(r_1,t)\ll m_vq_v^{\frac{r_{\omega_v}+m_v}{2}+\frac{m_v-r_1+e_v(t)}{2}}\cdot \textbf{1}_{r_{\omega_v}\leq m_v}, 
	\end{equation}		
	where the implied constant is absolute. 
\end{lemma}
\begin{proof}
Note that $(1+\alpha\varpi_v^{r_1-m_v})\beta \varpi_v^{-m_v}+\varpi_v^{-r_1}t+\alpha \varpi_v^{-m_v}\in\mathcal{O}_v$ amounts to 
$$
(\alpha^{-1}+\varpi_v^{r_1-m_v})\beta +\alpha^{-1}\varpi_v^{m_v-r_1}t+1\in\varpi_v^{m_v}\mathcal{O}_v.
$$

Changing the variable $\alpha\mapsto \alpha^{-1},$ we have
\begin{align*}
\mathcal{J}_v^{(2)}(r_1,t)=\sum_{\substack{\alpha,\beta\in (\mathcal{O}_v/\varpi_v^{m_v}\mathcal{O}_v)^{\times}\\ (\alpha+\varpi_v^{r_1-m_v})\beta +\alpha\varpi_v^{m_v-r_1}t+1\in\varpi_v^{m_v}\mathcal{O}_v}}\overline{\chi}(\alpha)\overline{\chi}(\beta)\omega_v(1+\alpha\varpi_v^{r_1-m_v}).
\end{align*}

Changing variables $\alpha\mapsto \alpha-\varpi_v^{r_1-m_v}$ and $\beta\mapsto \beta-\varpi_v^{m_v-r_1}t,$ $\mathcal{J}_v^{(2)}(r_1,t)$ becomes 
\begin{align*}
\sum_{\substack{\alpha,\beta\in (\mathcal{O}_v/\varpi_v^{m_v}\mathcal{O}_v)^{\times}\\ \alpha\beta \equiv t-1\pmod{\varpi_v^{m_v}}}}\overline{\chi}(\alpha-\varpi_v^{r_1-m_v})\overline{\chi}(\beta-\varpi_v^{m_v-r_1}t)\omega_v(1-t+\alpha\varpi_v^{r_1-m_v}).
\end{align*}

Let $h\in \mathcal{O}_v^{\times}.$ Let $\mathcal{J}_v^{(2)}(r_1,t,\psi_v,h)$ be defined by 
\begin{align*}
\sum_{\substack{\alpha,\beta\in (\mathcal{O}_v/\varpi_v^{m_v}\mathcal{O}_v)^{\times}\\ \alpha\beta \equiv t-1\pmod{\varpi_v^{m_v}}}}\overline{\chi}(\alpha-\varpi_v^{r_1-m_v})\overline{\chi}(\beta-\varpi_v^{m_v-r_1}t)\psi_v(h\alpha q_v^{-r_{\omega_v}}).
\end{align*}
Here we recall that $\psi_v$ is a fixed unramified additive chatacter of $F_v.$ By definition, we have $\mathcal{J}_v^{(2)}(r_1,t,\psi_v,h)=0$ if $r_{\omega_v}>m_v.$

Notice that $\chi$ is primitive. By Theorem 2G of \cite{Sch06} (cf. p.45) or Deligne's quasi-orthogonality of trace functions (cf. \cite{Del80}) and Lemmas 12.2 and 12.3 in \cite{IK04}, following the proof of Proposition 2 in \cite{BH08}, we have 
\begin{equation}\label{59}
\mathcal{J}_v^{(2)}(r_1,t,\psi_v,h) \ll m_v(q_v^{m_v-r_1+e_v(t)},q_v^{r_1-m_v},q_v^{m_v})^{\frac{1}{2}}q_v^{\frac{m_v}{2}}\cdot \textbf{1}_{r_{\omega_v}\leq m_v},  
\end{equation}
where the implied constant is absolute. In particular, \eqref{59} yields that 
\begin{equation}\label{60.}
\mathcal{J}_v^{(2)}(r_1,t,\psi_v,h) \ll m_vq_v^{\frac{m_v-r_1+e_v(t)}{2}}\cdot q_v^{\frac{m_v}{2}}\cdot \textbf{1}_{r_{\omega_v}\leq m_v}. 
\end{equation}

Since $\omega_v$ is primitive, we have the Gauss sum expansion
\begin{equation}\label{60}
\omega_v(\alpha)=\frac{1}{\tau(\overline{\omega}_v)}\sum_{h\in (\mathcal{O}_v/{\omega}_v^{r_{\omega_v}}\mathcal{O}_v)^{\times}}\overline{\omega}_v(h)\psi_v(h\alpha q_v^{-r_{\omega_v}}),
\end{equation}
where $q_v^{r_{\omega_v}}$ is the conductor of $\omega_v.$ 

Hence, \eqref{5} follows from \eqref{60.}, \eqref{60}, triangle inequality, and the fact that $|\tau(\overline{\omega}_v)|=q_v^{r_{\omega_v}/2}.$
\end{proof}

\subsection{Local Estimates: amplification} 

\begin{prop}\label{prop6.5}
Let notation be as before. Let $f=f(g;\mathbf{v},\mathbf{i})=\otimes_{v\in\Sigma_F}f_v$ be the test function defined in \textsection\ref{3.6.2}.  Let $v\mid \nu(f).$ Then 
\begin{equation}\label{6.15}
\mathcal{E}_v(t)\ll (3-e_v(t-1))\textbf{1}_{e_v(t)-e_v(t-1)\geq -e_v(\mathfrak{N}_f)}\cdot \|f_v\|_{\infty}\cdot q_v^{(2-e_v(t-1))s_0},
\end{equation}
where $\|f_v\|_{\infty}$ is the sup-norm of $f_v,$ and the implied constant is absolute.
\end{prop}
\begin{proof}
Changing variables $y_v\mapsto y_vx_v^{-1}$ we obtain from the definition \eqref{51..} that
\begin{align*}
\mathcal{E}_v(t)=\int_{F_v^{\times}}\int_{F_v^{\times}}f_v\left(\begin{pmatrix}
	y_vx_v^{-1}&x_v^{-1}t\\
	y_v&1
\end{pmatrix}\right)|x_vy_v|_v^{s_0}\chi_v(x_vy_v^{-1})d^{\times}y_vd^{\times}x_v.
\end{align*}

Since $f_v$ is bi-$G(\mathcal{O}_v)$-invariant, we have 
\begin{align*}
\mathcal{E}_v(t)=&\sum_{r_1\in \mathbb{Z}}\sum_{r_2\in \mathbb{Z}}f_v\left(\begin{pmatrix}
			\varpi_v^{r_2-r_1}&\varpi_v^{-r_1}t\\
					\varpi_v^{r_2}&1
				\end{pmatrix}\right)\chi_v(\varpi_v)^{r_1-r_2}q_v^{-(r_1+r_2)s_0}.
\end{align*}

Note that, for $g_v\in G(F_v),$ $f_v(g_v)\neq 0$ unless $g_v\in\supp f_v,$ namely (by \textsection\ref{3.2.4}), 
\begin{equation}\label{65.} 
g_v\in \begin{cases}
Z(F_v)K_{v}\diag(\varpi_{v}^{i_{v}},\varpi_{v}^{-i_{v}})K_{v},\ &\text{if $v\in \mathbf{v}_0,$}\\
Z(F_v)K_{v}\diag(\varpi_{v},1)K_{v},\ &\text{if $v\in \mathbf{v}_1,$}\\
Z(F_v)K_{v}\diag(1,\varpi_{v})K_{v},\ &\text{if $v\in \mathbf{v}_2.$}
\end{cases}
\end{equation}

\begin{enumerate}
\item Suppose $r_1\leq 0,$ and $r_2\leq 0.$ Then
\begin{equation}\label{66}
f_v\left(\begin{pmatrix}
\varpi_v^{r_2-r_1}&\varpi_v^{-r_1}t\\
\varpi_v^{r_2}&1
\end{pmatrix}\right)=f_v\left(\begin{pmatrix}
\varpi_v^{-r_1}(1-t)&\\
&\varpi_v^{r_2}
\end{pmatrix}\right).
\end{equation}
By \eqref{65.}, the RHS of \eqref{66} is zero unless $|r_1+r_2-e_v(t-1)|=e_v(\mathfrak{N}_f),$ which yields that $e_v(t-1)\leq e_v(\mathfrak{N}_f)\leq 2,$ and $r_1+r_2\geq -e_v(\mathfrak{N}_f)+e_v(t-1)\geq -2+e_v(t-1).$ So the contribution corresponding to $r_1\leq 0$ and $r_2\leq 0$ is majorized by 
\begin{align*}
\textbf{1}_{e_v(t-1)\leq e_v(\mathfrak{N}_f)}\sum_{\substack{e_v(t-1)-2\leq r_1\leq 0}}\sum_{\substack{r_2\leq 0\\ |r_1+r_2-e_v(t-1)|=e_v(\mathfrak{N}_f)}}
\|f_v\|_{\infty}q_v^{-(r_1+r_2)s_0},
\end{align*}
which is $\leq 2\cdot (3-e_v(t-1))\|f_v\|_{\infty}q_v^{(2-e_v(t-1))s_0}\textbf{1}_{e_v(t-1)\leq e_v(\mathfrak{N}_f)}.$
Here the implied constant is absolute. 
   	
\item Suppose $r_1\leq 0,$ $r_2\geq 1.$ Then
\begin{equation}\label{6.17}
f_v\left(\begin{pmatrix}
\varpi_v^{r_2-r_1}&p^{-r_1}t\\
\varpi_v^{r_2}&1
\end{pmatrix}\right)=f_v\left(\begin{pmatrix}
\varpi_v^{r_2-r_1}(1-t)&\\
					&1
				\end{pmatrix}\right).
\end{equation}
By \eqref{65.}, the RHS of \eqref{6.17} is zero unless $|r_2-r_1+e_v(t-1)|=e_v(\mathfrak{N}_f),$ which yields that $e_v(t-1)\leq e_v(\mathfrak{N}_f)-1\leq 1,$ and $r_1\geq r_2+e_v(t-1)-e_v(\mathfrak{N}_f)\geq e_v(t-1)-1.$ So the contribution corresponding to $r_1\leq 0$ and $r_2\geq 1$ is majorized by 
\begin{align*}
\textbf{1}_{e_v(t-1)\leq e_v(\mathfrak{N}_f)}\sum_{\substack{e_v(t-1)-1\leq r_1\leq 0}}\sum_{\substack{r_2\geq 1\\ |r_2-r_1+e_v(t-1)|=e_v(\mathfrak{N}_f)}}\|f_v\|_{\infty}q_v^{-(r_1+r_2)s_0},
\end{align*}
which is $\leq 2\cdot (2-e_v(t-1))\|f_v\|_{\infty}q_v^{(1-e_v(t-1))s_0}\textbf{1}_{e_v(t-1)\leq e_v(\mathfrak{N}_f)}.$
Here the implied constant is absolute.

\item Suppose $r_1\geq 1,$ $r_2\leq 0.$ Then
\begin{equation}\label{6.18}
f_{p}\left(\begin{pmatrix}
					p^{r_2-r_1}&p^{-r_1}t\\
					p^{r_2}&1
\end{pmatrix}\right)=f_{p}\left(\begin{pmatrix}
					p^{r_2-r_1}&\\
					&t-1
				\end{pmatrix}\right).
\end{equation}
By \eqref{65.}, the RHS of \eqref{6.18} is zero unless $|r_1-r_2+e_v(t-1)|=e_v(\mathfrak{N}_f),$ which yields that $e_v(t-1)\leq e_v(\mathfrak{N}_f)-1\leq 1,$ and $r_2\geq r_1+r_v(t-1)-e_v(\mathfrak{N}_f)\geq e_v(t-1)-1.$ So the contribution corresponding to $r_1\geq 1$ and $r_2\leq 0$ is majorized by 
\begin{align*}
\textbf{1}_{e_v(t-1)\leq e_v(\mathfrak{N}_f)}\sum_{\substack{e_v(t-1)-1\leq r_2\leq 0}}\sum_{\substack{r_1\geq 1\\ |r_1-r_2+e_v(t-1)|=e_v(\mathfrak{N}_f)}}\|f_v\|_{\infty}q_v^{-(r_1+r_2)s_0},
\end{align*}
which is $\leq 2\cdot (2-e_v(t-1))\|f_v\|_{\infty}q_v^{(1-e_v(t-1))s_0}\textbf{1}_{e_v(t-1)\leq e_v(\mathfrak{N}_f)}.$
Here the implied constant is absolute. 
				
\item Suppose $r_1\geq 1,$ $r_2\geq 1.$ Then
\begin{equation}\label{6.19}
f_{p}\left(\begin{pmatrix}
p^{r_2-r_1}&p^{-r_1}t\\
p^{r_2}&1
\end{pmatrix}\right)=f_v\left(\begin{pmatrix}
p^{-r_1}&\\
&p^{r_2}(t-1)
\end{pmatrix}\right).
\end{equation}
By \eqref{65.}, the RHS of \eqref{6.19} is zero unless $|r_1+r_2+e_v(t-1)|=e_v(\mathfrak{N}_f),$ which yields that $e_v(t-1)\leq e_v(\mathfrak{N}_f)-r_1-r_2\leq e_v(\mathfrak{N}_f)-2\leq 0,$ and $r_1\leq e_v(\mathfrak{N}_f)-r_2-e_v(t-1)\leq 2-1-e_v(t-1)=1-e_v(t-1).$ So the contribution corresponding to $r_1\geq 1$ and $r_2\geq 1$ is majorized by 
\begin{align*}
\textbf{1}_{e_v(t-1)\leq e_v(\mathfrak{N}_f)}\sum_{\substack{1\leq r_1\leq 1-e_v(t-1)}}\sum_{\substack{r_2\geq 1\\ |r_1-r_2+e_v(t-1)|=e_v(\mathfrak{N}_f)}}\|f_v\|_{\infty}q_v^{-(r_1+r_2)s_0},
\end{align*}
which is $\ll \|f_v\|_{\infty}q_v^{-2s_0}\textbf{1}_{e_v(t-1)\leq e_v(\mathfrak{N}_f)}.$
Here the implied constant is absolute. 
\end{enumerate}

Therefore, \eqref{6.15} follows from the above discussions and the fact that if $e_v(t-1)\leq e_v(\mathfrak{N}_f),$ then $e_v(t)-e_v(t-1)\geq -e_v(\mathfrak{N}_f).$
\end{proof}

\subsection{Local Estimates: archimedean}\label{sec6.4}
\begin{lemma}\label{lem6.6}
Let notation be as before. Let $v\mid \infty.$ Let $\mathcal{E}_v(t)$ be defined by \eqref{51..}. Then $\mathcal{E}_v(t)\ll_{\varepsilon} T_v^{1+\varepsilon}\log|t|_v,$ where the implied constant depends on $\varepsilon.$
\end{lemma}
\begin{proof}
Recall the definition \eqref{51..}:
\begin{align*}
\mathcal{E}_v(t)=\int_{F_v^{\times}}\int_{F_v^{\times}}f_v\left(\begin{pmatrix}
	y_v&x_v^{-1}t\\
	x_vy_v&1
\end{pmatrix}\right)|x_v|_v^{2s_0}|y_v|_v^{s_0}\overline{\chi}_v(y_v)d^{\times}y_vd^{\times}x_v.
\end{align*}

By the support of $f_v$ (cf. \textsection\ref{3.2.1}), we have $f_v\left(\begin{pmatrix}
	y_v&x_v^{-1}t\\
	x_vy_v&1
\end{pmatrix}\right)=0$ unless there are some $z_v\in F_v^{\times}$ such that
\begin{equation}\label{6.23}
z_v\begin{pmatrix}
	y_v&x_v^{-1}t\\
	x_vy_v&1
\end{pmatrix}\in \supp \tilde{f}_{v},
\end{equation}
which is $\subseteq \big\{g\in G(F_v):\ g=I_{n+1}+O(T_v^{-\varepsilon}),\ \Ad^*(g)\tau=\tau+O(T_v^{-\frac{1}{2}+\varepsilon})\big\}$ (cf. \eqref{245}). Suppose that \eqref{6.23} holds. Then $z_v=1+O(T_v^{-\varepsilon}).$ So 
\begin{align*}
\begin{cases}
y_v=1+O(T_v^{-\varepsilon})\\
|x_v|_v\ll 1\\
y_v(1-t)=1+O(T_v^{-\varepsilon})\\
x_v^{-1}t\ll 1.
\end{cases}
\end{align*}
As a consequence, we have 
\begin{equation}\label{6.24}
\begin{cases}
y_v=1+O(T_v^{-\varepsilon})\\
x_v\ll 1\\
|1-t|_v=1+O(T_v^{-\varepsilon})\\
|t|_v\ll |x_v|_v\ll 1.
\end{cases}
\end{equation}

Therefore, we have by \eqref{6.24} and the  triangle inequality that
\begin{align*}
\mathcal{E}_v(t)\ll \int_{1+O(T_v^{-\varepsilon})}\int_{|t|_v\ll |x_v|_v\ll 1}\|f_v\|_{\infty}d^{\times}x_vd^{\times}y_v\ll_{\varepsilon} T_v^{1+\varepsilon}\log|t|_v,
\end{align*}
where we make use of the sup-norm estimate \eqref{250}.
\end{proof}
\begin{remark}
In fact the upper bound $\mathcal{E}_v(t)\ll_{\varepsilon}T_v^{1+\varepsilon}\log|t|_v$ is actually quite sharp in terms of $\log |t|_v$. In the definition of regular orbital integrals we exclude $t=1,$ which corresponds to the fact that the dual orbital integral $J^{\Reg}_{\Geo,\du}(f,\textbf{s},\chi)$ has simple poles at $s_1+s_2\in \{0,1\}$. Here $\mathbf{s}=(s_1,s_2)\in\mathbb{C}^2.$ 
\end{remark}

Let $v\mid \infty.$ Define by 
\begin{equation}\label{6.21}
\mathcal{E}_v^{\dagger}:=\int_{F_v^{\times}}\int_{F_v}\max_{t\in F-\{0,1\}}\Big|f_v\left(\begin{pmatrix}
	y_v&x_v^{-1}t\\
	x_vy_v&1
\end{pmatrix}\right)\Big||x_v|_v^{2s_0}|y_v|_v^{s_0}dx_vd^{\times}y_v.
\end{equation}

\begin{lemma}\label{lem6.8}
Let notation be as before. Let $v\mid \infty.$ Let $\mathcal{E}_v^{\dagger}$ be defined by \eqref{6.21}. Then 
\begin{equation}\label{6.22}
\mathcal{E}_v^{\dagger}\ll C(\pi_v\otimes\chi_v)^{\frac{1}{4}+\varepsilon},
\end{equation}
where the implied constant depends on $\varepsilon,$ $F,$ $c_v,$ and $C_v$ defined in \textsection\ref{2.1.2}.  
\end{lemma}
\begin{proof}
By \eqref{2.7}, we have $\mathcal{E}_v^{\dagger}\ll \mathcal{E}_v^{(1)}+\mathcal{E}_v^{(2)},$ where
\begin{align*}
&\mathcal{E}_v^{(1)}:=T_v^{1/2+\varepsilon}\int_{1+O(T_v^{-\varepsilon})}\int_{|x_v|_v\ll T_v^{-1/2+\varepsilon}}\min\bigg\{1,\frac{T_v^{-1/2+\varepsilon}}{|x_vy_v|_v}\bigg\}dx_vd^{\times}y_v,\\
&\mathcal{E}_v^{(2)}:=T_v^{1/2+\varepsilon}\int_{1+O(T_v^{-\varepsilon})}\int_{T_v^{-1/2+\varepsilon}\ll |x_v|_v\ll 1}\min\bigg\{1,\frac{T_v^{-1/2+\varepsilon}}{|x_vy_v|_v}\bigg\}dx_vd^{\times}y_v.
\end{align*}
Here the factor $T_v^{1/2+\varepsilon}$ comes from the product of $T^{1+\varepsilon}$ and the integral over $y_v$ under the constraint $|\Ad^*(g)\tau-\tau|\ll T_v^{-1/2+\varepsilon},$ where $g=\begin{pmatrix}
	y_v&x_v^{-1}t\\
	x_vy_v&1
\end{pmatrix}.$ 

By the sup-norm bound \eqref{250} we obtain 
\begin{align*}
\mathcal{E}_v^{(1)}\ll & T_v^{1/2+\varepsilon}\int_{1+O(T_v^{-\varepsilon})}\int_{|x_v|_v\ll T_v^{-1/2+\varepsilon}}dx_vd^{\times}y_v\ll 
T_v^{\varepsilon},\\
\mathcal{E}_v^{(2)}\ll & T_v^{1/2+\varepsilon}\int_{1+O(T_v^{-\varepsilon})}\int_{T_v^{-1/2+\varepsilon}\ll |x_v|_v\ll 1}\frac{T^{-1/2+\varepsilon}}{|x_v|_v}dx_vd^{\times}y_v\ll 
T_v^{\varepsilon}.
\end{align*}
Therefore, \eqref{6.22} follows from the above estimates, and the fact that $T_v\asymp C(\pi_v\otimes\chi_v)^{1/2}.$ 
\end{proof}

\subsection{Bounding Regular Orbital Integrals: Proof of Theorem \ref{thmD}} 
\subsubsection{The support of the rationals $t\in F-\{0,1\}$}
\begin{lemma}\label{lem10}
Let notation be as before. Suppose $t\in F-\{0,1\}.$ Let $f=f(g;\mathbf{v},\mathbf{i})=\otimes_{v\in\Sigma_F}f_v$ be the test function defined in \textsection\ref{3.6.2}. Then the integral  $\prod_{v\in\Sigma_F} \mathcal{E}_v(t)$ converges absolutely and it vanishes unless 
\begin{equation}\label{13}
\frac{t}{t-1}\in \mathfrak{X}(Q,f):=\bigg\{y\in F^{\times}\cap \mathfrak{N}_f^{-1}\prod_{\substack{v<\infty\\ v\nmid\mathfrak{Q}\nu(f)}}\mathfrak{p}_v^{n_v}\mathcal{O}_F:\ |y|_{v}\ll 1,\ v\mid \infty\bigg\},
\end{equation}
where the implied constant depends only on $\supp f_{\infty}.$
\end{lemma}
\begin{proof}
By Lemma \ref{lem6.2} the integral $\mathcal{E}_v(t)=1$ for all but finitely many $v$'s. It then follows from Lemmas \ref{lem6.2} and \ref{lem6.6}, and Propositions \ref{4} and \ref{prop6.5} that $\prod_{v\in\Sigma_F} \mathcal{E}_v(t)$ converges absolutely and it is vanishing unless
\begin{equation}\label{6.27}
\begin{cases}
e_v(t)-e_v(t-1)\geq -e_v(\mathfrak{N}_f),\ & \text{if $v\mid\nu(f),$}\\
e_v(t)-e_v(t-1)\geq n_v,\ & \text{if $v\nmid\mathfrak{Q}\nu(f),$}\\
e_v(t)-e_v(t-1)\geq 0,\ & \text{if $v\mid\mathfrak{Q}.$}
\end{cases}
\end{equation}

Since $t/(t-1)\in F-\{0,1\},$ then \eqref{13} follows from \eqref{6.27}.
\end{proof}

\subsubsection{Estimate of nonarchimedean integrals}\label{6.5.2}
Fix an ideal $\mathfrak{R}\subset \mathcal{O}_{F}$ with the property that $e_v(\mathfrak{R})=n_v$ for $v\nmid\mathfrak{Q}\nu(f),$ and   $e_v(\mathfrak{R})=0$ for all $v<\infty$ and $v\mid\mathfrak{Q}\nu(f).$ For $t\in F-\{0,1\}$ with $t/(t-1)\in\mathfrak{X}(Q,f)$ (cf. \eqref{13}), we may write 
\begin{equation}\label{6.28}
t/(t-1)=u,\ \ u\in \mathfrak{R}\mathfrak{N}_f^{-1}\mathcal{O}_F.
\end{equation}
Then $1/(t-1)=u-1.$

 \begin{lemma}\label{lem6.10}
 Let notation be as above. Let $\mathcal{E}_v(t)$ be defined by \eqref{51..}. Set $\mathcal{E}_{\fin}(t):=\prod_{v<\infty}|\mathcal{E}_v(t)|.$ Let $t/(t-1)=u\in \mathfrak{R}\mathfrak{N}_f^{-1}\mathcal{O}_F$ be as in \eqref{6.28}. Then 
 \begin{align*}
\mathcal{E}_{\fin}(t)\ll_{\varepsilon}& \frac{(MQ\cdot N_F(u)N_F(u-1))^{\varepsilon}\cdot (MQ\mathcal{N}_f^2)^{2s_0}}{N_F(t-1)^{s_0}\mathcal{N}_f}\cdot \prod_{v\nmid\mathfrak{Q}\nu(f)}\frac{1}{\Vol(K_v[n_v])}\cdot\prod_{v\mid\mathfrak{Q}}\mathcal{J}_v(u),
\end{align*}
where $s_0$ is defined by \eqref{eq2.1} in  \textsection\ref{2.1.5.}, and 
\begin{equation}\label{6.29.}
\mathcal{J}_v(u):=q_v^{\frac{m_v+e_v(u-1)}{2}}\cdot \textbf{1}_{\substack{e_v(u-1)\geq 1}}+q_v^{\frac{r_{\omega_v}+n_v+e_v(u)}{2}}\cdot \textbf{1}_{\substack{e_v(u)\geq n_v-m_v}}\textbf{1}_{r_{\omega_v}\leq m_v}.
\end{equation}
Here $n_v$ and $M'$ are defined in \textsection\ref{2.1.5}, $\nu(f)$ and $\mathcal{N}_f$ is defined in \textsection\ref{3.6.2}.
\end{lemma}
\begin{proof}
\begin{enumerate} 
	\item Suppose $v\nmid\mathfrak{Q}\nu(f).$ By Lemma \ref{lem6.2}, 
\begin{align*}
\mathcal{E}_v(t)\ll \frac{(1-e_v(1-t))(1+e_v(t)-2e_v(1-t))}{\Vol(K_v[n_v])}\textbf{1}_{\substack{e_v(t-1)\leq 0}}\cdot q_v^{-(2n_v+e_v(t-1))s_0}.
\end{align*}

In this case we have $-e_v(t-1)=e_v(u-1).$ So $(1-e_v(1-t))(1+e_v(t)-2e_v(1-t))\ll [e_v(t)-e_v(t-1)]\cdot (1+e_v(u-1))^2.$ 

Note that $e_v(\mathfrak{N}_f)=0,$ which implies that $e_v(u)\geq 0.$ Hence, \eqref{6.28} leads to $e_v(t)-e_v(t-1)=e_v(u)\ll_{\varepsilon}q_v^{\varepsilon e_v(u)}.$ Also, $e_v(u-1)\ll_{\varepsilon}q_v^{\varepsilon e_v(u-1)}.$ Therefore, by the trivial bound $q_v^{-2n_vs_0}\ll 1,$ we have that 
\begin{equation}\label{6.29}
\mathcal{E}_v(t)\ll_{\varepsilon} \frac{q_v^{\varepsilon e_v(u(u-1))}}{\Vol(K_v[n_v])}\textbf{1}_{\substack{e_v(t-1)\leq 0}}\cdot q_v^{-e_v(t-1)s_0}.
\end{equation}

\item Suppose $v\mid\mathfrak{Q}.$ By Proposition \ref{4},
\begin{align*}
\mathcal{E}_v(t)\ll 
\begin{cases}
m_v(1-e_v(t))^2 q_v^{\frac{m_v-e_v(t)}{2}-3e_v(t)s_0} \ \ &\text{if $e_v(t)\leq -1,$}\\
\kappa_vq_v^{\frac{r_{\omega_v}+n_v+e_v(t)}{2}}\textbf{1}_{r_{\omega_v}\leq m_v}\ \ &\text{if $e_v(t)\geq n_v-m_v,$ and $e_v(t-1)=0,$}\\
0\ \ &\text{otherwise,}
\end{cases}
\end{align*}
where $\kappa_v=m_v(e_v(t)+m_v-n_v+1).$ Notice that $m_v\ll_{\varepsilon}q_v^{\varepsilon m_v},$ and
\begin{align*}
\begin{cases}
1-e_v(t)\ll 1-e_v(t-1)=1+e_v(u-1)\ll_{\varepsilon}q_v^{\varepsilon e_v(u-1)},\ \ \text{if $e_v(t)\leq -1$}.\\
\kappa_v=m_v(1+m_v-n_v+e_v(u))\ll_{\varepsilon}q_v^{\varepsilon e_v(MQ\cdot u)},\ \ \text{if $e_v(t-1)=0$}.
\end{cases}
\end{align*}
Hence, $\mathcal{E}_v(t)$ is majorized by the product of the factor $q_v^{\varepsilon e_v(MQ\cdot u(u-1))}$ and
\begin{align*}
\begin{cases}
q_v^{\frac{m_v+e_v(u-1)}{2}+3e_v(u-1)s_0} \ \ &\text{if $e_v(t-1)\leq -1,$}\\
q_v^{\frac{r_{\omega_v}+n_v+e_v(u)}{2}}\textbf{1}_{r_{\omega_v}\leq m_v}\ \ &\text{if $e_v(u)\geq n_v-m_v,$\ $e_v(t-1)=0,$}\\
0\ \ &\text{otherwise.}
\end{cases}
\end{align*}
In particular, we have, for $v\mid\mathfrak{Q},$ that
\begin{equation}\label{6.30}
\mathcal{E}_v(t)\ll_{\varepsilon}q_v^{\varepsilon e_v(MQ\cdot u)}\cdot \mathcal{J}_v(u)\ll_{\varepsilon}q_v^{\varepsilon e_v(MQ\cdot u(u-1))}\cdot \mathcal{J}_v(u),
\end{equation}
where $\mathcal{J}_v(u)$ is defined by \eqref{6.29.}. Here we make use of the fact that $e_v(u-1)\geq 0$ as $e_v(\mathfrak{N}_f)=0.$ 

\item Suppose $v\mid\nu(f).$ By Proposition \ref{prop6.5}, we have
\begin{align*}
\mathcal{E}_v(t)\ll (3-e_v(t-1))\cdot \|f_v\|_{\infty}\cdot q_v^{(2-e_v(t-1))s_0}.
\end{align*}
We have $\|f_v\|_{\infty}\ll q_v^{-e_v(\mathfrak{N}_f)/2}$ (cf. \textsection\ref{3.2.4}), and 
$$
(3-e_v(t-1))\ll_{\varepsilon}q_v^{\varepsilon e_v(\mathfrak{N}_f\cdot N_F(u-1))}\ll q_v^{\varepsilon e_v(\mathfrak{N}_f^2\cdot (u(u-1)))}, 
$$
where we make use of the fact that $e_v(\mathfrak{N}_fu)\geq 0.$ Consequently,  
\begin{equation}\label{6.31}
\mathcal{E}_v(t) \ll_{\varepsilon} \mathcal{N}_f^{-1}\cdot q_v^{\varepsilon e_v(\mathfrak{N}_f^2\cdot u(u-1))+(2-e_v(t-1))s_0}.
\end{equation}
\end{enumerate}

Gathering the above local estimates \eqref{6.29}, \eqref{6.30} and \eqref{6.31}, then Lemma \ref{lem6.10} follows.
\end{proof}

For $x_{\infty}=\otimes_{v\mid\infty}x_v\in F_{\infty}.$ For $t\in \mathfrak{X}(Q,f),$ parametrize $t/(t-1)$ via \eqref{6.28}. We define 
\begin{equation}\label{6.33}
\mathcal{C}(x_{\infty}):=\sum_{\substack{t\in F-\{0,1\},\ \frac{t}{t-1}=u\in \mathfrak{X}(Q,f)\\
|\frac{t}{t-1}|_v\ll |x_v|_v,\ v\mid\infty}}\mathcal{E}_{\fin}(t).
\end{equation}
\begin{lemma}\label{lem6.11}
Let notation be as before. Let $x_{\infty}\in F_{\infty}^{\times}.$ Let $\mathcal{C}(x_{\infty})$ be defined by \eqref{6.33}. Then
 \begin{equation}\label{6.34}
\mathcal{C}(x_{\infty})\ll_{\varepsilon,F}(MQ\mathcal{N}_f(1+|x_{\infty}|_{\infty}))^{\varepsilon}\cdot |x_{\infty}|_{\infty}\cdot \mathcal{N}_f\cdot  \prod_{v\mid\mathfrak{Q}}q_v^{\frac{\min\{r_{\omega_v},m_v\}+m_v}{2}},
\end{equation}
where the implied constant depends on $\varepsilon$ and $F.$
\end{lemma}
\begin{proof}
Recall that $\mathcal{J}_v(u)$ is defined by \eqref{6.29.}. Consider the auxiliary sum 
\begin{equation}\label{6.34.}
\mathcal{S}^{\dagger}(x_{\infty}):=\sum_{\substack{u\in \mathfrak{R}\mathfrak{N}_f^{-1}\mathcal{O}_F\cap  F^{\times}\\
|u|_v\ll |x_v|_v,\ v\mid\infty}}\prod_{v\mid\mathfrak{Q}}\mathcal{J}_v(u)\ll \prod_{v\mid\mathfrak{Q}}q_v^{\frac{m_v}{2}}\cdot \mathcal{S}(x_{\infty}),
\end{equation}
where
\begin{align*}
\mathcal{S}(x_{\infty}):=\sum_{\substack{u\in \mathfrak{R}\mathfrak{N}_f^{-1}\mathcal{O}_F\cap  F^{\times}\\
|u|_v\ll |x_v|_v,\ v\mid\infty}}\prod_{v\mid\mathfrak{Q}}\Bigg[q_v^{\frac{e_v(u-1)}{2}}1_{e_v(u-1)\geq 1}+q_v^{\frac{r_{\omega_v}+n_v-m_v+e_v(u)}{2}} \textbf{1}_{\substack{e_v(u)\geq n_v-m_v\\ r_{\omega_v}\leq m_v}}\Bigg].
\end{align*}
Based on Lemma \ref{lem6.10}, we can simplify the majorization of $\mathcal{C}(x_{\infty})$ by focusing on the upper bound of $\mathcal{S}^{\dagger}(x_{\infty})$ in the latter part of this proof.

We proceed to deal with $\mathcal{S}^{\dagger}(x_{\infty}).$ Expanding the product to obtain
\begin{align*}
\mathcal{S}(x_{\infty})\leq \sum_{\substack{\mathfrak{a}\subsetneq \mathcal{O}_F\\ \mathfrak{a}\mid \mathfrak{Q}}}\Bigg[\sum_{\substack{u\in \mathfrak{R}\mathfrak{N}_f^{-1}\mathcal{O}_F\cap  F^{\times}\\
|u|_v\ll |x_v|_v,\ v\mid\infty\\
e_v(u-1)\geq 1,\ v\mid\mathfrak{a}}}\prod_{v\mid \mathfrak{a}}q_v^{\frac{e_v(u-1)}{2}}+\prod_{v\mid \mathfrak{a}}q_v^{\iota_v}\sum_{\substack{u\in \mathfrak{R}\mathfrak{N}_f^{-1}\mathcal{O}_F\cap  F^{\times}\\
|u|_v\ll |x_v|_v,\ v\mid\infty\\ e_v(u)\geq n_v-m_v,\ v\mid\mathfrak{a}}}\prod_{v\mid \mathfrak{a}}q_v^{\frac{e_v(u)}{2}}\Bigg],
\end{align*}
where $\iota_v:=\frac{\min\{r_{\omega_v},m_v\}+n_v-m_v}{2}.$

Fix an ideal $\mathfrak{b}\subset \mathcal{O}_F$ such that $e_v(\mathfrak{b})=n_v-m_v$ for all $v\mid\mathfrak{a}$ and $e_v(\mathfrak{b})=0$ at $v<\infty$ and $v\nmid\mathfrak{a}.$ Then
\begin{equation}\label{6.34..}
\sum_{\substack{u\in \mathfrak{R}\mathfrak{N}_f^{-1}\mathcal{O}_F\cap  F^{\times}\\
|u|_v\ll |x_v|_v,\ v\mid\infty\\ e_v(u)\geq n_v-m_v,\ v\mid\mathfrak{a}}}\prod_{v\mid \mathfrak{a}}q_v^{\frac{e_v(u)}{2}}=\sum_{\substack{\mathfrak{L}\subseteq \mathcal{O}_F\\
\mathfrak{L}+\mathfrak{b}=\mathcal{O}_F}}\sum_{\substack{\mathfrak{J}\subseteq \mathfrak{b}\\
\mathfrak{J}+\mathfrak{L}=\mathcal{O}_F}}\sum_{\substack{u\in \mathfrak{R}\mathfrak{N}_f^{-1}\mathcal{O}_F\cap  F^{\times}\\
|u|_v\ll |x_v|_v,\ v\mid\infty\\ (u)=\mathfrak{R}\mathfrak{N}_f^{-1}\mathfrak{L}\mathfrak{J}}}\prod_{v\mid \mathfrak{a}}q_v^{\frac{e_v(u)}{2}},
\end{equation}
where $(u)$ is the principal ideal. By definition, 
$\mathfrak{R}\mathfrak{N}_f+\mathfrak{a}=\mathcal{O}_F.$ As a consequence, 
$$
\prod_{v\mid \mathfrak{a}}q_v^{\frac{e_v(u)}{2}}\leq N_F(\mathfrak{J})^{\frac{1}{2}},\ \ \text{for $(u)=\mathfrak{R}\mathfrak{N}_f^{-1}\mathfrak{L}\mathfrak{J}$}.
$$
Therefore, the RHS of \eqref{6.34..} is 
\begin{equation}\label{6.351}
\leq \sum_{\substack{\mathfrak{L}\subseteq \mathcal{O}_F\\
\mathfrak{L}+\mathfrak{b}=\mathcal{O}_F}}\sum_{\substack{\mathfrak{J}\subseteq \mathfrak{b}\\
\mathfrak{J}+\mathfrak{L}=\mathcal{O}_F}}\sum_{\substack{u\in \mathfrak{R}\mathfrak{N}_f^{-1}\mathcal{O}_F\cap  F^{\times}\\
|u|_v\ll |x_v|_v,\ v\mid\infty\\ (u)=\mathfrak{R}\mathfrak{N}_f^{-1}\mathfrak{L}\mathfrak{J}}}N_F(\mathfrak{J})^{\frac{1}{2}}.
\end{equation}

By switching the sums we see that \eqref{6.351} is 
\begin{align*}
\ll_{F}& \sum_{\substack{\mathfrak{L}\subseteq \mathcal{O}_F\\
\mathfrak{L}+\mathfrak{b}=\mathcal{O}_F\\ N_F(\mathfrak{L})\ll \frac{|x_{\infty}|_{\infty}N_F(\mathfrak{N}_f)}{N_F(\mathfrak{R})}}}\frac{|x_{\infty}|_{\infty}N_F(\mathfrak{N}_f)}{N_F(\mathfrak{R})N_F(\mathfrak{L})}\sum_{\substack{\mathfrak{J}\subseteq \mathfrak{b}\\
\mathfrak{J}+\mathfrak{L}=\mathcal{O}_F}}\frac{1}{N_F(\mathfrak{J})^{\frac{1}{2}}}\\
\ll_{F,\varepsilon}& \Bigg[\frac{|x_{\infty}|_{\infty}N_F(\mathfrak{N}_f)}{N_F(\mathfrak{R})}\Bigg]^{1+\varepsilon}\cdot N_F(\mathfrak{b})^{-\frac{1}{2}}\cdot \textbf{1}_{N_F(\mathfrak{R}\mathfrak{N}_f^{-1})\ll |x_{\infty}|_{\infty}},
\end{align*}
where the implied constants depend on $F$ and $\varepsilon.$ Similarly, 
\begin{align*}
\sum_{\substack{u\in \mathfrak{R}\mathfrak{N}_f^{-1}\mathcal{O}_F\cap  F^{\times}\\
|u|_v\ll |x_v|_v,\ v\mid\infty\\
e_v(u-1)\geq 1,\ v\mid\mathfrak{a}}}\prod_{v\mid \mathfrak{a}}q_v^{\frac{e_v(u-1)}{2}}\ll \frac{|x|_{\infty}^{1+\varepsilon}\mathcal{N}_f^{2+\varepsilon}}{N_F(\mathfrak{R})^{1+\varepsilon}}.
\end{align*}

Therefore, we obtain 
\begin{align*}
\mathcal{S}(x_{\infty})\ll \frac{|x|_{\infty}^{1+\varepsilon}\mathcal{N}_f^{2+\varepsilon}}{N_F(\mathfrak{R})^{1+\varepsilon}} \cdot\sum_{\substack{\mathfrak{a}\subsetneq \mathcal{O}_F\\ \mathfrak{a}\mid \mathfrak{Q}}}\Bigg[1+\prod_{v\mid \mathfrak{a}}q_v^{\frac{\min\{r_{\omega_v},m_v\}+n_v-m_v}{2}}\cdot  \prod_{v\mid \mathfrak{a}}q_v^{\frac{m_v-n_v}{2}}\Bigg].
\end{align*}

Note that $\sum_{\substack{\mathfrak{a}\subseteq \mathcal{O}_F\\ \mathfrak{a}\mid \mathfrak{Q}}}1\ll_{\varepsilon,F}Q^{\varepsilon}$ we then derive from the above estimate that 
\begin{equation}\label{6.35}
\mathcal{S}(x_{\infty})\ll_{\varepsilon,F}\prod_{v\mid\mathfrak{Q}}q_v^{\frac{\min\{r_{\omega_v},m_v\}}{2}}\cdot \frac{Q^{\varepsilon}\cdot |x_{\infty}|_{\infty}^{1+\varepsilon}\mathcal{N}_f^{2+\varepsilon}}{N_F(\mathfrak{R})}.
\end{equation}

For $t\in F-\{0,1\}$ with $\frac{t}{t-1}=u\in \mathfrak{R}\mathfrak{N}_f^{-1}\mathcal{O}_F,$ and $|\frac{t}{t-1}|_v\ll |x_v|_v,$ for all $v\mid\infty,$ we have
\begin{align*}
\frac{(MQ\cdot N_F(u)N_F(u-1))^{\varepsilon}\cdot (MQ\mathcal{N}_f^2)^{2s_0}}{N_F(t-1)^{s_0}}\ll (MQ\mathcal{N}_f^2(1+|x_{\infty}|_{\infty}))^{1000\varepsilon}.
\end{align*}

By the $\varepsilon$-convention in \textsection\ref{sec1.5.4} and Lemma \ref{lem6.10} we have
\begin{equation}\label{6.36}
\mathcal{C}(x_{\infty})\ll_{\varepsilon}\frac{(MQ\mathcal{N}_f^2(1+|x_{\infty}|_{\infty}))^{\varepsilon}}{\mathcal{N}_f}\cdot\prod_{v\nmid\mathfrak{Q}\nu(f)}\frac{1}{\Vol(K_v[n_v])}\cdot \mathcal{S}^{\dagger}(x_{\infty}),
\end{equation}
where $\mathcal{S}^{\dagger}(x_{\infty})$ is defined by \eqref{6.34.}. 

By the definition of $\mathfrak{R}\in\mathcal{O}_F$ in \textsection\ref{6.5.2}, it satisfies $e_v(\mathfrak{R})=n_v$ for $v\nmid\mathfrak{Q}\nu(f),$ and   $e_v(\mathfrak{R})=0$ for all $v<\infty$ and $v\mid\mathfrak{Q}\nu(f).$ Hence, 
\begin{equation}\label{6.37}
\frac{1}{N_F(\mathfrak{R})}\prod_{v\nmid\mathfrak{Q}\nu(f)}\frac{1}{\Vol(K_v[n_v])}\ll_{\varepsilon} M^{\varepsilon}.
\end{equation}

Then \eqref{6.34} follows from substituting \eqref{6.34.}, \eqref{6.35}, and \eqref{6.37} into \eqref{6.36}.
\end{proof}

\subsubsection{Proof of Theorem \ref{thmD}} 
Recall the definition \eqref{17.} in \textsection\ref{sec3.7}: 
\begin{align*}
J^{\Reg,\RNum{2}}_{\Geo,\bi}(f,\textbf{s}_0,\chi)=\sum_{t\in F-\{0,1\}}\int_{\mathbb{A}_F^{\times}}\int_{\mathbb{A}_F^{\times}}f\left(\begin{pmatrix}
	y&x^{-1}t\\
	xy&1
\end{pmatrix}\right)|x|^{s_1+s_2}|y|^{s_2}\overline{\chi}(y)d^{\times}yd^{\times}x.
\end{align*}

So the regular orbital integrals $J^{\Reg,\RNum{2}}_{\Geo,\bi}(f,\textbf{s}_0,\chi)$ is 
\begin{align*}
\ll \int_{F_{\infty}^{\times}}\int_{F_{\infty}^{\times}}|x_{\infty}|_{\infty}^{2s_0}|y_{\infty}|_{\infty}^{s_0}\sum_{\substack{t\in F-\{0,1\}\\ \frac{t}{t-1}\in \mathfrak{X}(Q,f)}}\mathcal{E}_{\fin}(t)\Big|f_{\infty}\left(\begin{pmatrix}
	y_{\infty}&x_{\infty}^{-1}t\\
	x_{\infty}y_{\infty}&1
\end{pmatrix}\right)\Big|d^{\times}y_{\infty}d^{\times}x_{\infty},
\end{align*}
where $\mathcal{E}_{\fin}(t):=\prod_{v<\infty}|\mathcal{E}_v(t)|.$

By the support of $f_{\infty}$ (cf. \eqref{6.24} in the proof of Lemma \ref{lem6.6}), 
\begin{equation}\label{6.32}
f_{\infty}\left(\begin{pmatrix}
	y_{\infty}&x_{\infty}^{-1}t\\
	x_{\infty}y_{\infty}&1
\end{pmatrix}\right)=0
\end{equation}
unless $y_{\infty}\asymp 1,$ $|x_{v}|_v\ll 1,$ and $\big|\frac{t}{t-1}\big|_{v}\ll |x_v|_{v},$ for all $v\mid\infty.$ Write 
$$t/(t-1)=u\mathfrak{N}_f^{-1}\mathfrak{R}$$ with $u\in\mathcal{O}_F$ as in \eqref{6.28}. Then 
 $J^{\Reg,\RNum{2}}_{\Geo,\bi}(f,\textbf{s}_0,\chi)$ is 
\begin{align*}
\ll \int_{F_{\infty}^{\times}}\int_{1+o(1)}\textbf{1}_{\substack{|x_v|_v\ll 1\\ v\mid\infty}}|y_{\infty}|_{\infty}^{s_0}\cdot \mathcal{C}(x_{\infty})\cdot \max_{t\in \mathfrak{X}(Q,f)}\Big|f_{\infty}\left(\begin{pmatrix}
	y_{\infty}&x_{\infty}^{-1}t\\
	x_{\infty}y_{\infty}&1
\end{pmatrix}\right)\Big|d^{\times}y_{\infty}d^{\times}x_{\infty},
\end{align*}
where $\mathcal{C}(x_{\infty})$ is defined by \eqref{6.33}. Employing Lemma \ref{lem6.11} we have 
\begin{align*}
J^{\Reg,\RNum{2}}_{\Geo,\bi}(f,\textbf{s}_0,\chi)\ll_{\varepsilon}(MQ\mathcal{N}_f^2)^{\varepsilon}\cdot \mathcal{N}_f\cdot \prod_{v\mid\mathfrak{Q}}q_v^{\frac{\min\{r_{\omega_v},m_v\}+n_v}{2}}\prod_{v\mid\infty}\mathcal{E}_v^{\dagger},
\end{align*}
where $\mathcal{E}_v^{\dagger}$ is defined by \eqref{6.21}. By Lemma \ref{lem6.8}, the above bound becomes
\begin{align*}
J^{\Reg,\RNum{2}}_{\Geo,\bi}(f,\textbf{s}_0,\chi)\ll C_{\infty}(\pi\otimes\chi)^{\varepsilon}(MQ\mathcal{N}_f)^{\varepsilon}\mathcal{N}_f\cdot \prod_{v\mid\mathfrak{Q}}q_v^{\frac{\min\{r_{\omega_v},m_v\}+m_v}{2}},
\end{align*}	
where the implied constant depends on $\varepsilon,$ $F,$ $c_v,$ and $C_v,$ $v\mid\infty$.

\section{Hybrid Subconvexity: Proof of Theorem \ref{B}}\label{sec7} 
Recall the intrinsic data in \textsection\ref{sec2.1}. Let $F$ be a number field. Let $\chi=\otimes_v\chi_v$ be a primitive unitary Hecke character of $F^{\times}\backslash\mathbb{A}_F^{\times}$. Let $\pi=\otimes_{v}\pi_v$ be a pure isobaric representation of $G(\mathbb{A}_F).$ Assume the Hypothesis \ref{hy}, i.e., the local representation $\pi_v\otimes\chi_v$ has \textit{uniform parameter growth of size $(T_v;c_v,C_v)$} for some constants $c_v$ and $C_v,$ and parameters $T_v,$ at all archimedean places $v\mid\infty.$


\subsection{The Spectral Side}\label{sec7.1}
 The spectral side $\mathcal{J}_{\Spec}^{\heartsuit}(\boldsymbol{\alpha},\boldsymbol{\ell})$ has been handled in \textsection\ref{sec3}. Here we recall the lower bound of $\mathcal{J}_{\Spec}^{\heartsuit}(\boldsymbol{\alpha},\boldsymbol{\ell})$ therein. 
 \thmf* 
\subsubsection{Choice of $\boldsymbol{\alpha}$} 
Define the sequence $\boldsymbol{\alpha}=(\alpha_{\mathfrak{m}})_{\mathfrak{m}\in\mathcal{L}}$ (cf. \textsection\ref{sec2.1.5}) by 
\begin{equation}\label{7.1.}
\alpha_{\mathfrak{m}}=\begin{cases}
\overline{\mu_{\pi}(\mathfrak{m})},\ &\text{if $\pi$ satisfies the assumption (a) in Theorem \ref{thm6},}\\
\overline{\mu_{\pi^{\dagger}}(\mathfrak{m})},\ &\text{if $\pi$ satisfies the assumption (b) in Theorem \ref{thm6}.}
\end{cases}
\end{equation}
 
Let $L\gg 1$ be such that $\log L\asymp \log Q.$ Let $\mathcal{L}$ be a subset of the set $\{\mathfrak{m}\in \mathcal{O}_{F}:\  L<N_F(\mathfrak{m})\leq 2L,\ (\mathfrak{m},\mathfrak{D}_FMQ)=1,\ \text{$\mathfrak{m}$ is square-free}\}.$ See \textsection\ref{sec2.1.5}. Define 
\begin{align*}
\mathcal{S}_1:=&\sum_{\mathfrak{m}_1,\mathfrak{m}_2\in\mathcal{L}}|\alpha_{\mathfrak{m}_1}|\cdot |\overline{\alpha_{\mathfrak{m}_2}}|\prod_{v_0\mid\gcd(\mathfrak{m}_1,\mathfrak{m}_2)}\sum_{i_{v_0}=0}^{1}\mathcal{N}_f^{-1},\\
\mathcal{S}_2:=&\sum_{\mathfrak{m}_1,\mathfrak{m}_2\in\mathcal{L}}|\alpha_{\mathfrak{m}_1}|\cdot |\overline{\alpha_{\mathfrak{m}_2}}|\prod_{v_0\mid\gcd(\mathfrak{m}_1,\mathfrak{m}_2)}\sum_{i_{v_0}=0}^{1}\mathcal{N}_f,
\end{align*}
where $\textbf{v}=\mathbf{v}_0\sqcup \mathbf{v}_1\sqcup \mathbf{v}_2$ and $\textbf{i}=\{(i_{v_0})_{v_0\in\textbf{v}_0}:\ i_{v_0}\in\{0,1\}\}$ are defined from $\mathfrak{m}_1,$ $\mathfrak{m}_2$ as follows: $\mathbf{v}_0=\{\text{$\mathfrak{p}$ prime}:\ \mathfrak{p}\mid \gcd(\mathfrak{m}_1,\mathfrak{m}_2)\},$ $\mathbf{v}_j=\{\text{$\mathfrak{p}$ prime}:\ \mathfrak{p}\mid \mathfrak{m}_j/\gcd(\mathfrak{m}_1,\mathfrak{m}_2)\},$ $j=1, 2$ (cf. \textsection\ref{3.2.5} for the notations), and $f=f(\cdot;\mathbf{v},\mathbf{i})$ is defined by \eqref{2.8}, and $\mathcal{N}_f$ is defined in \textsection\ref{2.3}.

\begin{lemma}\label{lem7.1}
Let notation be as before. Then
$\mathcal{S}_1\ll_{\varepsilon}L^{1+\varepsilon}C(\pi)^{\varepsilon}$ and $\mathcal{S}_2\ll_{\varepsilon}L^{3+\varepsilon}C(\pi)^{\varepsilon},$ where the implied constant depend on $\varepsilon.$
\end{lemma}
\begin{proof}
By definition of $\mathcal{N}_f$ in \textsection\ref{2.3}, we have 
\begin{equation}\label{7.2.}
\mathcal{S}_1\ll\sum_{\substack{\mathfrak{b}\in\mathcal{O}_F\\ N_F(\mathfrak{b})\leq 2L\\ \text{$\mathfrak{b}$ is square-free}}}\prod_{\mathfrak{p}\mid \mathfrak{b}}(1+N_F(\mathfrak{p})^{-1})\sum_{\substack{\mathfrak{m}_1,\mathfrak{m}_2\in\mathcal{O}_F\\ \gcd(\mathfrak{b},\mathfrak{m}_1\mathfrak{m}_2)=\mathcal{O}_F\\
\gcd(\mathfrak{m}_1,\mathfrak{m}_2)=\mathcal{O}_F\\ N_F(\mathfrak{m}_1)\leq {2L/N_F(\mathfrak{b})}\\
N_F(\mathfrak{m}_2)\leq {2L}/{N_F(\mathfrak{b})}
}}\frac{|\alpha_{\mathfrak{b}\mathfrak{m}_1}|\cdot |\alpha_{\mathfrak{b}\mathfrak{m}_2}|}{\sqrt{N_F(\mathfrak{m_1}\mathfrak{m_2})}}.
\end{equation}

Note that $\prod_{\mathfrak{p}\mid \mathfrak{b}}(1+N_F(\mathfrak{p})^{-1})\ll \log N_F(\mathfrak{b})\ll \log L,$ and 
\begin{align*}
\sum_{\substack{\mathfrak{m}_1,\mathfrak{m}_2\in\mathcal{O}_F\\ \gcd(\mathfrak{b},\mathfrak{m}_1\mathfrak{m}_2)=\mathcal{O}_F\\
\gcd(\mathfrak{m}_1,\mathfrak{m}_2)=\mathcal{O}_F\\ N_F(\mathfrak{m}_1)\leq {2L/N_F(\mathfrak{b})}\\
N_F(\mathfrak{m}_2)\leq {2L}/{N_F(\mathfrak{b})}
}}\frac{|\alpha_{\mathfrak{b}\mathfrak{m}_1}|\cdot |\alpha_{\mathfrak{b}\mathfrak{m}_2}|}{\sqrt{N_F(\mathfrak{m_1}\mathfrak{m_2})}}\leq |\alpha_{\mathfrak{b}}|^2\Bigg[\sum_{\substack{\mathfrak{m}\in\mathcal{O}_F\\ N_F(\mathfrak{m})\leq {2L/N_F(\mathfrak{b})}
}}\frac{|\alpha_{\mathfrak{m}}|}{\sqrt{N_F(\mathfrak{m})}}\Bigg]^2,
\end{align*}
which, by Cauchy-Schwarz inequality, and the Rankin-Selberg convolution (cf. Lemma 1 in \cite{Iwa92}), and the fact that $\alpha_{\mathfrak{b}\mathfrak{m}_j}=\alpha_{\mathfrak{b}}\alpha_{\mathfrak{m}_j}$, $1\leq j\leq 2,$ is
\begin{align*}
\ll |\alpha_{\mathfrak{b}}|^2\sum_{\substack{\mathfrak{m}\in\mathcal{O}_F\\ N_F(\mathfrak{m})\leq {2L/N_F(\mathfrak{b})}
}}\frac{1}{N_F(\mathfrak{m})}\sum_{\substack{\mathfrak{m}\in\mathcal{O}_F\\ N_F(\mathfrak{m})\leq {2L/N_F(\mathfrak{b})}
}}|\alpha_{\mathfrak{m}}|^2\ll_{\varepsilon} \frac{|\alpha_{\mathfrak{b}}|^2\cdot C(\pi)^{\varepsilon}L\log L}{N_F(\mathfrak{b})}.
\end{align*}
Here we also make use of the fact  that $C(\pi)\asymp C(\pi^{\dagger})^{1+o(1)}.$

Substituting the above estimates into \eqref{7.2.} we then obtain that
\begin{align*}
\mathcal{S}_1\ll_{\varepsilon}\sum_{\substack{\mathfrak{b}\in\mathcal{O}_F\\ N_F(\mathfrak{b})\leq 2L\\ \text{$\mathfrak{b}$ is square-free}}}\frac{|\alpha_{\mathfrak{b}}|^2\cdot C(\pi)^{\varepsilon}L^{\varepsilon}}{N_F(\mathfrak{b})}  \ll_{\varepsilon}L^{1+\varepsilon}C(\pi)^{\varepsilon}.
\end{align*}

Now we proceed to bound $\mathcal{S}_2.$ Similar to \eqref{7.2.}, we have

\begin{align*}
\mathcal{S}_2\ll\sum_{\substack{\mathfrak{b}\in\mathcal{O}_F\\ N_F(\mathfrak{b})\leq 2L\\ \text{$\mathfrak{b}$ is square-free}}}\prod_{\mathfrak{p}\mid \mathfrak{b}}(1+N_F(\mathfrak{p}))\sum_{\substack{\mathfrak{m}_1,\mathfrak{m}_2\in\mathcal{O}_F\\ \gcd(\mathfrak{b},\mathfrak{m}_1\mathfrak{m}_2)=\mathcal{O}_F\\
\gcd(\mathfrak{m}_1,\mathfrak{m}_2)=\mathcal{O}_F\\ N_F(\mathfrak{m}_1)\leq {2L/N_F(\mathfrak{b})}\\
N_F(\mathfrak{m}_2)\leq {2L}/{N_F(\mathfrak{b})}
}}|\alpha_{\mathfrak{b}\mathfrak{m}_1}|\cdot |\alpha_{\mathfrak{b}\mathfrak{m}_2}|\cdot \sqrt{N_F(\mathfrak{m}_1\mathfrak{m}_2)}.
\end{align*}

Notice that $\prod_{\mathfrak{p}\mid \mathfrak{b}}(1+N_F(\mathfrak{p}))\ll N_F(\mathfrak{b})^{1+\varepsilon}\ll L^{\varepsilon}N_F(\mathfrak{b}).$ So  
\begin{align*}
\mathcal{S}_2\ll_{\varepsilon}&L^{\varepsilon}\sum_{\substack{\mathfrak{b}\in\mathcal{O}_F\\ N_F(\mathfrak{b})\leq 2L\\ \text{$\mathfrak{b}$ is square-free}}}N_F(\mathfrak{b})|\alpha_{\mathfrak{b}}|^2\Bigg[\sum_{\substack{\mathfrak{m}\in\mathcal{O}_F\\ N_F(\mathfrak{m})\leq {2L/N_F(\mathfrak{b})}
}}|\alpha_{\mathfrak{m}}|\sqrt{N_F(\mathfrak{m})}\Bigg]^2.
\end{align*}

By Cauchy-Schwarz inequality and the Rankin-Selberg bound we obtain 
\begin{align*}
\mathcal{S}_2\ll_{\varepsilon}L^{\varepsilon}\sum_{\substack{\mathfrak{b}\in\mathcal{O}_F\\ N_F(\mathfrak{b})\leq 2L\\ \text{$\mathfrak{b}$ is square-free}}}\frac{N_F(\mathfrak{b})|\alpha_{\mathfrak{b}}|^2\cdot L^2}{N_F(\mathfrak{b})^2}\cdot \Bigg[\sum_{\substack{\mathfrak{m}\in\mathcal{O}_F\\ N_F(\mathfrak{m})\leq {2L/N_F(\mathfrak{b})}
}}|\alpha_{\mathfrak{m}}|^2\Bigg]\ll_{\varepsilon} L^{3+\varepsilon}C(\pi)^{\varepsilon}.
\end{align*}

Therefore, Lemma \ref{lem7.1} follows.
\end{proof}

\subsubsection{A Lower Bound of $\mathcal{J}_{\Spec}^{\heartsuit}(\boldsymbol{\alpha},\boldsymbol{\ell})$ } 
\begin{prop}\label{prop7.1}
Let notation be as before. Then 
\begin{align*}
\mathcal{J}_{\Spec}^{\heartsuit}(\boldsymbol{\alpha},\boldsymbol{\ell})\gg_{\varepsilon}C_{\infty}(\pi\otimes\chi)^{-\frac{1}{4}-\varepsilon}(MQ)^{-\varepsilon} |L(1/2,\pi\times\chi)|^2\cdot L^{2-\varepsilon}
\end{align*}
if $L\gg_{F,\varepsilon} Q^{\varepsilon}C_{\infty}(\pi)^{1/2+\varepsilon}M^{1+\varepsilon}.$
\end{prop}
\begin{proof}
Let $h$ be a smooth cut-off
function supported on $[1/2,4].
$ Then 
\begin{align*}
\sum_{\mathfrak{m}\in\mathcal{L}}|\mu_{\pi}(\mathfrak{m})|^2\gg_h\frac{1}{2\pi i}\int_{(2)}\frac{L^{(S)}(s,\pi\times\widetilde{\pi})}{L^{(S)}(2s,\pi\times\widetilde{\pi})}\hat{h}(s)L^sds,
\end{align*} 
where $\hat{h}$ is the Mellin transform of $h,$ $S=\{\text{$\mathfrak{p} $ prime}:\ \mathfrak{p}\mid C_{\fin}(\pi)C_{\fin}(\chi)\},$ and the superscript $(S)$ indicates that the Euler factors of the Rankin-Selberg $L$-function at the primes $\mathfrak{p}\in S.$ Shifting the contour from $\Re(s)=2$ to $\Re(s)=1/2+\varepsilon,$ we obtain, from the lower bound $\underset{s=1}{\Res}\ L^{(S)}(s,\pi\times\widetilde{\pi})\gg_{\varepsilon}(C(\pi)Q)^{-\varepsilon}$ (cf. \cite{HL94}) and the upper bound $L^{(S)}(1+2\varepsilon,\pi\times\widetilde{\pi})\ll_{\varepsilon}(C(\pi)Q)^{\varepsilon},$ that 
\begin{equation}\label{7.2}
\sum_{\mathfrak{m}\in\mathcal{L}}|\mu_{\pi}(\mathfrak{m})|^2\gg_{\varepsilon}L(C(\pi)Q)^{-\varepsilon}+O((LM)^{\frac{1}{2}+\varepsilon}C_{\infty}(\pi)^{\frac{1}{4}+\varepsilon}Q^{\varepsilon}).
\end{equation}

Likewise, we also have the estimate towards $\pi^{\dagger}:$
\begin{equation}\label{7.3}
\sum_{\mathfrak{m}\in\mathcal{L}}|\mu_{\pi^{\dagger}}(\mathfrak{m})|^2\gg_{\varepsilon}L(C(\pi)Q)^{-\varepsilon}+O((LM)^{\frac{1}{2}+\varepsilon}C_{\infty}(\pi)^{\frac{1}{4}+\varepsilon}Q^{\varepsilon}).
\end{equation}

Hence, Proposition \ref{prop7.1} follows from  Theorem \ref{thm6} and  \eqref{7.2} and \eqref{7.3}. 
\end{proof}

\subsection{The Geometric Side}\label{sec7.2}
By  \eqref{2.17} and \eqref{15.} in \textsection\ref{sec3.7}, 
\begin{align*}
J_{\Geo}^{\Reg}(f,\textbf{s},\chi):=J^{\Reg}_{\Geo,\sm}(f,\textbf{s},\chi)+J^{\Reg}_{\Geo,\du}(f,\textbf{s},\chi)+J^{\Reg,\RNum{2}}_{\Geo,\bi}(f,\textbf{s},\chi),
\end{align*}
where $\mathbf{s}=(s_1,s_2)$ with $\Re(s_1)+\Re(s_2)>1.$

Recall that $s_0\in [4^{-1}\exp(-3\sqrt{\log C(\pi\times\chi)}),\exp(-3\sqrt{\log C(\pi\times\chi)})]$ is the deformation parameter defined by \eqref{eq2.1} in \textsection\ref{2.1.5.}. 
In \textsection \ref{sec4}-\textsection\ref{sec6} we derive the holomorphic continuation at $\mathbf{s}_0=(s_0,s_0):$
\begin{align*}
J_{\Geo}^{\Reg,\heartsuit}(f,\textbf{s}_0,\chi)=&J^{\Reg}_{\Geo,\sm}(f,\textbf{s}_0,\chi)+J^{\Reg,+}_{\Geo,\du}(f,\textbf{s}_0,\chi)+J^{\Reg,\wedge}_{\Geo,\du}(f,\textbf{s}_0,\chi)\\
&+J^{\Reg,\Res}_{\Geo,\du}(f,\textbf{s}_0,\chi)+J^{\Reg,\RNum{2}}_{\Geo,\bi}(f,\textbf{s}_0,\chi).
\end{align*}

By Proposition \ref{prop12}, we have 
\begin{align*}
J^{\Reg}_{\Geo,\sm}(f,\textbf{s}_0,\chi)\ll_{\varepsilon} \mathcal{N}_f^{-1}[M,M'Q]^{1+\varepsilon}C_{\infty}(\pi\otimes\chi)^{\frac{1}{4}+\varepsilon},
\end{align*}
where $\mathcal{N}_f$ is defined by \eqref{61}, and the implied constant depends only on $F,$ $\varepsilon,$ and $c_v,$ $C_v$ at $v\mid\infty,$ cf. \textsection \ref{2.1.2}. By Proposition \ref{prop17}, Lemmas \ref{lem14.}, \ref{lem5.4}, and \ref{lem5.5}, we have 
\begin{align*}
J^{\Reg,+}_{\Geo,\du}(f,\textbf{s}_0,\chi)+J^{\Reg,\wedge}_{\Geo,\du}(f,\textbf{s}_0,\chi)+J^{\Reg,\Res}_{\Geo,\du}(f,\textbf{s}_0,\chi)
\end{align*}
is majorized by (up to a constant depending on $\varepsilon$ and $F$) 
\begin{align*}
s_0^{-1}\mathcal{N}_f^{-1+2s_0}[M,M'Q]^{1+\varepsilon}C_{\infty}(\pi\otimes\chi)^{\frac{1}{4}+\varepsilon}+C_{\infty}(\pi\otimes\chi)^{\varepsilon}[M,M'Q]^{\varepsilon}\mathcal{N}_f^{1+2s_0+\varepsilon}.
\end{align*}

Moreover, by Theorem \ref{thmD} we have 
\begin{align*}
J^{\Reg,\RNum{2}}_{\Geo,\bi}(f,\textbf{s}_0,\chi)\ll C_{\infty}(\pi\otimes\chi)^{\varepsilon}(MQ\mathcal{N}_f)^{\varepsilon}\mathcal{N}_f\cdot \prod_{v\mid\mathfrak{Q}}q_v^{\frac{\min\{r_{\omega_v},m_v\}+m_v}{2}}.
\end{align*}	

Gathering the above estimates yields an upper bound for $J_{\Geo}^{\Reg,\heartsuit}(f,\textbf{s}_0,\chi).$\begin{prop}\label{prop7.2}
Let notation be as before. Then 
\begin{align*}
J_{\Geo}^{\Reg,\heartsuit}(f,\textbf{s}_0,\chi)\ll &\mathcal{N}_f^{-1+2s_0}[M,M'Q]^{1+\varepsilon}C_{\infty}(\pi\otimes\chi)^{\frac{1}{4}+\varepsilon}\\
&+C_{\infty}(\pi\otimes\chi)^{\varepsilon}(MQ\mathcal{N}_f)^{\varepsilon}\mathcal{N}_f\cdot Q^{\frac{1}{2}}\cdot \gcd(M',Q)^{\frac{1}{2}},
\end{align*}
where the implied constant depends on $\varepsilon,$ $F,$ $c_v,$ and $C_v,$ $v\mid\infty$ (cf. \textsection\ref{2.1.2}).
\end{prop}

\begin{prop}\label{prop7.4}
Let notation be as before. Then 
\begin{align*}
\mathcal{J}_{\Geo}^{\heartsuit}(\boldsymbol{\alpha},\chi)\ll &[M,M'Q]^{1+\varepsilon}C_{\infty}(\pi\otimes\chi)^{\frac{1}{4}+\varepsilon}L^{1+\varepsilon}\\
&+C_{\infty}(\pi\otimes\chi)^{\varepsilon}(MQL)^{\varepsilon}L^3\cdot Q^{\frac{1}{2}}\cdot \gcd(M',Q)^{\frac{1}{2}},
\end{align*}	
where the implied constant depends on $\varepsilon,$ $F,$ $c_v,$ and $C_v,$ $v\mid\infty$ (cf. \textsection\ref{2.1.2}).
\end{prop}
\begin{proof}
By definition in \textsection\ref{3.7.3},
\begin{align*}
\mathcal{J}_{\Geo}^{\heartsuit}(\boldsymbol{\alpha},\chi):=&\sum_{\mathfrak{m}_1,\mathfrak{m}_2\in\mathcal{L}}\alpha_{\mathfrak{m}_1}\overline{\alpha_{\mathfrak{m}_2}}\prod_{v_0\mid\gcd(\mathfrak{m}_1,\mathfrak{m}_2)}\sum_{i_{v_0}=0}^{1}c_{v_0,i_{v_0}}J_{\Geo}^{\Reg,\heartsuit}(f(\cdot;\mathbf{v},\mathbf{i}),\mathbf{s}_0),
\end{align*} 
where $\textbf{v}=\mathbf{v}_0\sqcup \mathbf{v}_1\sqcup \mathbf{v}_2$ and $\textbf{i}=\{(i_{v_0})_{v_0\in\textbf{v}_0}:\ i_{v_0}\in\{0,1\}\}$ are defined from $\mathfrak{m}_1,$ $\mathfrak{m}_2$ as follows: $\mathbf{v}_0=\{\text{$\mathfrak{p}$ prime}:\ \mathfrak{p}\mid \gcd(\mathfrak{m}_1,\mathfrak{m}_2)\},$ $\mathbf{v}_j=\{\text{$\mathfrak{p}$ prime}:\ \mathfrak{p}\mid \mathfrak{m}_j/\gcd(\mathfrak{m}_1,\mathfrak{m}_2)\},$ $j=1, 2.$ See \textsection\ref{3.2.5} for the notations.

Employing the triangle inequality we then derive that
\begin{equation}\label{7.5}
\mathcal{J}_{\Geo}^{\heartsuit}(\boldsymbol{\alpha},\chi)\ll \sum_{\mathfrak{m}_1,\mathfrak{m}_2\in\mathcal{L}}|\alpha_{\mathfrak{m}_1}|\cdot |\overline{\alpha_{\mathfrak{m}_2}}|\prod_{v_0\mid\gcd(\mathfrak{m}_1,\mathfrak{m}_2)}\sum_{i_{v_0}=0}^{1}\big|J_{\Geo}^{\Reg,\heartsuit}(f(\cdot;\mathbf{v},\mathbf{i}),\mathbf{s}_0)\big|.
\end{equation}

Then Proposition \ref{prop7.4} follows from Lemma \ref{lem7.1} and Proposition \ref{prop7.2}, the estimate \eqref{7.5}, and  the fact that $\mathcal{N}_f\leq 2L.$
\end{proof}

\subsection{Put It All Together: Proof of Main Results}\label{sec7.3}
Recall that by Proposition \ref{prop7.1}, we have, under the assumption $L\gg_{F,\varepsilon} (MQ)^{\varepsilon}C_{\infty}(\pi)^{1/2+\varepsilon}M^{1+\varepsilon},$ that 
\begin{align*}
\mathcal{J}_{\Spec}^{\heartsuit}(\boldsymbol{\alpha},\boldsymbol{\ell})\gg_{\varepsilon}C_{\infty}(\pi\otimes\chi)^{-\frac{1}{2}-\varepsilon}(MQ)^{-\varepsilon} |L(1/2,\pi\times\chi)|^2\cdot L^{2-\varepsilon}.
\end{align*}

By Proposition \ref{prop7.4}, we have
\begin{align*}
\mathcal{J}_{\Geo}^{\heartsuit}(\boldsymbol{\alpha},\chi)\ll_{\varepsilon,F} &[M,M'Q]^{1+\varepsilon}C_{\infty}(\pi\otimes\chi)^{\frac{1}{4}+\varepsilon}L^{1+\varepsilon}\\
&+C_{\infty}(\pi\otimes\chi)^{\varepsilon}(MQL)^{\varepsilon}L^3\cdot Q^{\frac{1}{2}}\cdot \gcd(M',Q)^{\frac{1}{2}}.
\end{align*}	

Substituting the above estimates into the amplified relative trace formula (i.e., Theorem \ref{thmD'} in \textsection\ref{3.7.3}) we then obtain the following.
\begin{thmx}\label{thmE}
Let notation be as before. Let $\pi$ be a pure isobaric representation of $\mathrm{GL}(2,\mathbb{A}_F)$ with central character $\omega_{\pi}.$ Let $\chi$ be a Hecke character. Suppose that $\pi_{\infty}\otimes\chi_{\infty}$ has uniform parameter growth (cf. Hypothesis \ref{hy}), and the arithmetic conductor  $C_{\fin}(\chi)>1.$ Let  $L\gg_{F,\varepsilon} (MQ)^{\varepsilon}C_{\infty}(\pi)^{1/2+\varepsilon}M^{1+\varepsilon}.$ Then
\begin{align*}
|L(1/2,\pi\times\chi)|^2\ll_{\varepsilon,F} &[M,M'Q]^{1+\varepsilon}C_{\infty}(\pi\otimes\chi)^{\frac{1}{2}+\varepsilon}L^{-1+\varepsilon}\\
&+C_{\infty}(\pi\otimes\chi)^{\frac{1}{4}+\varepsilon}(MQL)^{\varepsilon}L\cdot Q^{\frac{1}{2}}\cdot \gcd(M',Q)^{\frac{1}{2}},
\end{align*}
where the implied constant depends on $\varepsilon,$ $F,$ $c_v,$ and $C_v,$ $v\mid\infty$ (cf. \textsection\ref{2.1.2}).\end{thmx}

\begin{proof}[Proof of Theorem \ref{A}]
Fix $\pi.$ Then $C_{\infty}(\pi\otimes\chi)\asymp C_{\infty}(\chi)^2,$ and 
$$
(MQ)^{\varepsilon}C_{\infty}(\pi)^{1/2+\varepsilon}M^{1+\varepsilon}\ll_{F,\varepsilon} C(\chi)^{1/4+\varepsilon}.
$$ 
Firstly we consider the situation that $C_{\fin}(\chi)>1.$ Take $L=C(\chi)^{1/4+\varepsilon}.$ Substituting it into Theorem \ref{thmE} to obtain
\begin{align*}
|L(1/2,\pi\times\chi)|^2\ll_{\varepsilon,F} Q^{1+\varepsilon}C_{\infty}(\chi)^{1+\varepsilon}L^{-1+\varepsilon}\ll C(\chi)^{\frac{3}{4}+\varepsilon}.
\end{align*}
where the implied constant depends on $\pi,$ $F$ and $\varepsilon.$ 

If $C_{\fin}(\chi)=1,$ replace $\chi$ by $\chi\chi_0,$ where $\chi_0$ is a fixed Hecke character induced from a Dirichlet character of a fixed modulus, say, $3$, and we also replace $\pi$ with $\pi\otimes\overline{\chi}_0.$ Then Theorem \ref{thmE} applies towards $\pi\otimes\overline{\chi}_0$ and $\chi\chi_0,$ leading to the same bound (but with a different implied constant relying on the modulus of $\chi_0$) for $L(1/2,\pi\times\chi).$ Hence, Theorem \ref{A} follows.
\end{proof}
\begin{proof}[Proof of Theorem \ref{B}]
In Theorem \ref{thmE} we take the parameter $L$ to be
\begin{align*}
\max\Bigg\{(MQ)^{\varepsilon}C_{\infty}(\pi)^{\frac{1}{2}+\varepsilon}M, [M,M'Q]^{\frac{1}{2}+\varepsilon}C_{\infty}(\pi\otimes\chi)^{\frac{1}{8}+\varepsilon}Q^{-\frac{1}{4}}\cdot \gcd(M',Q)^{-\frac{1}{4}}\Bigg\}.
\end{align*}

Consider the situation that $C_{\fin}(\chi)>1.$ We have the following two cases.
\begin{enumerate}
	\item Suppose that 
\begin{align*}
(MQ)^{\varepsilon}C_{\infty}(\pi)^{\frac{1}{2}+\varepsilon}M\ll [M,M'Q]^{\frac{1}{2}+\varepsilon}C_{\infty}(\pi\otimes\chi)^{\frac{1}{8}+\varepsilon}Q^{-\frac{1}{4}}\cdot \gcd(M',Q)^{-\frac{1}{4}}.
\end{align*}

Making use of the fact that $[M,M'Q]\leq MQ,$ we obtain 
\begin{align*}
|L(1/2,\pi\times\chi)|^2\ll_{\varepsilon,F}& M^{\frac{1}{2}+\varepsilon}C_{\infty}(\pi\otimes\chi)^{\frac{3}{8}+\varepsilon}Q^{\frac{3}{4}+\varepsilon}\gcd(M',Q)^{\frac{1}{4}}.
\end{align*}

\item Suppose that 
\begin{align*}
(MQ)^{\varepsilon}C_{\infty}(\pi)^{\frac{1}{2}+\varepsilon}M\gg  [M,M'Q]^{\frac{1}{2}+\varepsilon}C_{\infty}(\pi\otimes\chi)^{\frac{1}{8}+\varepsilon}Q^{-\frac{1}{4}}\cdot \gcd(M',Q)^{-\frac{1}{4}}.
\end{align*}

\begin{align*}
|L(1/2,\pi\times\chi)|^2\ll_{\varepsilon,F} & C_{\infty}(\pi\otimes\chi)^{\frac{1}{4}+\varepsilon}(MQ)^{\varepsilon}C_{\infty}(\pi)^{\frac{1}{2}+\varepsilon}M
Q^{\frac{1}{2}}\cdot \gcd(M',Q)^{\frac{1}{2}}.
\end{align*}
\end{enumerate}
In all, we have
\begin{align*}
|L(1/2,\pi\times\chi)|^2\ll_{\varepsilon,F}& M^{\frac{1}{2}+\varepsilon}C_{\infty}(\pi\otimes\chi)^{\frac{3}{8}+\varepsilon}Q^{\frac{3}{4}+\varepsilon}\gcd(M',Q)^{\frac{1}{4}}\\
&+M^{1+\varepsilon}C_{\infty}(\pi)^{\frac{1}{2}+\varepsilon}C_{\infty}(\pi\otimes\chi)^{\frac{1}{4}+\varepsilon}
Q^{\frac{1}{2}+\varepsilon}\cdot \gcd(M',Q)^{\frac{1}{2}}.
\end{align*}

If $C_{\fin}(\chi)=1,$ replace $\chi$ by $\chi\chi_0,$ where $\chi_0$ is a fixed Hecke character induced from a Dirichlet character of a fixed modulus, and we also replace $\pi$ with $\pi\otimes\overline{\chi}_0,$ as in the proof of Theorem \ref{A}. Therefore, Theorem \ref{B} follows.
\end{proof}

\bibliographystyle{alpha}

\bibliography{LY}

\end{document}